\documentclass{tac}
\usepackage{array}
\usepackage{marginnote}
\usepackage{xcolor}
\usepackage{amscd,amssymb,amsfonts,amsmath,latexsym}
\usepackage{hyperref}
\usepackage[all,cmtip]{xy}
\usepackage{mathrsfs}
\usepackage{graphicx}
\usepackage{bm}  

\usepackage{etoolbox}

\newcommand{\hB}{\hspace*{\fill}}
 \usepackage{slashed}
\usepackage[utf8]{inputenc}
\usepackage{microtype}
\usepackage[english]{babel}

\title{Homotopy theory with $*$-categories}
\author{Ulrich Bunke}
\address{Fakult{\"a}t f{\"u}r Mathematik,
Universit{\"a}t Regensburg,
93040 Regensburg,
GERMANY}

\eaddress{
ulrich.bunke@mathematik.uni-regensburg.de} 

 \thanks{Acknowledgement: The author was supported by the SFB 1085 \emph{Higher Invariants} funded by the Deutsche Forschungsgemeinschaft DFG. He thanks Markus Land    for valuable discussions and  pointing out the reference \cite{DellAmbrogio:2010aa}.  He profited from discussions with Denis-Charles Cisinski about injective model category structures. He furthermore
    thanks Dmitri Pavlov for convincing him that Proposition \ref{efuifzuew9fewfewfwf} should be true, and   Daniel Kasprowski and Christoph Winges for discussing applications of the results of the present paper to coarse homology theories.
    }
 
 \keywords{$*$-categories, model categories, $\infty$-categories, limits and colimits}
\amsclass{18C35,55U35}

\newtheorem{theorem}{Theorem}[section] 
\newtheorem{prop}[theorem]{Proposition}
\newtheorem{lem}[theorem]{Lemma}
\newtheorem{ddd}[theorem]{Definition}
\newtheorem{kor}[theorem]{Corollary}

\theoremstyle{remark}
\theoremstyle{definition}

\newtheorem{conv}[theorem]{Convention}

\newcommand{\pre}{\mathrm{pre}}
\newcommand{\Fun}{{\mathbf{Fun}}}
\newcommand{\cC}{{\mathcal{C}}}
\newcommand{\bD}{{\mathbf{D}}}
\newcommand{\Nerve}{{\mathrm{N}}}
\newcommand{\Map}{{\mathrm{Map}}}
\newcommand{\beins}{{\mathbf{1}}}
\newcommand{\bS}{{\mathbf{S}}}
\newcommand{\bC}{\mathbf{C}}
\newcommand{\End}{\mathrm{End}}
\newcommand{\cV}{{\mathcal{V}}}
\newcommand{\Set}{{\mathbf{Set}}}
\newcommand{\cF}{{\mathcal{F}}}
\newcommand{\id}{{\mathrm{id}}}
\newcommand{\cG}{{\mathcal{G}}}
\newcommand{\nat}{{\mathbb{N}}}

\newcommand{\Q}{{\mathbb{Q}}}
\newcommand{\R}{{\mathbb{R}}}
\newcommand{\C}{{\mathbb{C}}}
\newcommand{\Hom}{{\mathrm{Hom}}}
\newcommand{\pr}{{\mathrm{pr}}}
\renewcommand{\lim}{\operatorname*{\mathrm{lim}}}
\newcommand{\colim}{\operatorname*{\mathrm{colim}}}
\newcommand{\Sp}{\mathbf{Sp}}
\newcommand{\sSet}{{\mathbf{sSet}}}

\newcommand{\ma}{\mathrm{ma}}
\newcommand{\mi}{\mathrm{mi}}
\newcommand{\Mor}{\mathrm{Mor}}
\newcommand{\Orb}{\mathbf{Orb}}
\newcommand{\Ob}{\mathrm{Ob}}
\newcommand{\bB}{{\mathbf{B}}}
\newcommand{\bF}{{\mathbf{F}}}
\newcommand{\bnull}{\mathbf{0}}
\newcommand{\bG}{{\mathbf{G}}}
\newcommand{\cW}{{\mathcal{W}}}
\newcommand{\bA}{{\mathbf{A}}}
\newcommand{\bK}{{\mathbf{K}}}
\newcommand{\const}{{\mathrm{const}}}
\newcommand{\Alg}{{\mathbf{Alg}}}
\newcommand{\cU}{{\mathcal{U}}}
\newcommand{\cD}{{\mathcal{D}}}
\newcommand{\Cat}{{\mathbf{Cat}}}
\newcommand{\Spc}{\mathbf{Spc}}
 
\newcommand{\Ccat}{{ C^{\ast}\mathbf{Cat}}}
\newcommand{\Add}{\mathbf{Add}}

\newcommand{\CEq}{\mathrm{Coeq}}
\newcommand{\Coeq}{\CEq}
\newcommand{\scat}{\mathbf{{}^{*}Cat}}
\newcommand{\preCcat}{C_{\pre}^{\ast}\mathbf{Cat} }
\newcommand{\ClinCat}{ \mathbf{{}^{*}{}_{\C} Cat }}
\newcommand{\compl}{\mathrm{Compl}}
\newcommand{\Lin}{\mathrm{Lin_{\C}}}
\newcommand{\cFun}{\mathcal{F}un}
\newcommand{\Groupoids}{\mathbf{Grpd}}
\newcommand{\Bd}{\mathrm{Bd}}
\newcommand{\bbI}{\mathbb{I}}
\newcommand{\Free}{\mathrm{Free}}

\begin{document}
 
	\maketitle
	\begin{abstract}
We construct model category structures on various types of (marked) $*$-categories.
These structures are used  to present the infinity categories of (marked)
$*$-categories obtained by inverting (marked) unitary equivalences.
We use this presentation    to 	explicitly
 calculate the $\infty$-categorical $G$-fixed points and  G-orbits  for $G$-equivariant (marked) $*$-categories.
 \end{abstract}

\tableofcontents

\section{Introduction}

\subsection{Model categories}

If $\cC$ is a category and $W$ is a set of morphisms in $\cC$, then one can consider the $\infty$-category $ \cC[W^{-1}]$. If the relative category $(\cC,W)$ extends to a simplicial model category
in which all
objects are cofibrant, then we have   an equivalence between $\cC[W^{-1}]$ and the nerve $ \Nerve(\cC^{\mathrm{cf}})$ of the simplicial category of cofibrant/fibrant objects of $\cC$.   This explicit description of $\cC[W^{-1}]$   is sometimes very helpful in order to calculate mapping spaces in $\cC[W^{-1}]$, or in order to identify limits or colimits of diagrams in $\cC[W^{-1}]$.

In this note $\cC$ is a category of $*$-categories\footnote{In order to fix size issues we use three Grothendieck universes $\cU\subseteq \cV\subseteq \cW$.  The objects of $\cC$ will be categories  in $\cV$ which are locally $\cU$-small.   The category $\cC$  itself belongs to $\cW$ and is locally $\cV$-small.}. A $*$-category is a category with an involution $*$ fixing the objects. In such a category one can talk about unitary morphisms. Furthermore, one can talk about unitary transformations between functors between $*$-categories and therefore about unitary equivalences between $*$-categories. One natural choice for $W$  is the set the unitary equivalences.

There are cases where one is interested in $*$-categories with a distinguished subset of the unitary morphisms called marked morphisms. We call such a $*$-category a marked $*$-category. 
We can then consider the category $\cC^{+}$ of such marked $*$-categories with functors preserving the marked morphisms. Moreover we can talk about marked isomorphisms between functors between marked $*$-categories. In this case we let $W$ be the subset of morphisms which are invertible up to marked isomorphism.

In the present paper we consider the following categories $\cC$ of $*$-categories 
and their marked versions $\cC^{+}$.

 \begin{enumerate}
 \item $*$-categories $\scat_{1}$:  categories $\bA$ with an involution $*:\bA\to \bA^{\mathrm{op}}$.
 \item $\C$-linear $*$-categories $\ClinCat_{1}$:   $*$-categories enriched over $\C$-vector spaces with an anti-linear involution.
 \item pre-$C^{*}$-categories $\preCcat_{1}$:    $\C$-linear $*$-categories which admit a maximal $C^{*}$-completion. 
 \item $C^{*}$-categories $\Ccat_{1}$: pre-$C^{*}$-categories whose $\Hom$-vector spaces are complete in the maximal norm.
 \end{enumerate}

 If $\bA$ belongs to one of these examples, then a unitary morphism in $\bA$ is a morphism $u$ whose inverse is given by $u^{*}$. A
  marking on $\bA$ is a choice of a subset of unitary morphisms containing all identities
  which is closed under composition and the $*$-operation. A morphism between marked categories must send marked morphisms to marked morphisms. We write
  $\scat_{1}^{+}$, $\ClinCat_{1}^{+}$, $\preCcat_{1}^{+}$ and $\Ccat_{1}^{+}$ for the categories of marked objects in these examples. The subscript $1$ indicates that we consider them as $1$-categories.

The case of $C^{*}$-categories has been considered previously in \cite{DellAmbrogio:2010aa}. Many arguments in the present paper are modifications of    the arguments given in  \cite{DellAmbrogio:2010aa} in order to be applicable in   the other cases.

 \begin{remark}
We consider $\C$-linear $*$-categories since this case fits with the $C^{*}$-examples.
The assertions about the model category on $\ClinCat_{1}$ and the version of Theorem \ref{fewoijowieffwefwef} extends to the case where $\C$ is replaced by an arbitrary ring with involution. 

An analoguous theory for marked preadditive and additive categories  appears in   \cite{addcats}.

\hB
\end{remark}

 We now state the main result in detail.  Let $\cC$  belong to the list of categories \begin{equation}\label{fwefuihiu2r3t4}
 \{\scat_{1},\ClinCat_{1},\preCcat_{1},\Ccat_{1},\scat_{1}^{+},\ClinCat_{1}^{+},\preCcat_{1}^{+},\Ccat_{1}^{+} \}
\end{equation}

\begin{ddd}\label{fkjhifhiueiwhuiwhfiuewefewfe11} \mbox{}
\begin{enumerate}
\item\label{ewoiwoirwerwr} A weak equivalence in $\cC $ is a (marked) unitary equivalence.
\item A cofibration is a  morphism  in $\cC $ which is injective on objects.
\item A fibration is a morphism  in $\cC $ which has the right-lifting property with respect to trivial cofibrations.
\end{enumerate}
\end{ddd}

In condition \ref{ewoiwoirwerwr} the word \emph{marked} only applies to the four marked versions.
For the simplicial structure we refer to Definition   \ref{riooejrgegerreg}  below since its introduction needs more notation.
For the definition of the notion of a cofibrantly generated model category   we refer to \cite[Def. 2.1.17]{hovey}.  

\begin{theorem}\label{fioweofefewfwf} The structures described in Definition \ref{fkjhifhiueiwhuiwhfiuewefewfe11} and Definition \ref{riooejrgegerreg}  equip $\cC$ with a simplicial model category structure. 

If $\cC$ belongs to the list $$ \{\scat_{1},\ClinCat_{1},\Ccat_{1},\scat_{1}^{+},\ClinCat_{1}^{+},\Ccat_{1}^{+}\}\ ,$$ then the model category structure is cofibrantly generated and the underlying category is locally presentable.
\end{theorem}
 \begin{remark}
In the case of $\cC=\Ccat_{1}$ a proof of this theorem (except the local presentability) was given in \cite{DellAmbrogio:2010aa}. \hB
\end{remark}

\begin{remark}\label{dgiowefwefewfewf}
A  cofibrantly generated simplicial model  category which is in addition   locally presentable is called combinatorial  \cite{Dugger:aa}, \cite[A.2.6.1]{htt}. Hence $\scat_{1}$, $\scat_{1}^{+}$, $\ClinCat_{1}$, $\ClinCat_{1}^{+}$,  $\Ccat_{1} $ and $\Ccat_{1}^{+}$ have combinatorial simplicial model  category structures.

At the moment we do not know whether    $\preCcat_{1}$  or $\preCcat_{1}^{+}$  are cofibrantly generated or  locally presentable.
\hB
\end{remark}

\begin{remark}
The existence and combinatoriality of this model category structure on $\scat_{1}$ has been previously  asserted by Joyal in \cite{jnlab}.\footnote{I thank Philip Hackney for pointing this out.} \hB
\end{remark}

 All categories $\cC$ in the list \eqref{fwefuihiu2r3t4}  
have a notion of  (marked) unitary equivalences. Inverting the (marked) unitary equivalences $W_{\cC }$ in the realm of $(\infty,1)$ (short $\infty$)-categories we obtain  the list  $$   \{\scat,\ClinCat,\preCcat,\Ccat,\scat^{+},\ClinCat^{+},\preCcat^{+},\Ccat^{+} \}$$
 of $\infty$-categories   $\cC_{\infty}:=\cC [W_{\cC }^{-1}]$.

More precisely, we model $\infty$-categories as quasi-categories.  Our basic references are \cite{htt} and \cite{cisin}. We identify categories with $\infty$-categories using the nerve functor. In this case we will omit the nerve from the notation. If $(\cC,W)$ is a relative category, then  there exists a 
 localization functor \footnote{In \cite{MR3460765} this localization is called the Dwyer-Kan localization, since it has been first considered  by \cite{dwka} in the context of simplicial categories, and in order to distinguish it from the localizations considered in the book \cite{htt} wich are versions of Bousfield localizations.} 
 \begin{equation}\label{fwefoiu39r32r32r}  \ell:\cC  \to\cC_{\infty}:=\cC[W^{-1}]\ , \footnote{a more precise notation would be $\Nerve(\cC)\to \Nerve(\cC)[W^{-1}]$}\end{equation} see  \cite[Def. 1.3.4.1]{HA}, \cite[7.1.2]{cisin}.   It is
   characterized  essentially uniquely by the universal property that
$$\ell^{*}:\Fun(\cC_{\infty},\cD)\to \Fun_{W}(\cC ,\cD)$$
is an equivalence for every $\infty$-category $\cD$, where $\Fun_{W}$ denotes the full subcategory of functors sending the morphisms in $W$ to equivalences.

\begin{remark}\label{roigegergerg}
The model category structure on $\cC$ asserted in Theorem \ref{fioweofefewfwf} provides a model for $\cC_{\infty}$. 

In general, let $\cC$ be a simplicial model category with   weak equivalences $W$ and set $\cC_{\infty}:=\cC[W^{-1}]$. We consider the full subcategory $\cC^{\mathrm{cf}}$ of $\cC$ of cofibrant/fibrant objects which is enriched in Kan complexes.  If either all objects of $\cC$ are cofibrant, or $\cC$ admits functorial factorizations (e.g., if $\cC$ is combinatorial), then by  \cite{MR584566}, or \cite[Def. 1.3.4.15, Thm. 1.3.4.20]{HA} (and in addition  \cite[Rem. 1.3.4.16]{HA} in the second case) we have an equivalence $$\cC_{\infty}\simeq \Nerve(\cC^{\mathrm{cf}})\ ,$$ where $\Nerve(\cC^{\mathrm{cf}})$ is the   nerve \cite[Def. 1.1.5.5]{htt} of the fibrant simplicial category $ \cC^{\mathrm{cf}}$. In particular, for $A,B$ in $\cC^{\mathrm{cf}}$ we have an equivalence of spaces
 \begin{equation}\label{f34kij3lk4grgr1111}
\Map_{\cC_{\infty}}(\ell(A),\ell(B))\simeq \ell_{\sSet}(\Map_{\cC}(A,B))\ ,
\end{equation}
where $\Map_{\cC}(A,B)$ is the simplicial mapping set and
 $\ell_{\sSet}:\sSet\to \sSet[W^{-1}]\simeq  \Spc$ is the usual localization of the category of simplicial sets at the weak homotopy equivalences.  In order to see this we could use \cite[1.1.(iv)]{MR584566} in order relate  $\ell_{\sSet}(\Map_{\cC}(A,B))$ with $\ell_{\sSet}(\Map_{L^{H}(\cC,W)}(A,B))$, where $L^{H}$ denotes the hammock localization, and \cite[Prop. 1.2.1]{MR3460765}
 in order to relate $L^{H}(\cC,W)$ with the $\infty$-categorical localization $\cC_{\infty}$.
 
 Note that for the equivalence \eqref{f34kij3lk4grgr1111} it actually suffices to assume that $A$ is cofibrant and $B$ is fibrant.

\hB
    \end{remark}

 By \cite[A.3.7.6]{htt} the $\infty$-category, associated as described in Remark \ref{roigegergerg},  to a  simplicial and combinatorial model category is presentable. Consequently Theorem \ref{fioweofefewfwf} implies:
 
 \begin{kor}\label{wfeoifoweifjwfewfewfewfewf}
The $\infty$-categories
$\scat$, $\scat^{+}$, $\ClinCat$, $\ClinCat^{+}$, $\Ccat$ and $\Ccat^{+}$ are presentable.
\end{kor}

\subsection{Homotopy fixed points and orbits}\label{gtio34fervervevervrv}

 Let $G$ be a group. The category  of  $G$-objects in   a category $\cC$  is defined as the functor category 
$\Fun(BG,\cC)$. Here $BG$ is the category   with one object $pt$ and $\Hom_{BG}(pt,pt)=G$ such that the composition is given by the multiplication in $G$. 

We now assume that the category $\cC$ belongs to the list $$\{\scat_{1},\ClinCat_{1},\preCcat_{1},\Ccat_{1},\scat_{1}^{+},\ClinCat_{1}^{+},\preCcat_{1}^{+},\Ccat_{1}^{+}\}\ .$$ By 
  $\ell:\cC\to \cC_{\infty}$ we denote the localization   \eqref{fwefoiu39r32r32r} which inverts the (marked) unitary equivalences. Furthermore, we let \begin{equation}\label{vr4iuhiuhrfvrui3f3rfvc}
\ell_{BG}:\Fun(BG,\cC)\to \Fun(BG,\cC_{\infty})
\end{equation}  denote the functor given by post-composition with $\ell$.
 We consider an object $\bA$ of $\cC$ with an action of $G$, i.e, an object of $\Fun(BG,\cC)$. 
 One of the purposes of the present paper is to calculate the object $$\lim_{BG}\ell_{BG}(\bA)\ .$$
 Calculation of this limit amounts more precisely to provide an object $\bB$ of $\cC$ and an equivalence $$\lim_{BG}\ell_{BG}(\bA)\simeq \ell(\bB)\ .$$ 
 Such an object $\bB$ will be defined in Definition \ref{ioweoffewfewf} where it is denoted by  $\hat \bA^{G}$.
   The construction of $\hat \bA^{G} $   as such
is not very surprising and reflects the construction of a two-categorical limit.
 In  Theorem \ref{fewoijowieffwefwef} we verify that 
it indeed represents the $\infty$-categorical  limit, i.e., that  we have an equivalence
$$ \lim_{BG}\ell_{BG}(\bA)\simeq  \ell(\hat \bA^{G})\ .$$
In order to approach the   task  of  the calculation of the $\infty$-categorical limit of the $G$-object $\ell_{BG}(\bA)$, using Theorem \ref{fioweofefewfwf}, we present the $\infty$-category   $\Fun(BG,\cC_{\infty})$ in terms of an injective model category structure on $\Fun(BG,\cC)$, see Remark \ref{roigegergerg}. We then observe that $$\lim_{BG} R(\bA) \cong \hat \bA^{G}  \ ,$$ where $R:\Fun(BG,\cC)\to \Fun(BG,\cC)$ is an explicitly given fibrant replacement functor.  
In model categorical language one would say that $\hat \bA^{G}$ represents the homotopy $G$-invariants in $\bA$.
We then use general results from $\infty$-category theory in order to justify that these homotopy invariants indeed represent the limit in the $\infty$-categorical sense.

We now turn to $G$-orbits.
We  assume that $\cC$ belongs to the list
 $$\{\scat_{1},\ClinCat_{1}, \Ccat_{1},\scat_{1}^{+},\ClinCat_{1}^{+},\Ccat_{1}^{+}\}\ .$$
  If $G$ is a group and $\bA$ is an object of $\cC$, then by $\underline{\bA}$ we denote the object of $\Fun(BG,\cC)$ given by $\bA$ with the trivial action of $G$. 
We are interested in the calculation of the colimit
$$\colim_{BG}\ell_{BG}(\underline{\bA})$$ in $\cC_{\infty}$. This again amounts to provide an object $\bB$ of $\cC$ and an equivalence
$$ \colim_{BG}\ell_{BG}(\underline{\bA})\simeq \ell(\bB)\ .$$
In Section \ref{rgiojogerggerg} we construct a bifunctor
$$\cC\times \Groupoids_{1}\to \cC\ , \quad (\bA,\cG)\mapsto \bA \sharp \cG\ .$$
Our main result is  Theorem \ref{weoijoijvu9bewewfewfwef}  which asserts that  $$\colim_{BG} \ell_{BG}(\underline{\bA})\simeq \ell(\bA\sharp BG)\ .$$
The main point here is again that we calculate a colimit in the infinity-categorical sense.
To this end,  using Theorem \ref{fioweofefewfwf}, we present the $\infty$-category $\Fun(BG,\cC_{\infty})$ in terms of a projective model category structure on $\Fun(BG,\cC)$. Then we show that $$ \colim_{BG} L(\underline{\bA})\cong  \bA\sharp BG \ ,$$
where $L$ is an explicit cofibrant replacement functor. In model categorical language one would say that $
 \bA\sharp BG$ represents the homotopy $G$-orbits of $\underline{\bA}$. We then again use general results from $\infty$-category theory in order to justify that these homotopy orbits   represent the colimit in the $\infty$-categorical sense.

The calculation of $G$-orbits will be applied in order to identify 
the values of an induction functor (Definition  \ref{fgwo9gfwegwwfwefewfwf}) 
$$J^{G}:\bC\stackrel{\underline{(-)}}{\to}\Fun(BG,\cC)\stackrel{\ell_{BG}}{\to} \Fun(BG,\cC_{\infty})
\stackrel{{\mathrm{LKan}}}{\to} \Fun(\Orb(G),\cC_{\infty})\ ,$$ where
${\mathrm{LKan}}$ is the left-Kan extension functor associated  to the 
canonical inclusion $BG\to \Orb(G)$.
By Proposition \ref{rguihiufhwrufwfwefwefewf}, for a subgroup $H$ of $G$,
we get an equivalence
$$J^{G}(\bC)(H\backslash G)\simeq \ell(\bC\sharp BH)\ .$$
This result will be applied in order to  identify the coefficients of certain equivariant homology theories.

\section{(Marked) $*$-categories and linear versions}\label{ifjwofjwefwef}
In this section we introduce the notion of a $*$-category and various $\C$-linear versions.
\begin{ddd}
A $*$-category $(\bA,*)$  is a small category $A$  with an  involution $*:\bA\to \bA^{\mathrm{op}}$  which fixes the objects.
A  morphism between $*$-categories is a functor  between the underlying categories which preserves the involutions. 
\end{ddd}
 
We let $\scat_{1}$ denote the category of   $*$-categories and morphisms  between $*$-categories. Usually we will just use the notation $\bA$ instead of $(\bA,*)$.

 If $f$ is a morphism in a $*$-category, then we will write $f^{*}$ for the image of $f$ under the involution $*$.
 
\begin{example}\label{fioefjewoifewfewfwfw}
Let $G$ be a group. Then the category   $BG$ with morphisms $\Hom_{BG}(pt,pt)=G$ can be  turned  into a $*$-category by setting
$$g^{*}:=g^{-1}\ .$$  
More generally, if $\bG$ is any groupoid, then we can consider $\bG$ as  a $*$-category with the $*$-operation  given by  $g^{*}:=g^{-1}$.
\hB
\end{example}

\begin{ddd}
A  $\C$-linear $*$-category is a $*$-category $\bA$  which  is in addition enriched over $\C$-vector spaces  such that  and for   all objects $a,a^{\prime}$ of $\bA$ the map
$$*:\Hom_{\bA}(a,a^{\prime})\to \Hom_{\bA}(a^{\prime},a)$$ is anti-linear. A morphism between $\C$-linear $*$-categories is a morphism between   $*$-categories which is also a functor between $\C$-vector space enriched categories. 
 \end{ddd}

We let $\ClinCat_{1}$ denote  the category of $\C$-linear $*$-categories and morphisms between $\C$-linear $*$-categories.

\begin{remark}
Note that in a $\C$-linear $*$-category $\bA$ for any two objects $a,a^{\prime}$ in $\bA$
the morphism space $\Hom_{\bA}(a,a^{\prime})$ is not empty. But it may be the zero vector space.
This in particular applies to the endomorphisms $\End_{\bA}(a)$. If this is the zero vector space, then we have $\id_{a}=0$. In this case the object $a$  will be called a zero object. The morphism spaces from and to zero objects are zero vector spaces. \hB

\end{remark}

\begin{example}\label{gerigoergregrege}
An algebra $A$ over $\C$ with an anti-linear involution can be  considered as a $*$-category  $\bA$ with one object $pt$ and the $\C$-vector space   of endomorphisms $\Hom_{\bA}(pt,pt):=A$. The composition is given by the algebra multiplication, and the involution on $A$ induces the $*$-functor. 

An important  example is the underlying algebra  with involution of a $C^{*}$-algebra.

We will also encounter the $\C$-linear $*$-category $\Delta^{0}_{\ClinCat_{1}}$ associated to the algebra $\C$. 
The notation indicates that this $\C$-linear $*$-category is the object classifier (see Definition \ref{fiojoiwfefwefewfew}) in $\ClinCat_{1}$.

A particular example is the zero algebra $0$ and the associated $\C$-linear $*$-category $\bnull$.  The unique object in this $\C$-linear $*$-category is a zero object.
\hB\end{example}

In the following we consider the zero algebra as a $C^{*}$-algebra.
Let $\bA$ be  a $\C$-linear $*$-category.  
\begin{ddd}\label{ffjweowefew}A representation of $\bA$ in a $C^{*}$-algebra $B$ is a map
$\rho:\Mor(\bA)\to B$  with the following properties:
\begin{enumerate}
\item For every two objects $a,a^{\prime}$ of $\bA$ the restriction of $\rho$ to the subset   $\Hom_{\bA}(a,a^{\prime})$  of      $\Mor(\bA)$ is $\C$-linear.
\item $\rho(f\circ g)=\rho(f)\rho(g)$ for all pairs of composeable morphisms $f,g$ in $ \Mor(\bA)$
\item $\rho(f^{*})=\rho(f)^{*}$ for all $f$ in  $\Mor(\bA)$
\end{enumerate}
\end{ddd}

\begin{remark}
One could say that a representation  $\rho$ as in Definition \ref{ffjweowefew}
 is a possibly non-unital morphism of $\C$-linear $*$-categories
$ \bA\to  \bB $, where $\bB$ is the (possibly non-unital) $\C$-linear $*$-category associated to the $C^{*}$-algebra  $B$ as in Example \ref{gerigoergregrege}.
Since we do not want to talk about non-unital morphisms we avoid to use this interpretation. 
\hB
\end{remark}

\begin{example}\label{fifjehwofewfewfwf}
We consider the commutative $\C$-algebra $$A:=\C[x](((x-r)^{-1})_{r\in \R})$$ with $*$-operation given by $x^{*}=x$ as a $\C$-linear $*$-category $\bA$.
Assume that $\rho:\bA\to B$ is a  representation in a non-zero $C^{*}$-algebra. Then $\rho(1)$  is a  non-zero selfadjoint  idempotent which commutes with $\rho(a)$ for all elements $a$ of $A$. We can form the non-zero  $C^{*}$-algebra
$B^{\prime}:=\rho(1)B\rho(1)$ with unit $1^{\prime}:=\rho(1)$ and obtain a unital representation $\rho^{\prime}:\bA\to B^{\prime}$ given by $\rho^{\prime}(a):=\rho(1)\rho(a)\rho(1)$. Then 
$\rho^{\prime}(x)$ would be a selfadjoint element in $B^{\prime}$ with empty spectrum. This is impossible. Therefore
$\bA$ does not admit any   representation in a non-zero $C^{*}$-algebra. \hB
\end{example}

Let $\bA$ be  a $\C$-linear $*$-category and $f$ be a morphism in $\bA$.
\begin{ddd}We define the maximal norm of the morphism  $f$ by 
 $$\|f\|_{\max}:=\sup_{\rho} \|\rho(f)\|_{B} \ ,$$
where the supremum  is taken over all representations of $\bA$ in $C^{*}$-algebras $B$  and $\|-\|_{B}$ denotes the norm of $B$. 
 \end{ddd}
 A priory we have
$\|f\|_{\max}\in [-\infty,\infty]$. 
We note the following facts about the maximal norm on a $\C$-linear $*$-category $\bA$:
\begin{enumerate} \item  For every morphism $f$ of $\bA$ we have  $\|f\|_{\max}\ge 0$ since we always have the  representation into the zero algebra. 
 \item For every two composeable morphisms $f$ and $g$  of $\bA$ we have the inequality
$\|g\circ f\|_{\max}\le \|g\|_{\max}\|f\|_{\max}$. 
\item For every morphism $f$ of $\bA$ we have the $C^{*}$-equality $\|f\|_{\max}^{2}=\|f^{*}f\|_{\max}$.   \item \label{fiofiowefwfwefwef}
For every pair $f,g$ of parallel morphisms  in $\bA$ we have the $C^{*}$-inequality 
$\|f\|_{\max}^{2}\le \|f^{*}\circ f+g^{*}\circ g\|_{\max}$. 
  \end{enumerate}
The last three properties hold true since they are satisfied in every representation of $\bA$.

 The maximal norm is preserved by unitary equivalences  between  $\C$-linear $*$-categories, see Lemma \ref{vijroiregrgregreg} below.

\begin{ddd}
A pre-$C^{*}$-category is a $\C$-linear $*$-category in which   every morphism has a finite norm.
\end{ddd}
We let $\preCcat_{1}$ denote the full subcategory of $\ClinCat_{1}$ of pre-$C^{*}$-categories.

\begin{example}
If $\bA$ comes from a $C^{*}$-algebra $A$ with $\|-\|_{A}$ as in Example \ref{gerigoergregrege}, then the representations of $\bA$ in $C^{*}$-algebras $B$    correspond to $C^{*}$-algebra homomorphisms $A\to B$. Since morphisms between $C^{*}$-algebras are norm-bounded by $1$ we have the equality
$$\|-\|_{\max}= \|-\|_{A}\ .$$
Therefore
$\bA$ is a pre-$C^{*}$-category. \hB
\end{example}

 \begin{example}\label{eoihewioewfwfewfw}
Let $A:=\C[x]$ be the polynomial ring  with the $*$-operation given by complex conjugation. We consider $A$ as a $\C$-linear $*$-category $\bA$. Every real number $\mu$ provides a $*$-homomorphism $\rho_{\mu}:
A\to \C$ given by evaluation of the polynomials at $\mu$.  We have $\|\rho_{\mu}(x)\|=|\mu|$. Hence
$\|x\|_{\max}=\infty$. Consequently, the $\C$-linear $*$-category $\bA$ is not a  pre-$C^{*}$-category. 

If $K$ is a compact subset of $   \R$, then we define for $f$ in $\C[x]$
$$\|f\|_{K}:=\sup_{k\in K}\rho_{k}(f)\ . $$ If
we take the closure of $\C[x]$ with respect to this norm, then we get the $C^{*}$-algebra of continuous functions $C(K)$ on $K$.  This shows that a $\C$-linear $*$-category  which is not a pre-$C^{*}$-category
still may have many non-trivial $C^{*}$-closures. 
\hB
\end{example}

\begin{example}
We consider the $\C$-linear $*$-category $\bA$ from Example \ref{fifjehwofewfewfwf}. Since it has no non-trivial $*$-representations the maximal norm on it is trivial. Consequently it is a pre-$C^{*}$-category. \hB
 \end{example}

\begin{ddd}
A $C^{*}$-category is a pre-$C^{*}$-category   in which for every pair of objects  
the morphism vector space  is a Banach space with respect to the norm $\|-\|_{\max}$.
\end{ddd}

We let $\Ccat_{1}$ denote the full subcategory of $\preCcat_{1}$ of $C^{*}$-categories.

\begin{remark}\label{wreoiwoifwefewf}
According to the  usual definition   a $C^{*}$-category $\bC$  is a $\C$-linear $*$-category in which the morphism spaces are equipped with norms such that:
\begin{enumerate}
\item The morphism spaces are complete.
\item For composable morphisms we have $\|f\circ g\|\le \|g\|\|f\|$.
\item 
 The $C^{*}$-identity $\|f^{*}\circ f\|=\|f\|^{2}$ holds true for every morphism $f$. 
 \item For every pair $f,g$ of parallel morphisms the $C^{*}$-inequality 
 $\|f\|^{2}\le \|f^{*}f+g^{*}g\|$ holds true.
\end{enumerate}
We claim that the   maximal norm on such a category considered as a $\C$-linear $*$-category  coincides with the given norm. 
Since every  representation of   $C^{*}$-categories (in the usual definition)  is norm-decreasing
we conclude that  the maximal norm on $\bC$ is smaller than the given norm.

In order to show equality we form the $C^{*}$-algebra $$A(\bC):=\bigoplus_{c,c^{\prime}\in \bC} \Hom_{\bC}(c,c')$$
 (the sum is taken in the category of Banach spaces and involves completion) with the composition given by matrix multiplication and the obvious $*$-operation.
 Let $\rho:\bC\to A(\bC)$  denote the canonical representation. Then $\|f\|=\|\rho(f)\|_{A(\bC)}$.

This implies that our definition of a $C^{*}$-category coincides with the classical one.
Furthermore, being a $C^{*}$-category is just a property of a $\C$-linear $*$-category and not an additional structure.
\hB

\end{remark}
We have a chain of functors   \begin{equation}\label{whchiwehciuwchuiwecwecwc}
\Ccat_{1} \subseteq \preCcat_{1}\subseteq \ClinCat_{1} \to \scat_{1}\to \Cat_{1}\ ,
\end{equation}
where the first two are fully faithful.

\begin{remark}\label{eiuuifwefewfewfw}
The categories  $\Ccat_{1}$,  $\ClinCat_{1}$,    $ \scat_{1}$ and $\Cat_{1}$ are closed under taking full subcategories.
For $\Ccat_{1}$ this follows from Remark \ref{wreoiwoifwefewf},   and for the other three examples this is clear.

For $\preCcat_{1}$,  going to a full subcategory may increase  the maximal norm since there might be functors from the subcategory to $C^{*}$-algebras     which do not extend to the whole category. We can not exclude that it becomes infinite. \hB
\end{remark}

Let $\cC$ be a member of the list \begin{equation}\label{wvwevewvwvv} \{\scat_{1},\Ccat_{1} ,  \ClinCat_{1} ,          \Ccat_{1} \}
\end{equation} 
and  $\bA$ be an object of $\cC$.

Let $a,a^{\prime}$   be objects  of $\bA$,  and $u$ be a morphism in $ \Hom_{\bA}(a,a^{\prime})$.
\begin{ddd}\label{fowpfewfefw}
The morphism $u$ is called unitary if
$u^{*}u=\id_{a}$ and $uu^{*}=\id_{a^{\prime}}$.

\end{ddd}
 Note that the zero morphism between two zero objects is  unitary.

\begin{ddd} 
A marking on $\bA$ is a choice of a subset of the unitary morphisms containing all identities  which is preserved by the involution $*$ and closed under  composition.
\end{ddd}

We can talk about marked objects in $\cC$.  A marked object in $\cC$  has an underlying object in $\cC$ obtained by forgetting the marking. 

\begin{ddd}A morphism between two marked objects in $\cC$ is a morphism between the underlying objects in $\cC$  which sends marked morphisms to marked morphisms.
\end{ddd}

In this way we obtain categories
$\Ccat_{1}^{+}$,  $\ClinCat_{1}^{+}$,    $ \scat_{1}^{+}$,  or $\Cat_{1}^{+}$
of marked objects and morphisms between marked objects.

\begin{remark}\label{wfiuwehifewfwefwe} Let $\cC$ be in the list $\{\Ccat_{1}^{+}$,  $\ClinCat_{1}^{+}$,    $ \scat_{1}^{+}$,  $\Cat_{1}^{+}\}$ and $\bA$ be an object of $\cC$.
  Then we can consider the subcategory
 $\bA^{+}$ of $\bA$ with the same objects as $\bA$ and marked morphisms.  Note that $\bA^{+}$ is a groupoid.  A morphism $f:\bA\to \bB$ in $\cC$ induces a morphism of    categories
 $f^{+}:\bA^{+}\to \bB^{+}$.
   \end{remark}

\begin{example}\label{34toihoi34t34t34t34tf} Let $\cC$ be in the list  $\{\scat_{1},\Ccat_{1} ,  \ClinCat_{1} ,          \Ccat_{1} \}
$ and 
 $\bA$ be an object of $\cC$. Then we can consider the marked
category $\mi(\bA)$ in $\cC^{+}$ whose marked morphisms are exactly the identities.
We can also consider the marked category $\ma(\bA)$ in $\cC^{+}$ whose marked morphisms are all unitary morphisms.
\hB
\end{example}

 \section{Adjunctions}
 
 The inclusions in the chain \eqref{whchiwehciuwchuiwecwecwc} are right or left adjoints of adjunctions  which we will now describe. The presence of these adjunctions turns out to be useful at various places.
 They will be lifted to infinity-categorical versions in Section \ref{ufhewiuhewfewf}.
 
 As a convention we will denote forgetful functors by the symbol $\cF$ with a subscript    indicating which structure or property is forgotten. Most of the functors below come in two versions, one for the unmarked, and one for the marked case.  We will use the same notation for both.   
 
 Given a (marked) category $\bA$ we can form the free (marked) $*$-category $\Free_{*}(\bA) $ on $\bA$.
  We have   adjunctions \begin{equation}\label{ferferfeoijiorojr34r34r34r}
\Free_{*}:\Cat_{1}\leftrightarrows \scat_{1}:\cF_{*}\ , \quad \Free_{*}:\Cat_{1}^{+}\leftrightarrows \scat_{1}^{+}:\cF_{*}\ ,
\end{equation}
 where $\cF_{*}$ denotes the  functor which forgets the $*$-operation.

 \begin{remark}
If $\bA$ is a category, then $\Free_{*}(\bA)$ is obtained from $\bA$ by  adding a morphism $f^{*}:a^{\prime}\to a$ for every morphism $ f:a\to a^{\prime}$ in $\bA$ with the only relation that $(f\circ g)^{*}=g^{*}\circ f^{*}$. 

A marked category (i.e., an object of $\Cat_{1}^{+}$) is a category with a distinguished set of isomorphisms.  
For a   marked category $\bA$, in the definition of $\Free_{*}(\bA)$,  we adopt the additional  relation
$f^{*}=f^{-1}$ for all marked   marked morphisms $f$ of $\bA$ (which are isomorphisms by definition).
 \hB
 \end{remark}

We furthermore have adjunctions \begin{equation}\label{fwefweiufhui23r}
\Lin: \scat_{1}\leftrightarrows  \ClinCat_{1}:\cF_{\C}\ , \quad \Lin: \scat^{+}_{1}\leftrightarrows  \ClinCat_{1}^{+}:\cF_{\C}\ ,
\end{equation}
where $\Lin$ is the linearization functor, and $\cF_{\C}$ denotes the functor which forgets the $\C$-linear structure.

 \begin{remark} Here is the explicit description of $\Lin$. If $\bA$ is a $*$-category, then
$\Lin(\bA)$ has the same objects as $\bA$, but its $\C$-vector space of morphisms is given by
$$\Hom_{\Lin(\bA)}(a,a^{\prime}):=\C[\Hom_{\bA}(a,a^{\prime})]\ ,$$
and the composition is defined in the canonical way.
 
The  functor  $*:\Lin(\bA)\to \Lin(\bA)^{\mathrm{op}}$ is defined as the anti-linear extension of $*$ on $\bA$.

For a set $X$ and an element $x$ of $X$ we consider $x$ as an element of the complex vector space $\C[X]$ generated by $X$ in the canonical way. This gives a canonical map  of sets $X\to \C[X]$.

In the marked case 
the set of  marked morphisms in $\Hom_{\Lin(\bA)}(a,a^{\prime})$ is defined to be the image of the set of marked morphisms in $\bA$ under the  
canonical map 
$$\Hom_{\bA}(a,a^{\prime})\to  \C[\Hom_{\bA}(a,a^{\prime})]=\Hom_{\Lin(\bA)}(a,a^{\prime})\ .$$

 For a $\C$-linear $*$-category $\bB$ we check  the natural bijection \begin{equation}\label{vervevlkjl3jlkferfeferfe}
\Hom_{\ClinCat_{1}}(\Lin(\bA),\bB)\cong \Hom_{\scat_{1}}(\bA,\cF_{\C}(\bB))\ .
\end{equation} 
It identifies a morphism $\Phi:\bA\to \cF_{\C}(\bB)$ with a morphism $\Psi:\Lin(\bA)\to \bB$.
The functors coincide on objects.
Given $\Phi$ and a morphism $f$ in $\Hom_{\Lin(\bA)}(a,a^{\prime})$ we define $$\Psi(f):=\sum_{\phi\in \Hom_{\bA}(a,a^{\prime})}\lambda_{\phi}\Phi(\phi)$$ in $\Hom_{\bB}(\Phi(a),\Phi(a^{\prime}))$, where the equality  $f=\sum_{\phi\in \Hom_{\bA}(a,a^{\prime})}\lambda_{\phi}\phi$ uniquely determines the collection of complex numbers $(\lambda_{\phi})_{\phi\in \Hom_{\bA}(a,a^{\prime}) }$.
Vice versa, if
$\Psi$ is given, then for $f$ in $\Hom_{\bA}(a,a^{\prime})$ we define
$\Phi(f):=\Psi(f)$ in $\Hom_{\bB}(\Psi(a),\Psi(a^{\prime}))$.

In the marked case, by an inspection of these explicit formulas, one checks that the bijection \eqref{vervevlkjl3jlkferfeferfe} restricts to a bijection between the sets of 
 morphisms between marked objects.
\hB
\end{remark}

\begin{remark}
If $\bB$ contains zero objects, then
$\Hom_{\ClinCat_{1}}(\bB, \Lin(\bA))=\emptyset$ for every $*$-category $\bA$. \hB
\end{remark}

  We  have  adjunctions \begin{equation}\label{ewfwoijioiffewfwefw}
 \compl:\preCcat_{1} \leftrightarrows \Ccat_{1}:\cF_{-}\ , \quad  \compl:\preCcat_{1}^{+} \leftrightarrows \Ccat_{1}^{+}:\cF_{-}
\end{equation}
where $\cF_{-}$ forgets the completeness condition and  $\compl$ is the  completion functor.

\begin{remark}
In the following we give an explicit description of the completion functor.

Let $\bA$ be a pre-$C^{*}$-category.
 The completion functor is the identity on objects. Furthermore, it completes 
 the morphism spaces in the norm $\|-\|_{\max}$.  Note that this completion involves factoring out the subspace of elements of zero norm.
In this completion process the objects $a$ of $\bA$ with $\|\id_{a}\|=0$ become zero objects.

In the marked case the set of  marked morphisms in $\Hom_{\compl(\bA)}(a,a^{\prime})$  
is defined to be  the image of the set of marked morphisms in $\Hom_{ \bA }(a,a^{\prime})$
under the natural map
$
\Hom_{ \bA }(a,a^{\prime})\to \Hom_{ \compl(\bA)}(a,a^{\prime})$.
This is well-defined since this map  preserves unitaries.

Let us check   the bijection  \begin{equation}\label{ed23d2d23d32d32d3}
\Hom_{\Ccat_{1}}(\compl(\bA),\bB)\cong \Hom_{\preCcat_{1}}(\bA,\cF_{-}(\bB))
\end{equation} 
for a $C^{*}$-category $\bB$. 
This bijection sends a  morphism  $\Phi:\bA\to \cF_{-}(\bB)$  to  a morphism $\Psi :\compl(\bA)\to \bB$ and vice versa.
These functors coincide on objects.
Given $\Phi$ we can define $\Psi$ by extension by continuity (using that the morphism spaces in $\bB$ are complete).   Note that $\Phi$ necessarily annihilates all morphisms of zero norm and therefore factorizes over the quotients taken in the process of completion.
 
Vice versa, given $\Psi$ we define $\Phi$ by restriction along the natural maps
$$\Hom_{\bA}(a,a^{\prime})\to \Hom_{\compl(\bA)}(a,a^{\prime})$$ for all
pairs of objects $a,a^{\prime}$ of $\bA$.

One easily checks that these processes are inverse to each other.

In the marked case we observe by an inspection that the bijection \eqref{ed23d2d23d32d32d3} identifies morphisms between marked objects.
\hB
\end{remark}

\begin{example}
We continue with Example \ref{fioefjewoifewfewfwfw}. The linearization $$\Lin(BG) \cong \C[G]$$ of $G$ with its usual involution is  a $\C$-linear $*$-category. It is actually a pre-$C^{*}$-category.
In order to see this note that the elements of the group go to partial isometries in every representation.
Furthermore, $$\End_{\compl(\Lin(BG))}(pt)=:C^{*}_{\max}(G)$$ is the maximal group-$C^{*}$-algebra. \hB
\end{example}

\begin{example}
We consider the $\C$-linear $*$-category $\bA$ from Example \ref{fifjehwofewfewfwf}.  We get
$\compl(\bA)\cong \bnull$.  \hB
 \end{example}

The relation between $\C$-linear $*$-categories and pre-$C^{*}$-categories is more complicated.
Let $\bA$ be a $\C$-linear $*$-category. Then we can define a subcategory
$\Bd(\bA)$ of $  \bA$ which has the same objects as $\bA$, and whose morphisms are those morphisms of $\bA$ with finite maximal norm. Note that $\Bd(\bA)$ is closed under composition and contains the identities of the objects since they are sent to selfadjoint projections in any representation of $\bA$ in a $C^{*}$-algebra and therefore have maximal norm bounded  above by $1$. Hence $\Bd(\bA)$ is indeed a category. It is furthermore clear that $\Bd(\bA)$ is a $\C$-linear $*$-category with enrichment and $*$-operation induced from $\bA$.

Since  $\Bd(\bA)$ may have representations to $C^{*}$-algebras  which do not extend to $\bA$ we can not  expect that $\Bd(\bA)$ is a pre-$C^{*}$-category. But we can iterate the construction and consider
$$\Bd^{\infty}(\bA):=\bigcap_{n\in \nat} \Bd^{n}(\bA)\ .$$
This is a $\C$-linear $*$-subcategory of $\bA$.

Assume now that $\bA$ is a marked  $\C$-linear $*$-category. Then all marked morphisms of $\bA$ also belong to $\Bd(\bA)$ since unitary morphisms in $\bA$ have norm bounded by one. So $\Bd(\bA)$ and hence $\Bd^{\infty}(\bA)$ has a canonical marking consisting of all marked morphisms in $\bA$.

\begin{example} We continue with Example
\ref{eoihewioewfwfewfw}. The bounded elements in $\C[x]$ are the constant polynomials. Hence $\Bd(\C[x])\cong \Delta^{0}_{\ClinCat_{1}}$.\hB
\end{example}

Let $\bA$ be a (marked)  $\C$-linear $*$-category.
\begin{lem}\label{uhfiuhfiwefewfewfewfewf}\mbox{}
\begin{enumerate}
\item \label{fuhweuifewfwfw} $\Bd^{\infty}(\bA)$ is a (marked) pre-$C^{*}$-category.
\item \label{frefrkjhgiekgergerg} Any functor $\bC\to \bA$ of (marked) $\C$-linear $*$-categories where $\bC$ is a (marked) pre-$C^{*}$-category factorizes uniquely over
$\Bd^{\infty}(\bA)$.
\item\label{frefrkjhgiekgergerg1} There are  adjunctions \begin{equation}\label{fiioioefjeofiewfewfwf}
\cF_{\pre}:\preCcat_{1}\leftrightarrows \ClinCat_{1}:\Bd^{\infty}\ , \quad  \cF_{\pre}:\preCcat_{1}^{+}\leftrightarrows \ClinCat_{1}^{+}:\Bd^{\infty}\ ,
\end{equation}
where $\cF_{\pre}$ denotes the inclusion.
\end{enumerate}
\end{lem}
 \begin{proof}
 We can identify $\Bd^{\infty}(\bA)\cong \Bd(\Bd^{\infty}(\bA))$. 
 Let   $f$ be a morphism in $\Bd^{\infty}(\bA)$. Then we can consider $f$ as a morphism in  $ \Bd(\Bd^{\infty}(\bA))$ which implies that $f$ has a finite
 maximal norm. This proves \ref{fuhweuifewfwfw}.
 
Let $\phi:\bC\to \bA$ be a morphism of $\C$-linear $*$-categories where $\bC$ is a pre-$C^{*}$-category. We show \ref{frefrkjhgiekgergerg} by contradiction. Assume that there exists a natural number $n $ such that
$\phi(\bC)\subseteq \Bd^{n}(\bA)$, but $\phi(\bC)\not\subseteq \Bd^{n+1}(\bA)$.
Then there exists a morphism $f$ in  $\bC$ and  a family of   representations $(\rho_{k})_{k\in \nat}$ of $\Bd^{n}(\bA)$ in a family of  $C^{*}$-algebras $(B_{k})_{k\in \nat}$ such that $\|\rho_{k}(\phi(f))\|_{B_{k}}\ge k$ for every natural number $k$. Since  the composition $\rho_{k}\circ \phi$ is a    representation  of $\bC$ we see that  $k\le \|f\|_{\max}$ for  every natural number $k$.   This contradicts the assumption that $\bC$ is a pre-$C^{*}$-category.

In the marked case we observe that $\phi$ sends marked morphisms in $\bC$ to $\Bd^{\infty}(\bA)$ since marked morphisms are unitary and $\phi$ preserves unitaries.

The Assertion \ref{frefrkjhgiekgergerg1} now follows from Assertion \ref{frefrkjhgiekgergerg}. \end{proof}

Let $\cC$ be a member of the list $$\{\Cat_{1},\scat_{1},\ClinCat_{1},\preCcat_{1},\Ccat_{1}\}$$  and $\cC^{+}$ denote the corresponding marked version.
We have a canonical functor $\cF_{+}:\cC^{+}\to \cC$ which forgets the marking. This functor fits into adjunctions
\begin{equation}\label{wecoihorgrtb}
\mi :\cC\leftrightarrows \cC^{+}:\cF_{+}\ , \quad \cF_{+}:\cC^{+}\leftrightarrows \cC :\ma 
\end{equation}

The left-adjoint $\mi$ of $\cF_{+}$ marks the identities, and the right-adjoint $\ma$ marks all unitaries (or invertibles in the case of $\Cat_{1}$, respectively).

\section{Classifier categories}\label{fiowufoefewfewfewfwf}

In this section we discuss the representability of the functors which take  the sets
objects, (bounded) morphisms, unitary morphisms (or marked morphisms) of a  (marked) $*$-category in the respective cases. The role of this section is   an illustration. We use the opportunity to explain the categorical meaning of the examples which will be used later.

  Let $\cC$ be in the list
 $$\{\Cat_{1},\scat_{1},\ClinCat_{1},\preCcat_{1},\Ccat_{1},\Cat_{1}^{+},\scat_{1}^{+},\ClinCat_{1}^{+},\preCcat_{1}^{+},\Ccat_{1}^{+}\}\ .$$
 \begin{lem} 
The functor $\cC\to \Set$ which sends a category   in $\cC$ to its set of   objects is representable.
 \end{lem}
\begin{ddd}\label{fiojoiwfefwefewfew}The object $\Delta_{\cC}^{0}$ which represents this functor will be called the object classifier.\end{ddd}
 \begin{proof}
 This is a case-by-case discussion.
In $\Cat_{1}$ we can set
$$\Delta_{\Cat_{1}}^{0}:=pt\ ,$$
the category with one object  $pt$ and one morphism $\id_{pt}$.
Then in view of the adjunction \eqref{ferferfeoijiorojr34r34r34r} we have   $$\Delta_{\scat_{1}}^{0}\cong \Free_{*}(\Delta_{\Cat_{1}}^{0})\ .$$
Its underlying category is again $\Delta^{0}_{\Cat_{1}}$.

In view of the adjunction  \eqref{fwefweiufhui23r} we have
$$\Delta^{0}_{\ClinCat_{1}}\cong \Lin(\Delta_{\scat_{1}}^{0})\ .$$
This is the $\C$-linear $*$-category associated to the $C^{*}$-algebra $\C$ and hence is a $C^{*}$-category. We conclude that 
$$\Delta^{0}_{\Ccat_{1}}\cong \Delta^{0}_{\preCcat_{1}}\cong \Delta^{0}_{\ClinCat_{1}}$$
(as $\C$-linear $*$-categories).
For $\cC$ in the list $$\{\Cat_{1},\scat_{1},\ClinCat_{1},\preCcat_{1},\Ccat_{1}\}$$
  the marked version of the object classifier is characterized  by
$$\Delta^{0}_{\cC^{+}}\cong \mi (\Delta^{0}_{\cC})\ ,$$ see \eqref{wecoihorgrtb} for notation.
\end{proof}

Usually we will omit the subscript $\cC$ when the context is clear and just write $\Delta^{0}$ for the object classifier. A similar conventions applies to the other classifier objects below.

\begin{lem}\label{gigjoer2tergtertet} For $\cC$ in $\{\Cat_{1},\scat_{1},\ClinCat_{1}, \Cat_{1}^{+},\scat_{1}^{+},\ClinCat_{1}^{+}\}$
the functor $\cC\to \Set$ which sends  a category  in $\cC$  to its set of morphisms   is representable.

For $\cC$ in $\{\preCcat_{1}, \Ccat_{1},\preCcat_{1}^{+}, \Ccat_{1}^{+}\}$ this functor is not representable.
\end{lem}
\begin{ddd}\label{gergjeiogjogergergregergergergee}The object $\Delta_{\cC}^{1}$ which represents this functor will be called the morphism classifier.\end{ddd}
 \begin{proof}
 This is again   a case-by-case discussion.
 The morphism classifiers   have two objects $0$ and $1$ corresponding to the source and target of the morphism.
 
 In $\Cat_{1}$ we let $\Delta_{\Cat_{1}}^{1}$ be the category with one non-trivial morphism $a:0\to 1$.
 Then in view of the adjunction \eqref{ferferfeoijiorojr34r34r34r} we have   $$\Delta_{\scat_{1}}^{1}\cong \Free_{*}(\Delta_{\Cat_{1}}^{1})\ .$$
 For example, a morphism $0\to 1$ in this category is a word $aa^{*}aa^{*}\dots a^{*}a$. 
 In view of the adjunction  \eqref{fwefweiufhui23r} we have
$$\Delta^{1}_{\ClinCat_{1}}\cong \Lin(\Delta_{\scat_{1}}^{1})\ .$$

In the marked cases, for $\cC$ in the list $ \{\Cat_{1},\scat_{1},\ClinCat_{1}\}$, 
 the morphism classifiers are characterized by  $$\Delta^{1}_{\cC^{+}}\cong \mi_{\cC}(\Delta^{1}_{\cC})\ .$$

We now come to the non-existence assertion.
Assume that the pre-$C^{*}$-category $\bC$ represents the morphism-set functor in $\preCcat_{1}$.  
Let $a:0\to 1$ be the universal morphism. As in any non-trivial  pre-$C^{*}$-category 
there exists morphisms of arbitrary large maximal norm (just scale) we have $\|a\|_{\max}=\infty$ contradicting the assumption that $\bC$ is a  pre-$C^{*}$-category. The same reasoning applies to $\Ccat_{1}$ and the marked versions.
 \end{proof}

The following replaces the morphisms classifier in  the $C^{*}$-category cases. Let $\cC$ be a member of the list $\{\Ccat_{1},\Ccat^{+}_{1}\}$. The following result is \cite[Ex. 3.8]{DellAmbrogio:2010aa}.
\begin{lem}

The functor $\cC \to \Set$ which sends a category in $\cC$ to its set of morphisms with maximal norm bounded by $1$ is representable.
\end{lem} 
\begin{ddd}\label{rgieroge34t34t34t34t4}The object $\Delta_{\cC}^{1,\mathrm{\mathrm{bd}}}$ which represents this functor will be called the bounded morphism classifier.\end{ddd}
\begin{proof} In order to construct the bounded morphism classifier for $\Ccat_{1}$ we
start with the $\C$-linear $*$-category $$\Delta^{1}_{\ClinCat_{1}}\cong \Lin(\Free_{*}(\Delta^{1}_{\Cat_{1}}))\ .$$
The universal morphism $a$ in $\Delta^{1}_{\Cat_{1}}$ can be considered as a morphism of $\Delta^{1}_{\ClinCat_{1}}$ in the natural way.
 We add formal inverses of
$\lambda \id_{0}-a^{*}a$ in $\End_{\Delta^{1}_{\ClinCat_{1}}}(0)$ and $\lambda \id_{1}-aa^{*} $ in $\End_{\Delta^{1}_{\ClinCat_{1}}}(1)$  for all $\lambda$ in $\C$ with $|\lambda|>1$ and obtain  a new $\C$-linear $*$-category $\widetilde{\Delta^{1}_{\ClinCat_{1}}}$.   Spectral theory  implies that $\|a\|_{\max}=\|a^{*}\|_{\max}\le 1$. Hence 
$\widetilde{\Delta^{1}_{\ClinCat_{1}}}$ is a pre-$C^{*}$-category.
One easily checks that
$$ 
\Delta_{\Ccat_{1}}^{1,\mathrm{bd}}:=\compl(\widetilde{\Delta^{1}_{\ClinCat_{1}}})$$
has the required universal properties. 
In the marked case we have 
$$ 
\Delta_{\Ccat_{1}^{+}}^{1,\mathrm{\mathrm{bd}}}\cong \mi (\Delta_{\Ccat_{1}^{+}}^{1,\mathrm{\mathrm{bd}}})\ .$$ \end{proof}

\begin{remark}\label{geieojeoijreoreirferferf}
We do not know whether $\preCcat_{1}^{(+)}$ has a bounded morphism classifier or an appropriate replacement. This is the reason that we can not show that the model category structure on   $\preCcat_{1}^{(+)}$ is cofibrantly generated.
\hB
\end{remark}

We now discuss  unitaries.

\begin{remark} \label{gregoi34jf34f34f} If $u$ is a  unitary morphism in a $\C$-linear $*$-category,  then   any representation sends $u$ to a partial isometry. Hence $\|u\|_{\max}\le 1$. 
But it may happen that $\|u\|=0$. This is e.g. the case if $u$ is the identity of the unique object of the pre-$C^{*}$-category 
considered in Example \ref{fifjehwofewfewfwf}.

 For a $*$-category $\bA$ the counit $\bA\to \Lin(\bA)$ preserves unitaries.
 Similarly, for a pre-$C^{*}$-category $\bA$  the natural morphism $\bA\to \compl(\bA)$   preserves unitaries.

For a $\C$-linear $*$-category $\bB$  the counit  $\Bd^{\infty}(\bB)\to \bB$  is bijective on unitaries.
\hB
\end{remark}

Let $\cC$ be in the list
 $$\{\scat_{1},\ClinCat_{1},\preCcat_{1},\Ccat_{1},\scat_{1}^{+},\ClinCat_{1}^{+},\preCcat_{1}^{+},\Ccat_{1}^{+}\}\ .$$
 \begin{lem} 
 The functor $\cC\to \Set$ which sends a category in $ \cC$ to  its set of unitary morphisms  is representable.
 \end{lem}
\begin{ddd}\label{fewiojewofwefewfewf}The object $\beins_{\cC}$ which represents this functor will be called the unitary morphism  classifier.\end{ddd}
 \begin{proof}
 We perform a case-by-case discussion.
In $\scat_{1}$  we define $\beins_{\scat_{1}}$ to be the category with objects $0$ and $1$ and non-trivial morphisms
$u:0\to 1$ and $u^{*}=u^{-1}:1\to 0$.
 In view of the adjunction  \eqref{fwefweiufhui23r}   we have
$$\beins_{\ClinCat_{1}}\cong \Lin(\beins_{\scat_{1}})\ .$$
Since the generator $u$ is sent to a unitary in any representation
it is clear that $\beins_{\ClinCat_{1}}$ is a pre-$C^{*}$-category. 
Since  $\|u\|_{\max}=1$  it is actually a $C^{*}$-category.

Hence we have isomorphisms
$$\beins_{\preCcat_{1}}\cong \beins_{\ClinCat_{1}}\cong \beins_{\Ccat_{1}}$$ 
(as $\C$-linear $*$-categories).
In the marked cases, for $\cC$ in the list $\{\scat_{1},\ClinCat_{1},\preCcat_{1},\Ccat_{1}\}$ we have $\beins_{\cC^{+}}\cong \mi_{\cC}(\beins_{\cC})$, i.e., the universal unitary in $\beins_{\cC^{+}}$  is not marked.
\end{proof}

Let $\cC$ be a member of the list $$\{\scat_{1},\ClinCat_{1},\preCcat_{1},\Ccat_{1}\}$$

\begin{lem}
The functor
$\cC^{+}\to \Set$ which sends a category in $\cC^{+}$ to its set of marked morphisms is representable.
\end{lem}
\begin{ddd}\label{groiegegergegererer}
The object $\beins^{+}_{\cC}$ which represents this functor will be called the marked  morphism  classifier.
\end{ddd}
\begin{proof}
We have
$\beins^{+}_{\cC}\cong \ma(\beins_{\cC})$, i.e., the universal unitary is now marked.
\end{proof}

We now consider just categories.
\begin{lem}
 The functor which sends a category to its set of invertible morphisms
is representable. \end{lem}

\begin{ddd}\label{wffiweoffewfewfewf}
We call a category $\bbI$ which represents this functor the classifier of invertible morphisms.
\end{ddd}
\begin{proof}  
   The   groupoid $\bbI$  of the shape
$$\xymatrix{0 \ar@/_0.5cm/[r]&1\ar@/_0.5cm/[l]}$$
has the desired properties. The morphism $0\to 1$ is the universal invertible morphism.
  \end{proof}

\begin{remark}The groupoid $\bbI$ is also the 
  morphism classifier in $\Groupoids_{1}$. \hB
\end{remark}

\section{Unitary equivalences and  $\infty$-categories of $*$-categories}\label{ufhewiuhewfewf}

In this section we introduce the $\infty$-categories
of $*$-categories, $\C$-linear $*$-categories, pre-$C^{*}$-categories,   $C^{*}$-categories and their marked versions by inverting unitary (or marked, respectively) equivalences.

Let $\cC$ belong to the list  $$\{ \scat_{1},\ClinCat_{1},\preCcat_{1},\Ccat_{1},\Cat_{1}^{+},\scat_{1}^{+},\ClinCat_{1}^{+},\preCcat_{1}^{+},\Ccat_{1}^{+}\}\ .$$
Furthermore, let  $f,g:\bA\to \bB$ be a parallel pair of morphisms   in $\cC$.
\begin{ddd}
We say that $f$ and $g$ are (marked) unitary   equivalent, if there exists a natural  isomorphism of functors $u:f\to g$ such that
$u(a)$    is a (marked) unitary   morphism in $ \Hom_{\bB}(f(a),g(a))$ for every object  $a$ of $\bA$.  
\end{ddd}
Here the word \emph{marked}   applies in the marked cases. In these cases the word unitary can be omitted since marked morphisms are unitary by definition.

Let $f:\bA\to \bB$ be a morphism    in $\cC$.
\begin{ddd}\label{gihriugh3i4g3rgegegergege}
 The morphism $f$  is a (marked) unitary   equivalence    if there exists a morphism $g:\bB\to \bA$  in $\cC$  such that
$f\circ g$ is unitary (marked) isomorphic to  $\id_{\bB}$ and $g\circ f$ is  (marked) unitary   isomorphic  to $\id_{\bA}$. 
\end{ddd}
  
 The following characterization of unitary or marked equivalences will be useful later. 
 We let \begin{equation}\label{gvt4kjn4kjrtbrgbrgbrbrb}
\cF_{\mathrm{all}}:\cC\to \Cat
\end{equation}  be the functor which  takes the underlying category (i.e., forgets all additional structures and properties). Furthermore,  in the marked cases,  we consider the functor
 $$(-)^{+}:\cC\to \Cat$$
 which takes the subcategory of marked morphisms, see Remark \ref{wfiuwehifewfwefwe}.
 Finally recall the functor $\ma$   from the unmarked to the marked versions which marks all unitaries, see \eqref{wecoihorgrtb}.

 Let $f:\bA\to \bB$ be a morphism in $\cC$.
\begin{lem}
	\label{lem:markedequivs} \mbox{} \begin{enumerate}\item \label{friofj35t9u5gerkjg34t} in the marked cases: The morphism
	$f$  is a  marked  equivalence if and only if  $\cF_{\mathrm{all}}(f)$ and
	$f^{+}$ are equivalences of categories.
	\item  \label{friofj35t9u5gerkjg34t1} in the unmarked cases:  The morphism
	$f$  is a unitary    equivalence if and only if  $\cF_{\mathrm{all}}(f)$ and
	$\ma(f)^{+}$ are equivalences of categories.\end{enumerate}
 \end{lem}
\begin{proof} We start with \ref{friofj35t9u5gerkjg34t}.
If $f$ is a marked equivalence, then by Definition \ref{gihriugh3i4g3rgegegergege} there is an inverse morphism $g:\bB\to \bA$  up to marked isomorphism. 
Then $ \cF_{\mathrm{all}}(g)$ and $g^{+}$ are inverse equivalences of $\cF_{\mathrm{all}}(f)$ and $f^{+}$, respectively.

We now assume that $\cF_{\mathrm{all}}(f)$ and $f^{+}$ are equivalences of  categories. Then there exists a functor   $g^{+}:\bB^{+}\to \bA^{+}$ and   isomorphisms of functors
  $$u: \id_{\bB^+}\to  f^+\circ g^+\ , \quad  v: \id_{\bA^{+}}\to  g^{+}\circ f^{+}\ .$$  We   define a morphism  $g : \bB \to \bA$ in $\cC$ as follows:
  \begin{enumerate} \item
  on objects: For an object $b$ of $\bB$ we define 
  $g(b) := g^+(b)$. \item  on morphisms:
  For objects $b,b^{\prime}$ of $B$ we define  $g:\Hom_{\bB}(b,b^{\prime})\to \Hom_{\bA}(g(b),g(b^{\prime}))$ as the composition
	\[ \Hom_\bB(b,b') \xrightarrow{\cong ,!} \Hom_\bB(f(g(b)),f(g(b')) \xleftarrow{\cong,\cF_{\mathrm{all}}(f)} \Hom_\bA(g(b),g(b'))\ ,\]
	where the isomorphism marked by $!$ is given by
	$$\phi \mapsto u_{b^{\prime}}\circ  \phi\circ u_{b}^{-1} $$
	and we use that $\cF_{\mathrm{all}}(f)$ is an equivalence of categories for the second isomorphism.\end{enumerate}
	Note that $g$ preserves marked morphisms since $u_{b}$ and $u_{b^{\prime}}$ are marked and $\cF_{\mathrm{all}}(f)$ induces a  bijection between the subsets of marked morphisms (since $f^{+}$ is assumed to be an equivalence). Furthermore, since $u_{b}$ and $u_{b^{\prime}}$ are unitary (since marked morphisms must be unitary), $g$ is a morphism of $*$-categories. Finally, in the $\C$-enriched cases, $g$ is compatible with the enrichments.

The morphism $g$   is the required  inverse to $f$ up to marked isomorphism. The transformations  $u$ and $v$ can be interpreted as marked isomorphisms
$$u: \id_{\bB}\to  f\circ g\ , \quad  v: \id_{\bA}\to  g\circ f\ .$$ 

We now show  \ref{friofj35t9u5gerkjg34t1}. 
If $f$ is a unitary equivalence,  then there is an inverse morphism $g:\bB\to \bA$  up to unitary  isomorphism. 
Then $ \cF_{\mathrm{all}}(g)$ and $\ma(g)^{+}$ are inverse equivalences of $\cF_{\mathrm{all}}(f)$ and $\ma(f)^{+}$, respectively.

We now assume that $\cF_{\mathrm{all}}(f)$ and $\ma(f)^{+}$ are equivalences of  categories. 
Then by the first case \ref{friofj35t9u5gerkjg34t} we know that $\ma(f):\ma(\bA)\to \ma(\bB)$ is a marked equivalence. Let $g: \ma(\bB)\to  \ma(\bA)$ be an inverse of $\ma(f)$ up to marked isomorphism.
Then $\cF_{+}(g):\bB\to \bA$ ($\cF_{+}$ forgets the marking, see \eqref{wecoihorgrtb}) is an inverse of $f$ up to unitary isomorphism.
\end{proof}

\begin{remark}
In the case $\cC=\Ccat_{1}$ it was shown
 in \cite[Lemma 4.6]{DellAmbrogio:2010aa} that a morphism $f:\bA\to \bB$ is a unitary equivalence if and only if $\cF_{\mathrm{all}}(f)$ is an equivalence of categories, i.e., that the second condition in Lemma \ref{lem:markedequivs}.\ref{friofj35t9u5gerkjg34t1} involving $\ma(f)^{+}$ is redundant. The argument uses a special property of $C^{*}$-categories, namely the existence of polar decompositions of morphisms \cite[Prop. 2.6]{DellAmbrogio:2010aa}.\hB
\end{remark}

The following lemma about maximal norms morally belongs to Section \ref{ifjwofjwefwef} but can only be stated at this place since it involves the notion of unitary equivalences.

\begin{lem}\label{vijroiregrgregreg}
If $\Phi:\bA\to \bB$ is a unitary equivalence between $\C$-linear $*$-categories, then for every morphism $f$ in $\bA$ we have
$\|\Phi(f)\|_{\max}=\|f\|_{\max}$.
\end{lem}
\begin{proof}
By precomposition with $\Phi$ every representation of $\bB$ in a $C^{*}$-algebra yields a representation of $\bA$ in the same $C^{*}$-algebra. This immediately implies the inequality
$$\|\Phi(f)\|_{\max}\le \|f\|\ .$$
Let now $\Psi:\bB\to \bA$ be an inverse equivalence.
Then there exists a unitary morphism
$u$ in $\bA$ such that $u\circ \Psi(\Phi(f))\circ u^{*}=f$.
This gives (using $\|u\|_{\max}\le 1$, see Remark \ref{gregoi34jf34f34f})
$$\|f\|_{\max}=\|u\circ \Psi(\Phi(f))\circ u^{*}\|_{\max}\le \|  \Psi(\Phi(f))\|_{\max}\le \|\Phi(f)\|_{\max}\ .$$
\end{proof}

\begin{remark}\label{gigjeroigjoergergergregreger}
{We will   use the following fact about adjunctions. We consider two relative categories $(\cC,W_{\cC})$ and $(\cD,W_{\cD})$ and
 a pair of  adjoint functors \begin{equation}\label{gerg4r34r34r4}
L:\cC\leftrightarrows \cD:R \ .
\end{equation}  By $$\ell_{\cC}:\cC\to \cC[W_{\cC}^{-1}]\ , \quad \ell_{\cD}: \cD \to \cD[W_{\cD}^{-1}]$$ we denote the corresponding localization functors, see   \eqref{fwefoiu39r32r32r}.}
 
 {We now assume that $L$ and $R$ are compatible with the sets $W_{\cC}$ and $W_{\cD}$ in the sense 
   that $\ell_{\cD}\circ L$ sends the morphisms in $W_{\cC}$ to equivalences in $\cD[W_{\cD}^{-1}]$, and that $\ell_{\cC}\circ R$ sends the morphisms in $W_{\cD}  $ to equivalences in $\cC[W_{\cC}^{-1}]$.  Then the functors $L$ and $R$  descend essentially uniquely to functors $$\bar L:\cC[W^{-1}_{\cC}]\leftrightarrows  \cD[W_{\cD}^{-1}]:\bar R\ .$$  In this case the adjunction $ L\dashv R$ naturally induces an adjunction
 $\bar L \dashv \bar R$.
 A reference for these facts is \cite[Prop. 7.1.14]{cisin}.
 \footnote{
 Alternatively,
 one can use \cite[Prop. 5.2.2.8]{htt} as follows.   Let  $u:\id_{\cC}\to R\circ L$  be the unit and   $v:L\circ R\to \id_{\cD}$ be the counit of the adjunction \eqref{gerg4r34r34r4}. Then   $u$ induces a transformation
 $$\bar u:\id_{\cC[W_{\cC}^{-1}]}\to \bar R\circ \bar L\ .$$ 
We must   show that $\bar u$ is a unit transformation in the sense of \cite[Def. 5.2.2.7]{htt}, i.e., that
for any two objects $C$ in $\cC[W_{\cC}^{-1}]$ 
and $D$ in  $\cD[W_{\cD}^{-1}]$ the induced morphism
$$\Map_{\cD[W_{\cD}^{-1}]}(\bar L(C),D)\stackrel{\bar R}{\to}
\Map_{\cC[W_{\cC}^{-1}]}(\bar R(\bar L(C)),\bar R(D))\stackrel{\bar u(C)}{\to} \Map_{\cC[W_{\cC}^{-1}]}( C,\bar R(D))$$
is an equivalence of spaces. Using the fact that $\bar u$ and $\bar v:\bar L\circ \bar R\to \id_{\cD[W_{\cD}^{-1}]}$ induced by $v$ satisfy the triangle identities up to equivalence we see that the desired inverse equivalence is given by 
$$\Map_{\cC[W_{\cC}^{-1}]}( C,\bar R(D))\stackrel{\bar L}{\to} \Map_{\cD[W_{\cD}^{-1}]}( \bar L(C),\bar L(\bar R(D)))\stackrel{\bar v(D)}{\to}  \Map_{\cD[W_{\cD}^{-1}]}( \bar L(C), D)\ .$$} 
}
   \hB
   
   If $L$ comes from a left Quillen functor between combinatorial model categories, then we could also proceed as sketched in \cite[Rem. 1.3.4.27]{HA}. \hB
\end{remark}

Let $\cC$ be in the list $$\{\scat_{1},\ClinCat_{1},\preCcat_{1},\Ccat_{1},\scat_{1}^{+},\ClinCat_{1}^{+},\preCcat_{1}^{+},\Ccat_{1}^{+}\}\ ,$$  and let $W_{\cC}$  denote the (marked) unitary equivalences in $\cC$ as defined in Definition \ref{gihriugh3i4g3rgegegergege}.
\begin{ddd}
We define the $\infty$-categories $$ \scat:=\scat_{1}[W^{-1}_{\scat_{1}} ] \ , \quad \scat^{+
}:=\scat_{1}^{+}[W_{\scat_{1}^{+}}^{-1} ]\ .$$
$$\ClinCat:=\ClinCat_{1}[W_{\ClinCat_{1}}^{-1}]\ , \quad \ClinCat^{+}:=\ClinCat_{1}^{+}[W_{\ClinCat^{+}_{1}}^{-1}] \ ,$$
 $$\preCcat:=\preCcat_{1}[W_{ \preCcat_{1}}^{-1}]\ , \quad \preCcat^{+}:=\preCcat^{+}_{1}[W_{ \preCcat^{+}_{1}}^{-1}]  \ ,$$
$$\Ccat:=\Ccat_{1}[W_{\Ccat_{1}}^{-1} ] \ , \quad \Ccat^{+}:=\Ccat^{+}_{1}[W_{ \Ccat^{+}_{1}}^{-1} ]\ .$$
\end{ddd}

\begin{lem}\label{wefoijweofewfwf45}
The adjunctions \eqref{fwefweiufhui23r} induce   adjunctions  \begin{equation}\label{fwefweifrfufhui23r}
\Lin: \scat\leftrightarrows  \ClinCat:\cF_{\C}\ , \quad \Lin: \scat^{+}\leftrightarrows  \ClinCat^{+}:\cF_{\C}
\end{equation}
\end{lem}
\begin{proof} 
We   observe  that the forgetful functor  $\cF_{\C}$ and the linearization functor  $\Lin$  
preserve (marked) unitary   equivalences. Hence they descend naturally to the $\infty$-categories. 
 {By Remark \ref{gigjeroigjoergergergregreger} we obtain an adjunction between these descended functors.}
\end{proof}

\begin{lem}
The adjunctions \eqref{ewfwoijioiffewfwefw} induce    adjunctions
 \begin{equation}\label{ewfwoijioiffewfwef1w}
 \compl:\preCcat  \leftrightarrows \Ccat :\cF_{-}\ , \quad \compl:\preCcat^{+}  \leftrightarrows \Ccat^{+}:\cF_{-}
\end{equation}
\end{lem}
\begin{proof}
We first observe that the forgetful functor   $\cF_{-}$ and the completion functor  $\compl$ preserve (marked) unitary equivalences. Hence they descend naturally to the $\infty$-categories.  {By Remark \ref{gigjeroigjoergergergregreger} we obtain an adjunction between these descended functors.} \end{proof}

 \begin{lem}\label{fiofuweoifuwoefwefewfw}
 The adjunctions \eqref{fiioioefjeofiewfewfwf} induce  adjunctions
 \begin{equation}\label{fiioioefjeofiewfe22334wfwf}
\cF_{\pre}:\preCcat\leftrightarrows \ClinCat:\Bd^{\infty}\ , \quad \cF_{\pre}:\preCcat^{+}\leftrightarrows \ClinCat^{+}:\Bd^{\infty}\ .
\end{equation}
 \end{lem}
\begin{proof}
The forgetful functor $\cF_{\pre}$ preserves (marked) unitary   equivalences. The operation $\Bd^{\infty}$ also preserves   (marked) unitary  equivalences since
$\Bd^{\infty}(\bA)$ contains all unitary (marked, resp.) morphisms of $\bA$.  {Hence  both functors descend naturally to the $\infty$-categories.   By Remark \ref{gigjeroigjoergergergregreger} we obtain an adjunction between these descended functors.}  \end{proof}

\begin{conv}\label{fewoifuwe9of}{\em
We use the same notation $\ell$ for all the localization functors:
For $\cC$ in the list $$\{\scat_{1},\ClinCat_{1},\preCcat_{1},\Ccat_{1},\scat_{1}^{+},\ClinCat_{1}^{+},\preCcat_{1}^{+},\Ccat_{1}^{+}\}$$
we write \begin{equation}\label{dewkh23kr23r32r}
\ell:\cC\to \cC_{\infty}
\end{equation}
for the corresponding localization. \hB}
   \end{conv}

  \section{The tensor and power structure over groupoids}\label{rgiojogerggerg}
 
In this section we let   $\bG$ be a  category. Later we will assume that it is a groupoid. 

For a category $\bA$ we consider the functor category
 $\cFun( \bG,\bA)$ whose objects are functors from $\bG$ to $\bA$, and whose morphisms are natural transformations between functors. If 
  $\bA$ is a $*$-category, then we define an involution
 $$*:\cFun(\bG,\bA)\to \cFun(\bG,\bA)$$ such that it sends a morphism $f=(f_{g})_{g\in \bG}:a\to a^{\prime}$   in $\cFun(  \bG,\bA)$ with $f_{g}:a(g)\to a^{\prime}(g)$ to the  morphism $f^{*}:=(f_{g}^{*})_{g\in \bG}:a^{\prime}\to a$.

Assume furthermore that $\bA$ is a $\C$-linear $*$-category. Then the enrichment of
 $\bA$ over complex vector spaces naturally induces an enrichement of $\cFun( \bG,\bA)$. In this case 
  $\cFun( \bG,\bA)$ has the structure of  a $\C$-linear $*$-category.
If $\bA$ is marked, then $\cFun(\bG,\bA)$ is a marked $*$-category or
marked $\C$-linear $*$-category whose marked morphisms are those transformations
$(f_{g})_{g\in G}$ where $f_{g}$ is marked for all $g$ in $G$.

For $\cC$ in the list
$$\{\scat_{1},\scat_{1}^{+},\ClinCat_{1},\ClinCat_{1}^{+}\}$$ we  therefore get a functor $$\cFun( -,-): \Cat_{1}^{\mathrm{op}}\times \cC\to  \cC\ .$$

 Let $\bA$ be a (marked) $*$-category.
\begin{ddd} We call a functor  $a$ in $\cFun( \bG	,\bA)$(marked)  unitary   if $a(\phi) $ is unitary (marked) for all morphisms $\phi$ in $ \Hom_{\bG}(g,h)$. \end{ddd}

We let  $\cFun^{u}( \bG,\bA)$  
denote the full subcategory of $ \cFun(\bG,\bA)$ of unitary   functors. It is a $*$-category by Remark \ref{eiuuifwefewfewfw}. Similarly, if $\bA$ is a  $\C$-linear $*$-category, then so is $\cFun^{u}( \bG,\bA)$.

If $\bA$ is marked, then we let
$\cFun^{u}( \bG,\bA)$ denote the full subcategory of $\cFun( \bG, \bA)$ of marked functors.
It is again a marked $*$-category. If $\bA$ is a marked $\C$-linear $*$-category, then so is $\cFun^{u}( \bG,\bA)$.

 Note that a functor in $\cFun^{u}( \bG,\bA)$ sends all morphisms of $\bG$ to invertible morphisms in $\bA$. This observation will be used e.g. in Remark \ref{ewiogjwoeirgerwgewrgwtwehtrh} below.

For $\cC$ in the list
$$\{\scat_{1} ,\ClinCat_{1},\scat_{1}^{+} ,\ClinCat_{1}^{+}\}$$ we 
   have  a functor
$$\cFun^{u}( -,-): \Cat_{1}^{\mathrm{op}}\times \cC \to \cC\ . $$

\begin{remark} \label{uhweiuvwevwvewdw} In the marked case
the notation $\cFun^{u}( \bG,\bA)$ is actually an abuse of notation since this could also be interpretet as the category of unitary functors between $\bG$ and $\bA$ after forgetting the marking.  But we prefer to use this notation with the interpretation as above in order to state formulas below in a form  which applies to the unmarked as well as to the marked cases.   

{If $\bA$ is a (marked) pre-$C^{*}$-category or a (marked) $C^{*}$-category, then as a convention, in order to   form the (marked) $\C$-linear $*$-category  $\cFun^{u}(\bG,\bA)$ we will consider  $\bA$ as a (marked) $\C$-linear category and interpret $\cFun^{u}(\bG,\bA)$ as a (marked) $\C$-linear $*$-category.}\hB
\end{remark}

 \begin{remark}\label{igjoigregwergwergrgwrg}
 If $\bC$ is a $C^{*}$-category, then we can not expect that $\cFun^{u}(\bG,\bA)$ is a $C^{*}$-category again. For the simplest counter example let $\bG$ be an infinite set and $\bA$ be the category associated to a $C^{*}$-algebra $A$.
 Then $\cFun^{u}(\bG,\bA)$ is the $\C$-linear $*$-category with one object and with morphisms $ \prod_{\bG}A $. But this is not even a pre-$C^{*}$-category. In order to get a $C^{*}$-category, for the morphisms we should take the   uniformly bounded sequences
 $\prod_{g\in \bG}^{\mathrm{\mathrm{bd}}}A$. So we must define a uniformly bounded subfunctor
 $$\cFun^{\mathrm{\mathrm{bd}}}(\bG,\bA)\subseteq \cFun^{u}(\bG,\bA)\ .$$\hB
  \end{remark}
  
 \begin{ddd}
 For a (marked) $\C$-linear $*$-category $\bA$ we define 
the uniformly bounded subfunctor  by $$\cFun^{\mathrm{\mathrm{bd}}}(\bG,\bA):=\Bd^{\infty}(\cFun^{u}(\bG,\bA))\ .$$  
\end{ddd}
 
 By definition $\cFun^{\mathrm{\mathrm{bd}}}(\bG,\bA)$ is a (marked) pre-$C^{*}$-category. 
  {In particular,} for $\cC$ in the list $\{ \preCcat_{1}, \preCcat_{1}^{+}\}$
 we have defined a functor
 $$\cFun^{\mathrm{\mathrm{bd}}}(-,-):\Cat_{1}^{\mathrm{op}}\times \cC\to \cC\ .$$
 
  \begin{example} We have
 $\cFun^{\mathrm{\mathrm{bd}}}(\Delta^{0}_{\Cat_{1}},\bA)\cong \Bd^{\infty}(\bA)$. \hB \end{example}
 
 \begin{example}\label{eiugwgregwergregwerg}
 Assume that $\bG$ is a set and that $\bA$ is associated to a   $C^{*}$-algebra $A$. Then
 $\cFun^{\mathrm{\mathrm{bd}}}(\bG,\bA)$ can be identified with the category with one object  and the morphisms given by the $C^{*}$-algebra $\prod_{g\in \bG}^{\mathrm{bd}}A$.  \hB
  \end{example}

\begin{ddd} \label{egkjwoggrwegwergw} If $\bA$ is a (marked)  $\C$-linear $*$-category, then we define the (marked) $C^{*}$-category  $$\cFun^{C^{*}}( \bG,\bA):=\compl(\cFun^{\mathrm{bd}}(\bG,\bA))\ .$$   \end{ddd}

 {In particular,} for $\cC$ in the list $\{ \Ccat_{1}, \Ccat_{1}^{+}\}$
 we have defined a functor
 $$\cFun^{C^{*}}(-,-):\Cat_{1}^{\mathrm{op}}\times \cC\to \cC\ .$$

 \begin{remark}\label{ewiogjwoeirgerwgewrgwtwehtrh}
If $\bA$ is a (marked) $C^{*}$-category, then in Definition \ref{egkjwoggrwegwergw} one can actually omit the application of the completion functor.
Let $\bG[\bG^{-1}]$ denote the groupoid obtained from $\bG$ by universally inverting all morphisms. 
Then  the natural functor $\bG\to \bG[\bG^{-1}]$ induces an isomorphism
$$\cFun^{u}( \bG[\bG^{-1}],\bA)\stackrel{\cong}{\to} \cFun^{u}( \bG ,\bA)$$ for any (marked)  $*$-category $ \bA$.
If $\bA$ is $\C$-linear, then we get an isomorphism
$$\cFun^{\mathrm{bd}}( \bG[\bG^{-1}],\bA)\stackrel{\cong}{\to} \cFun^{\mathrm{bd}}( \bG ,\bA)\ .$$
If $\bA$ is a (marked) $C^{*}$-category, then in  Corollary \ref{eroijoietert}  we will see that $\cFun^{\mathrm{bd}}( \bG[\bG^{-1}],\bA)$ 
is a (marked)  $C^{*}$-category.  Consequently,  $\cFun^{\mathrm{bd}}( \bG ,\bA)$ is already  a (marked)  $C^{*}$-category.
\hB\end{remark}

 \begin{remark}
 If $\bA$ is a $C^{*}$-category, then by Remark \ref{ewiogjwoeirgerwgewrgwtwehtrh}   
 $\cFun^{\mathrm{bd}}(\bG,\bA)$ is again a $C^{*}$-category. 
 
 We have an obvious alternative candidate $\cFun^{r}(\bG,\bA)$ for the functor $C^{*}$-category
which is also defined as a subcategory
of $\cFun^{u}(\bG,\bA)$ as follows.  The objects of $\cFun^{r}(\bG,\bA)$ are the objects of $\cFun^{u}(\bG,\bA)$.
 Recall that given two functors $a,a^{\prime}$ in $ \cFun^{u}(\bG,\bA)$ we have an inclusion 
  $$\Hom_{ \cFun^{u}( \bG,\bA)}(a,a^{\prime})\subseteq \prod_{g\in \bG}\Hom_{\bA}(a(g),a^{\prime}(g))\ ,$$
    where a family $(f_{g})_{g\in \bG}$  belongs to this subspace if for every morphism $\phi$ in $\Hom_{\bG}(g,h)$
\begin{equation}\label{dqwduigui1eg12ee}
f_{h} \circ a(\phi)=a^{\prime}(\phi)\circ f_{g}\ .
\end{equation} 
We now define the morphisms of $\cFun^{r}(\bG,\bA)$  by 
$$\Hom_{\cFun^{r}(\bG,\bA)}(a,a^{\prime}):= \Hom_{\cFun^{u}(\bG,\bA)}(a,a^{\prime})\cap  \prod^{\mathrm{bd}}_{g\in \bG}\Hom_{\bA}(a(g),a^{\prime}(g))\ .$$
Since the relations  \eqref{dqwduigui1eg12ee} are linear and continuous the morphism space  $\Hom_{\cFun^{r}(\bG,\bA)}(a,a^{\prime})$ is a closed linear subspace of the Banach space $ \prod^{\mathrm{bd}}_{g\in \bG}\Hom_{\bA}(a(g),a^{\prime}(g))$ with the norm
$$\|(f_{g})_{g\in \bG}\|:=\sup_{g\in \bG} \|f_{g}\|_{\max}\ .$$ 
So 
  $\Hom_{ \cFun^{r}( \bG,\bA)}(a,a^{\prime})$ inherits a Banach space structure.  Let now $f\in \Hom_{ \cFun^{r}(\bG,\bA)}(a,a^{\prime})$. Then we have
 $$\|f^{*}\circ f\|=\sup_{g\in \bG}\|f^{*}_{g}\circ f_{g}\|=\sup_{g\in \bG} \|f_{g}\|=\|f\|\ ,$$
 i.e., the $C^{*}$-identity is satisfied. Similarly we conclude the $C^{*}$-inequality \ref{fiofiowefwfwefwef}.
 It follows from the universal property of $\Bd^{\infty}$ that we have a morphism of  $C^{*}$-categories
 $$\cFun^{r}( \bG,\bA)\to \cFun^{\mathrm{bd}}( \bG,\bA)\ .$$
 In general we do not know whether this is an isomorphism.
 But it is an isomorphism if  $\bG$ is a set and $\bA$ is associated to a $C^{*}$-algebra, see Example \ref{eiugwgregwergregwerg}.
  \hB
 \end{remark}

Let $\bG$ be a groupoid and $\bA$ be a $*$-category.
We   consider {$\bG$ as a $*$-category in the canonical manner (see Example \ref{fioefjewoifewfewfwfw}) and form  the $*$-category $\bA\times \bG$ (the existence of the product is ensured by Theorem \ref{gbeioergergregerg}).
Explicitly, the $*$-category  $\bA\times \bG$ is the product category with} 
 the $*$-operation
\begin{equation}\label{kdjhquiiuwdwqdwqd4}
(f,\phi)^{*}:=(f^{*},\phi^{-1})\ .
\end{equation} 
Since this definition involves the inverse of the morphism $\phi$ of $\bG$  it is important to assume that $\bG$ is a groupoid. 

If $f$ is a unitary morphism in $\bA$ and $\phi$ is any morphism in $\bG$, then $(f,\phi)$ is a  unitary morphism in $\bA\times \bG$.
If $\bA$ is a marked $*$-category, then  in  the $*$-category $\bA\times \bG$ we mark all morphisms
of the form $(f,\phi)$ with $f$ marked in $\bA$ and $\phi$ in $\bG$ arbitrary.

For $\cC$ in the list $\{\scat_{1},\scat_{1}^{+}\}$ 
we thus have defined a functor
$$-\times -:\cC\times \Groupoids_{1}\to\cC\ .$$

In the $\C$-linear case we must modify this construction.
For  a $\C$-linear $*$-category $\bA$  and a groupoid $\bG$  we define 
$\bA\otimes \bG$
to be the category with the objects of $\bA\times \bG$,   and whose morphisms are given by the complex
vector spaces
$$\Hom_{\bA\otimes \bG}((a,g),(a^{\prime},g^{\prime})):=\bigoplus_{\phi\in \Hom_{\bG}(g,g^{\prime})}\Hom_{\bA}(a,a^{\prime})$$
with the obvious $*$-operation and composition. Note that the sum over an empty index set is the zero vector space.

We note that
$\bA\times \bG$ is a wide subcategory of $\bA\otimes \bG$ in a natural way. 
If $\bA$ is marked, then in $\bA\otimes \bG$ we again mark all morphisms of the form $(f,\phi)$ for $f$ a marked morphism in $\bA$ and an arbitrary morphism $\phi$ of $\bG$.

For $\cC$ in the list $\{\ClinCat_{1},\ClinCat_{1}^{+}\}$ 
we thus have defined a functor
$$-\times -:\cC\times \Groupoids_{1}\to\cC\ .$$

\begin{example}
If $\bA$ is a (marked) $*$-category, then we have an isomorphism
$$\Lin(\bA)\otimes \bG\cong \Lin(\bA\times \bG)\ .$$
\hB\end{example}

\begin{example}\label{efwoiweiofwefwwefwef}
 If $\bG=BH$ for some group $H$, then  the $\C$-linear $*$-category
$\Delta^{0}_{\ClinCat_{1}}\otimes BH$ is isomorphic to  the $\C$-linear $*$-category associated to the group ring $\C[H]$ with its usual involution.
\hB
\end{example}

Let $\bG$ be a groupoid.
 \begin{lem}\label{roigoergergerger}
If $\bA$ is  a (marked) pre-$C^{*}$-category, then so is $\bA\otimes \bG$.\end{lem}
\begin{proof}
It suffices to show that \begin{equation}\label{iojioffwef243}
\bA\otimes \bG= \Bd^{\infty}(\bA\otimes \bG)\ .
\end{equation}
 
Every morphism in $\bA\otimes \bG$ is a finite linear combination of morphisms of the form  $(f,\phi)$ of  $\bA\times \bG$ with $f:a\to a^{\prime}$ and $\phi:g\to g^{\prime}$.  We can decompose $(f,\phi)=(\id_{a^{\prime}},\phi)\circ (f,\id_{g})$.

Let $\rho: \bA\otimes \bG \to B$ be a representation in a $C^{*}$-algebra $B$. 
Then we can  restrict $\rho$ to a  representation of 
$\bA\cong \bA\times \{g\}\subseteq \bA\otimes \bG$. We conclude that $\|\rho(f,\id_{g})\|_{B}\le \|f\|_{\max}$, where $\|-\|_{\max}$ denotes the maximal norm on $\bA$. Furthermore, because of \eqref{kdjhquiiuwdwqdwqd4} we know that 
$\rho(\id_{a^{\prime}},g)$ is a partial isometry   in $B$. Therefore
$\|\rho(f,\phi)\|_{B} \le \|f\|_{\max}$. 
This shows that $(f,g)\in \Bd^{\infty}(\bA\otimes \bG)$.  Hence also all finite linear combinations of such elements belong to
$\Bd^{\infty}(\bA\otimes \bG)$. This shows  the desired equality \eqref{iojioffwef243}.
\end{proof}

For $\cC$ in the list $\{\preCcat_{1},\preCcat_{1}^{+}\}$ 
we thus have defined a functor
$$-\otimes -:\cC\times \Groupoids_{1}\to\cC\ .$$

For a (marked) $\C$-linear  $*$-category $\bA$  and a groupoid $\bG$ we define
$$\bA\otimes_{\max}\bG:=\compl(\Bd^{\infty}(\bA\otimes \bG))\ .$$
 If $\bA$ was a (marked) pre-$C^{*}$-category, then by Lemma \ref{roigoergergerger}  we can simplify this to 
 $$\bA\otimes_{\max}\bG\cong \compl( \bA\otimes \bG)\ .$$

For $\cC$ in the list $\{\Ccat_{1},\Ccat_{1}^{+}\}$ 
we thus have defined a functor
$$-\otimes_{\max} -:\cC\times \Groupoids_{1}\to\cC\ .$$

\begin{example}\label{fwoijfofewfewfewfewfew}
We consider the $C^{*}$-category $\Delta_{\Ccat_{1}}^{0}$  associated to  the $C^{*}$-algebra $\C$.
For a groupoid $\bG$ the $C^{*}$-category $\Delta_{\Ccat_{1}}^{0}\otimes_{\max}\bG$ is the maximal groupoid $C^{*}$-category.
In particular, if $\bG= BH$ for a group $H$, then
$\Delta^{0}_{\Ccat_{1}}\otimes_{\max} BH$ is isomorphic to the $C^{*}$-category associated to the maximal group $C^{*}$-algebra $C^{*}_{\max}(H)$. \hB
\end{example}

 \begin{conv}\label{fwerifhiwefewfewf}{\em 
  In order to avoid a case-dependent notation we write $\sharp$ for the tensor structures with groupoids   $\times$, $\otimes$, or $\otimes_{\max}$ in the respective cases.  We will furthermore use the notation $\cFun^{?}$, where $?$ is $u$, $\mathrm{bd}$, or    $C^{*}$     in the respective cases. See Table \ref{default}.
  
  \begin{table}[htp]
\caption{}
\begin{center}
\begin{tabular}{|c||c|c|}\hline
case&$\sharp$&?\\\hline $\scat_{1}^{(+)}$&$\times$&$u$\\\hline 
$\ClinCat_{1}^{(+)}$&$\otimes$&$u$\\\hline $\preCcat_{1}^{(+)}$&$\otimes$&$\mathrm{bd}$\\\hline$\Ccat_{1}^{(+)}$&$\otimes_{\max}$&$C^{*}$\\\hline
\end{tabular}
\end{center}
\label{default}
\end{table} \hB}\end{conv}

\begin{example}  \label{griuhuiwfwefwefef} Let  $\cC$ be in the list $$\{\scat_{1},\ClinCat_{1},\preCcat_{1},\Ccat_{1},\scat_{1}^{+},\ClinCat_{1}^{+},\preCcat_{1}^{+},\Ccat_{1}^{+}\}$$ and recall the  morphism classifier object $\bbI$ in $\Groupoids_{1}$ from  Definition \ref{wffiweoffewfewfewf}. 
Let $f_{0},f_{1}:\bC\to \bD$ be two morphisms in $\cC$.
Then we have a bijective correspondence between (marked) unitary isomorphisms
$u:f_{0}\to f_{1}$ and functors $U:\bC\sharp\bbI\to \bD$ with $U\circ \iota_{i}=f_{i}$, where $\iota_{i}:\bC\cong \bC\sharp \Delta^{0}\to \bC\sharp \bbI$ for $i=0,1$ is induced by the objects $0$ and $1$ of $\bbI$.
Given $U$, the  (marked) unitary isomorphism $u$ is  obtained by $u=(U(\id_{c},0\to 1))_{c\in \bC}$.
Vice versa, given $u$, we can define $U$ on morphisms by
$U(\id_{c},0\to 1)=u_{b}$ and $U(a,\id_{0}) := f_{0}(a)$ and compatibility with compositions and $*$.\hB
\end{example}

Let $\bG$ be a groupoid and   $\cC$ be in the list $$\{\scat_{1},\ClinCat_{1},\preCcat_{1},\Ccat_{1},\scat_{1}^{+},\ClinCat_{1}^{+},\preCcat_{1}^{+},\Ccat_{1}^{+}\}\ .$$
\begin{prop}\label{efiheiufhwiufwfwe}
For $\bC$ and $\bA$ in $\cC$
we have  a natural exponential law $$\Hom_{\cC}(\bC\sharp \bG,\bA)\cong \Hom_{\cC}(\bC,\cFun^{?}(\bG,\bA))\ .$$  
\end{prop}
\begin{proof}
We start with giving the natural bijection in the case of $*$-categories.
Let $\Phi$ be  in $\Hom_{\scat_{1}}(\bC\times \bG,\bA)$. The bijection identifies this morphism with  a morphism
$\Psi$ in  $\Hom_{\scat_{1}}(\bC,\cFun^{u}(\bG,\bA))$ given  by
$$\Psi(c)(g):=\Phi(c,g)\ , \quad  \Psi(c)(\phi):=\Phi(\id_{c},\phi)\ , \quad \Psi(f):=(\Phi(f,\id_{g})_{g\in \bG})\ .$$
Here $c$ is an object of $\bC$, $g$ is an object of $\bG$, $\phi$ is a morphism in $\bG$, and $f$ is a morphism in $\bC$.
Note that $\Psi(c)$ takes values in unitary functors since $$\Psi(c)(\phi)^{*}=\Phi(\id_{c},\phi)^{*}=\Phi(\id_{c},\phi^{*})=\Phi(\id_{c},\phi^{-1})=\Psi(c)(\phi)^{-1}\ .$$
Vice versa let the morphism $\Psi$ be given. Then  the bijection sends it to the  morphism $\Phi$ given by
$$\Phi(c,g):=\Psi(c)(g)\ , \quad \Phi(f,\phi) :=\Psi(c^{\prime})(\phi)\circ \Psi(f)(g)\ ,$$
where $f:c\to c^{\prime}$ is a morphism in $\bC$ and $\phi :g\to g^{\prime} $ is  a morphism in $\bG$.

The same formulas work  in the $\C$-linear case for morphisms of the form $(f,\phi)$. The maps are then extended linearly. 
In the case of pre-$C^{*}$-categories we must check that
$\Psi$ takes values in the subfunctor
$\cFun^{\mathrm{bd}}(\bG,\bA)$ of $ \cFun^{u}(\bG,\bA)$. But this is clear since $\Psi$ is a morphism between $\C$-linear categories, $\bC$ is a pre-$C^{*}$-category, and the universal property \ref{frefrkjhgiekgergerg} of $\Bd^{\infty}$.

We finally consider the case of $C^{*}$-categories. In this case we could cite  \cite{DellAmbrogio:2010aa}. Here is the argument.
We first observe that for a $C^{*}$-category $\bA$ and pre-$C^{*}$-category $\bC$ we have a bijection \begin{equation}\label{veroihrioufererf}
\Hom_{\preCcat_{1}}(\bC\otimes \bG,\cF_{-}(\bA))\cong \Hom_{\Ccat_{1}}(\bC\otimes_{\max} \bG,\bA)
\end{equation}
 by the universal property of the completion. We can use this bijection, the already verified bijection
 $$\Hom_{\preCcat_{1}}(\bC\otimes \bG,\cF_{-}(\bA))\cong \Hom_{\preCcat_{1}}(\bC,\cFun^{\mathrm{bd}}(\bG,\cF_{-}(\bA)))\ , $$  and the completion map
 \begin{equation}\label{veroihriofffffufererf}
\Hom_{\preCcat_{1}}(\bC,\cFun^{\mathrm{bd}}(\bG,\cF_{-}(\bA)))\to \Hom_{\preCcat_{1}}(\bC,\cF_{-}(\cFun^{C^{*}}(\bG,\bA)))
\end{equation} 
in order to produce a $\Psi$ from a given $\Phi$.

For the inverse we note
 that since $\bA$ is complete,   the evaluation functors
$$e_{g}:\cFun^{\mathrm{bd}}(\bG,\cF_{-}(\bA))\to \bA$$ for $g$ in $\bG$ extend to the completion and provide functors
$$\bar e_{g}:\cFun^{C^{*}}(\bG,\bA)\to \bA\ .$$
Hence the formula which expresses  $\Phi$  in terms of $\Psi$ makes sense. It defines, by linear extension, an element  in $\Hom_{\preCcat_{1}}(\bC\otimes \bG,\cF_{-}(\bA))$ which gives the desired $\Phi$ in $\Hom_{\Ccat_{1}}(\bC\otimes_{\max} \bG,\bA)$ using the isomorphism  \eqref{veroihrioufererf} above.

In the marked case we just observe  the following. If $\Phi$ is a functor between marked categories, then $\Psi$ takes values in marked functors, and vice versa, if $\Psi$ has this property, then $\Phi$ preserves marked morphisms.
\end{proof}

Let $\bG$ be a groupoid.
 \begin{kor}\label{eroijoietert}
 If $\bA$ is a (marked) $C^{*}$-category, then  the completion morphism 
 $$\cFun^{\mathrm{bd}}(\bG,\cF_{-}(\bA))\to \cF_{-}(\cFun^{C^{*}}(\bG,\bA))$$ is an isomorphism and $\cFun^{\mathrm{bd}}(\bG,\bA)$ is a (marked) $C^{*}$-category. 
 \end{kor}
\begin{proof}In the proof of Proposition \ref{efiheiufhwiufwfwe}
 we have actually shown  that for every pre-$C^{*}$-category
$\bC$ there is  a natural isomorphism
$$\Hom_{\preCcat_{1}}(\bC\otimes_{\max} \bG,\cF_{-}(\bA))\cong \Hom_{\preCcat_{1}}(\bC,\cF_{-}(\cFun^{C^{*}}(\bG,\bA)))\ .$$
Furthermore we have a natural isomorphism
$$\Hom_{\preCcat_{1}}(\bC\otimes_{\max} \bG,\bA)\cong  \Hom_{\preCcat_{1}}(\bC,\cFun^{\mathrm{bd}}(\bG,\bA))\ .$$
By an inspection of the construction of these bijections we see that the resulting isomorphism  $$ \Hom_{\preCcat_{1}}(-,\cFun^{\mathrm{bd}}(\bG,\bA))\cong  \Hom_{\preCcat_{1}}(-,\cF_{-}(\cFun^{C^{*}}(\bG,\bA)))$$  of functors $\preCcat_{1}^{\mathrm{op}}\to \Set$ is induced by the completion morphisms $$ \cFun^{\mathrm{bd}}(\bG,\bA)\to \cF_{-}(\cFun^{C^{*}}(\bG,\bA))\ .$$ By the Yoneda Lemma it 
  is therefore an isomorphism, too.
\end{proof}

  In the following we introduce the  fundamental groupoid functor  $\Pi:\sSet\to \Groupoids_{1}$. 
   \begin{ddd}\label{fewiohfwiof2323r23r}  The  fundamental groupoid functor  is defined as a left-adjoint  of an  adjunction
$$\Pi:\sSet\leftrightarrows \Groupoids_{1}:\Nerve\ ,$$
where $\Nerve$ takes the nerve of a groupoid.\end{ddd}
Explicitly, the fundamental groupoid $\Pi(K)$ of a simplicial set  $K$ is the groupoid freely generated by the path  category $P(K)$ of $K$. The category $P(K)$ in turn is given as follows:
\begin{enumerate}
\item The objects of $P(K)$ are the $0$-simplices.
\item The morphisms of $P(K)$ are generated by the $1$-simplices of $K$ subject to the relation
$g\circ f\sim h$ if there exists a $2$-simplex $\sigma$ in $K$ with
$d_{2} \sigma=f$, $d_{0}\sigma=g$ and $d_{1}\sigma=h$.
\end{enumerate}

 In the following definition we use   the notation introduced in Covention \ref{fwerifhiwefewfewf}. 
  Let $\cC$ be a member of the list  $$ \{\scat_{1},\ClinCat_{1},\preCcat_{1},\Ccat_{1},\scat_{1}^{+},\ClinCat_{1}^{+},\preCcat_{1}^{+},\Ccat_{1}^{+}\}\ .$$

\begin{ddd}\label{riooejrgegerreg}
We define the tensor and cotensor structure of $\cC$ with $\sSet$ by
$$\cC\times \sSet\to \cC\ , \quad (\bA,K)\mapsto {\bA\sharp K:=} \bA\sharp \Pi(K)$$
and
$$ \sSet^{\mathrm{op}} \times \cC  \to \cC\ , \quad (K,\bB)\mapsto  {\bB^{K}}:=\cFun^{?}(\Pi(K),\bB )\ .$$
 For objects $\bA,\bB$ of $\cC$ we define the simplicial mapping space 
$\Map(\bA,\bB)$ in $\sSet$ by
$$\Map(\bA,\bB)[-]:=\Map(\bA\sharp  {\Delta^{-}} ,\bB)\ .$$
\end{ddd}
Then for every two objects 
 $\bA, \bB$ in $\cC$ and every simplicial set  $K$, by Proposition \ref{efiheiufhwiufwfwe}  we have natural bijections\begin{equation}\label{lrwrknewkjfewkjfwefwfwef}
\Hom_{\sSet}(K, \Map(\bA,\bB))\cong \Hom_{\cC}(\bA\sharp K,\bB)\cong \Hom_{\cC}(\bA, \bB^{K})\ .\end{equation}
In this way we have defined a simplicial enrichment of $\cC$.

\section{The resolution}\label{oijfowiefjweoifewfewfewfe}

Let $G$ be a group. In the present section we consider the  functor categories from the arrow category  $\tilde G$.  
These functor categories will serve as explicit fibrant resolutions later in Section \ref{geroijoegergregeg}. 

A $G$-groupoid (or $G$-category) is a groupoid (or category) with a strict action of $G$, i.e., an object of $\Fun(BG,\Groupoids_{1})$ (or $\Fun(BG,\Cat_{1})$).

\begin{ddd}\label{fewoihfiweiofjwefoiewfewfewf}   The arrow category $\tilde G$ is the  $G$-groupoid
 defined as follows: \begin{enumerate} \item The    set of objects of $\tilde G$ is the underlying set of $G$.
  \item For every pair of objects  $g,h$ of $\tilde G$ the set of morphisms $\Hom_{\tilde G}(g,h)$ consists of one point which we will denote by $g\to h$. The composition of morphisms is defined in the only possible way.    \item   The group $G$ acts on the groupoid $\tilde G$ by left multiplication. \end{enumerate}
  \end{ddd}

  For a $G$-category $\bA$     the functor category
 $\cFun(\tilde G,\bA)$  is  again a  $G$-category.
 The group $G$ acts on this functor category as follows. 
  If $a:\tilde G\to \bA $ is a functor, then we set
 $g(a):=g\circ a\circ g^{-1}$. The action on morphisms is similar.
  We interpret this construction as a functor
  $$\cFun(\tilde G,-):\Fun(BG,\Cat_{1})\to \Fun(BG,\Cat_{1})\ .$$
 This construction extends to the various versions of  (marked) $*$-categories.  
For $\cC$ in the list  $$ \{\scat_{1},\ClinCat_{1},\preCcat_{1},\Ccat_{1},\scat_{1}^{+},\ClinCat_{1}^{+},\preCcat_{1}^{+},\Ccat_{1}^{+}\}$$
we get a functor
$$\cFun^{?}(\tilde G,-):\Fun(BG,\cC)\to \Fun(BG,\cC)\ ,$$
 see Convention \ref{fwerifhiwefewfewf} for notation.
 
\begin{lem}\label{fri9ofwefewfewf}
If $\bA$ is a (marked) pre-$C^{*}$-category, then the natural morphism
$$\cFun^{\mathrm{bd}}(\tilde G,\bA)\to \cFun^{u}(\tilde G,\bA)$$
is an isomorphism of (marked) $\C$-linear $*$-categories.
\end{lem}
\begin{proof}
We must show that every morphism in $ \cFun^{u}(\tilde G,\bA)$ has a finite   maximal norm.
 If $a$ is a unitary functor from $\tilde G$ to $\bA$, then we have a unitary isomorphism
 $$h_{a}:=(a(1\to g))_{g\in \tilde G}:\const(a(1))\to a\ ,$$
 where $\const(a(1))$ denotes the constant functor with value $a(1)$. Unitarity implies the norm estimates
  $\|h_{a}\|_{\max}\le 1$ and $\|h_{a}^{-1}\|_{\max}\le 1$, see Remark \ref{gregoi34jf34f34f}.
 
 If $f:a\to a^{\prime}$ is a morphism in $ \cFun^{u}(\tilde G,\bA)$, then we have the relation.
 $$h_{a^{\prime}}\circ \const(f(1))\circ h^{-1}_{a}=f\ .$$ This implies the inequality
 $$\|f\|_{\max}\le   \| \const(f(1))\|_{\max}\le \|f(1)\|_{\max}\ .$$
  \end{proof}

 Using Corollary \ref{eroijoietert} we conclude:
 \begin{kor}\label{regoijeroigergjergeg}
For a (marked) $C^{*}$-category $\bA$ the natural maps are isomorphisms 
 $$\cFun^{C^{*}}(\tilde G,\bA)\stackrel{\cong}{\leftarrow}\cFun^{\mathrm{bd}}(\tilde G,\bA)\stackrel{\cong}{\to}\cFun^{u}(\tilde G,\bA)$$  isomorphism of (marked) $\C$-linear $*$-categories.
\end{kor}

 \begin{remark}\label{grrigherjiogreg434t3t34t4t34t}
 In view of Corollary \ref{regoijeroigergjergeg} and Lemma \ref{fri9ofwefewfewf} we have an  isomorphism of (marked) $\C$-linear $*$-categories.
 $$\cFun^{?}(\tilde G,\bA)\cong \cFun^{u}(\tilde G,\bA)$$ in all cases, see Convention \ref{fwerifhiwefewfewf} and  Remark \ref{uhweiuvwevwvewdw} for the usage of notation. \hB
  \end{remark}

We have a $G$-equivariant version  of the
 exponential law.
Let  $\cC$ be in the list  $$ \{\scat_{1},\ClinCat_{1},\preCcat_{1},\Ccat_{1},\scat_{1}^{+},\ClinCat_{1}^{+},\preCcat_{1}^{+},\Ccat_{1}^{+}\}\ .$$   

  \begin{prop}\label{eoifhewiufewfewfewfewfewf}  For $\bA$ and $\bC$ in $\Fun(BG,\cC)$ we have  a natural isomorphism $$\Hom_{\Fun(BG,\cC)}(\bC\sharp \tilde G,\bA)\cong \Hom_{\Fun(BG,\cC)}(\bC,\cFun^{?}(\tilde G,\bA))\ .$$

 \end{prop}
 \begin{proof}
 This follows from an inspection of the proof of Proposition \ref{efiheiufhwiufwfwe}.\end{proof}

 \section{Completeness, cocompleteness, and local presentability}\label{foijowefwefewfewf}
 
 The first goal of this section is to show the following theorem. 
\begin{theorem}\label{gbeioergergregerg} The categories in the list  $$\{\scat_{1},\ClinCat_{1},\preCcat_{1},\Ccat_{1},\scat_{1}^{+},\ClinCat_{1}^{+},\preCcat_{1}^{+},\Ccat_{1}^{+}\}$$ are complete and cocomplete. 
 \end{theorem}
\begin{remark}
The case of $C^{*}$-categories is due to \cite{DellAmbrogio:2010aa}. For completeness of the presentation  we  reprove this case together with the others.\hB
\end{remark}

\newcommand{\DGr}{\mathbf{DirGraph}}
\begin{remark}\label{rgrweuihzreiurefrergerg} 
In the proof we use that the categories of categories  and marked categories     $\Cat_{1}$ and $\Cat_{1}^{+}$ are complete and cocomplete. In particular we must understand colimits in   $\Cat_{1}$  in some detail.  
Let $\DGr$ denote the category of directed graphs. Then we have an adjunction
\begin{equation}\label{iovjoijio3j4f3rf3ff}
\Free_{\mathrm{Cat}}:\DGr\leftrightarrows \Cat_{1}:\cF_{\circ}\ ,
\end{equation} where $\cF_{\circ}$ sends a category to its underlying directed graph (i.e., forgets the composition), and
$\Free_{\mathrm{Cat}}$ sends a directed graph to the category freely generated by it. 
Colimits in $\DGr$ are formed by taking the colimits of the sets  of vertices and edges separately.

If $\bC$ is a category, then the counit of the adjunction \eqref{iovjoijio3j4f3rf3ff}
provides a functor
\begin{equation}\label{ogih3iojf3f3fc}
\Free_{\mathrm{Cat}}(\cF_{\circ}(\bC))\to \bC\ .
\end{equation} 
It is a bijection on objects. We
 $R(\bC)$ denote the equivalence relation on the morphisms of  $\Free_{\mathrm{Cat}}(\cF_{\circ}(\bC)) $ generated by the action of the functor \eqref{ogih3iojf3f3fc} on morphisms.  This relation is compatible with the category structure and the functor \eqref{ogih3iojf3f3fc} induces  an isomorphism
 $$\Free_{\mathrm{Cat}}(\cF_{\circ}(\bC))/R(\bC)\cong \bC\ .$$
 We now consider 
 a diagram $\bA:I\to \Cat_{1}$. Then
 we form the category $\Free_{\mathrm{Cat}}(\colim_{I}\cF_{\circ}(\bA))$. On the morphisms of this  category we consider the smallest equivalence relation $R$  compatible with the category structure  which contains the images of
 $R(\bA(i))$ under  the canonical maps $$\Mor(\Free_{\mathrm{Cat}}(\cF_{\circ}(\bA(i))))\to\Mor (  \Free_{\mathrm{Cat}}(\colim_{I}\cF_{\circ}(\bA)))$$
 for all $i$ in $I$. Then one can check that
 $$\Free_{\mathrm{Cat}}(\colim_{I}\cF_{\circ}(\bA))/R\cong \colim_{I}\bA\ .$$
We will in particular need the following conclusion of this discussion: Every morphism  in $ \colim_{I}\bA$
 is   a finite composition of morphisms of the form $u_{i}(f_{i})$ where,  for $i$ in $I$,
 $u_{i}:\bA(i)\to \colim_{I}\bA$ is the canonical functor and $f_{i}$ is in $\Mor(\bA(i))$. \hB

\end{remark}

\begin{proof}(of Theorem \ref{gbeioergergregerg}) We first show completeness. Our starting point is the fact that
 the $1$-category $\Cat_{1}$ is complete. If $\bA:I\to \Cat_{1}$ is a diagram of categories, then we have canonical isomorphisms of sets
\begin{equation}\label{fiouoifwefwefwef}
\Ob(\lim_{I}\bA)\cong \lim_{I}\Ob(\bA)\ , \quad \Mor(\lim_{I}\bA)\cong \lim_{I}\Mor(\bA)\ .
\end{equation}

Next we consider  $\scat_{1}$. Because of the adjunction \eqref{ferferfeoijiorojr34r34r34r}  a limit  in $\scat_{1}$ (provided it exists) can be calculated as follows. One first calculates the limit in   $\Cat_{1}$ (i.e.,  after forgetting  the $*$-operation), and then restores the $*$-operation using the functoriality of the limit.  Thus let $\bA:I\to \scat_{1}$ be a diagram.  We choose 
a category $\bB$ together with a morphism $  \const(\bB)\to  \cF_{*}(\bA) $ which exhibits $\bB$ as a representative of the limit $\lim_{I} \cF_{*}(\bA)$ .
Then we can take $\bB^{\mathrm{op}}$ and $  \const(\bB)^{\mathrm{op}}\to \cF_{*}(\bA)^{\mathrm{op}}$  in order to exhibit the  limit
$\lim_{I}\cF_{*}( \bA)^{\mathrm{op}}$. The transformation $*:\cF_{*}(\bA)\to \cF_{*}(\bA)^{\mathrm{op}}$ now determines a unique functor  
$*:\bB\to \bB^{\mathrm{op}}$  such that
$$\xymatrix{\const (\bB)\ar[d]\ar[r]^{\const(*)}&\const(\bB^{\mathrm{op}})\ar[r]^{\cong} &\const(\bB)^{\mathrm{op}}\ar[d]\\ \cF_{*}(\bA)\ar[rr]^{*}&&\cF_{*}(\bA)^{\mathrm{op}}}$$ commutes. This functor is the identity on objects and turns
 $\bB$  into  a $*$-category. Together with the morphism of diagrams of $*$-categories
$\const(\bB)\to \bA$
it represents the limit  $\lim_{I}\bA$ in $\scat_{1}$.

In order to deal with the marked case we observe that $\Cat_{1}^{+}$ is complete.
If $\bA:I\to \Cat_{1}^{+}$ is a diagram of marked categories, then in order to construct the marked category $\lim_{I}\bA$, using the forgetful functor $\cF_{+}$ from \eqref{wecoihorgrtb} we form the 
category $\lim_{I}\cF_{+}( \bA)$ and mark
the morphisms   which are the elements of  $\Mor(\lim_{I}\cF_{+}(\bA))$  whose evaluation (use \eqref{fiouoifwefwefwef}) at all objects $i$ of $I$ are marked.
We can now repeat the constructions above with $\bA$ and $\bB$ marked.

The same idea works for a diagram $\bA:I\to \ClinCat_{1}$ of  $\C$-linear $*$-categories. 
Since the forgetful functor $\cF_{\C}$ from  \eqref{fwefweiufhui23r}  is the right-adjoint of the adjunction  
it preserves limits. Hence we have
$$\cF_{\C}(\lim_{I}\bA)\cong \lim_{I} \cF_{\C}(\bA)$$ if the limit on the left-hand side  exists. The limit on the 
right-hand side is interpreted in $\scat_{1}$ and exists as we have seen above.
In the following we argue that  limit on the left-hand side  indeed exists.   
For a diagram $\bA:I\to \ClinCat_{1}$ we first define the object $\lim_{I}\cF_{\C}(\bA)$ in $\scat_{1}$,
 then observe that the limit  has an induced complex enrichment. We then observe that the canonical
 morphism
 $$\const(\lim_{I}\cF_{\C}(\bA))\to \bA$$ is compatible with the enrichment and exhibits  the $*$-category
 $\lim_{I}\cF_{\C}(\bA)$ together with the enrichment as the limit $\lim_{I}\bA$.
 
 The same argument applies in the marked case.

Let $\cF_{\pre}$ and $\Bd^{\infty} $ be the functors as in \eqref{fiioioefjeofiewfewfwf}.
For a diagram of pre-$C^{*}$-categories $\bA:I\to \preCcat_{1}$ we have an isomorphism
  \begin{equation}\label{g4goihio34tj34oi34f}
\lim_{I} \bA\cong \Bd^{\infty}(\lim_{I}\cF_{\pre}(\bA))\ ,
\end{equation}
where the limit on the right-hand side is interpreted in $\ClinCat_{1}$.
This immediately follows from the adjunction  \eqref{fiioioefjeofiewfewfwf} since $\cF_{\pre}$ is a fully faithful inclusion. 
The same argument applies in the marked case.

Let $\cF_{-} $ be the forgetful functor  from \eqref{ewfwoijioiffewfwefw}. Since it is the right-adjoint of an adjunction
  it is clear that it preserves limits. We consider a diagram of $C^{*}$-categories $\bA:I\to \Ccat_{1}$. Then we have
$$\cF_{-}(\lim_{I}\bA)\simeq \lim_{I}\cF_{-}(\bA)$$ if   the limit $ \lim_{I}\bA$   exists. 
The limit on the right-hand side is interpreted in $\preCcat_{1}$ and exists as seen above.
We now argue that the limit on the left-hand side indeed exists. 
Every limit can be expressed as a finite combination of equalizers and products. It therefore suffices to show that
$\lim_{I}\cF_{-}(\bA)$ is a $C^{*}$-category in the case that $I$ is a set or a finite category.

If $I$ is a set, then the  limit is represented by the bounded product
$$\lim_{I}\bA\cong \prod_{i\in I}^{\mathrm{bd}} \bA(i)\ .$$
Indeed, this product has the universal property 
since morphisms between $C^{*}$-categories are norm-bounded by $1$. 
 
 If $I$ is finite, then we realize $\lim_{I}\bA$ as a subcategory of
$\prod_{I}\bA(i)$ cut out by linear $*$-invariant continuous equations.  This is  again $C^{*}$-category.   

We now consider the marked case. If $\bA:I\to\Ccat_{1}^{+}$ is a diagram of marked $C^{*}$-categories, then
$ \lim_{I}\cF_{-}(\bA)$ is a marked pre-$C^{*}$-category. Above we have seen that it is  also a marked $C^{*}$-category which then necessarily represents $\lim_{I}\bA$.

This finishes the proof of completeness in all cases.

We now show cocompleteness. 

We start with the cocompleteness of $\scat_{1}$. In the argument     we employ   the fact that $\Cat_{1}$ is cocomplete, see Remark \ref{rgrweuihzreiurefrergerg}.
  Let $\bA:I\to \scat_{1}$ be a diagram. Then we choose 
a category $\bB$ together with a morphism $\cF_{*}(\bA)\to \const(\bB)$ which exhibits $\bB$ as a  colimit $\colim_{I} \cF_{*}( \bA)$, where $\cF_{*}$ is as in \eqref{ferferfeoijiorojr34r34r34r}.
Then we can take $B^{\mathrm{op}}$ and $\cF_{*}(\bA)^{\mathrm{op}}\to \const(\bB)^{\mathrm{op}}$ as a representative of the  colimit
$\colim_{I} \cF_{*}(\bA)^{\mathrm{op}}$. Similarly as in the case of limits, the transformation $*:\cF_{*}(\bA)\to \cF_{*}(\bA)^{\mathrm{op}}$ now induces a functor
$*:\bB\to \bB^{\mathrm{op}}$ in the canonical way such that  it  is the identity on objects. Hence $\bB$ has the structure of a $*$-category. The canonical morphism of diagrams of $*$-categories
$\bA\to \const(\bB)$ exhibits   $  \bB$ as the colimit $\colim_{I}\bA$.

In the marked case we use  that we have already shown that $\scat_{1}$  is cocomplete. If $\bA:I\to \scat_{1}^{+}$ is a diagram, then in order to construct the marked $*$-category $\colim_{I}\bA
$ we first form the $*$-category $\bB:=\colim_{I}\cF_{+}(\bA)$, where $\cF_{+}$ is as in \eqref{wecoihorgrtb}. In $\bB$ we   mark all
morphisms which are compositions of   morphisms of the form $u_{i}(f_{i})$ for $i$ in $I$, a marked morphisms $f_{i}$ in $\bA(i)$, and where  
$u_{i}:\cF_{+}(\bA(i))\to \colim_{I}\cF_{+}(\bA)=\bB$ is the canonical $*$-functor.
Since $u_{i}$ is a $*$-functor we have for a marked $f_{i}$ that
$$u_{i}(f_{i})^{*}=u_{i}(f_{i}^{*})=u_{i}(f_{i}^{-1}
)=u_{i}(f_{i})^{-1}\ .$$
This implies that the marked morphisms in $\bB$ are unitary.
By construction they are closed under composition so that $\bB$ is a marked $*$-category, and 
$\bA\to \const(\bB)$ is a morphism of diagrams of marked $*$-categories.
It exhibits $\bB$ as the colimit $\colim_{i}\bA$ in $\scat_{1}^{+}$.

 In order to construct colimits in   $\ClinCat_{1}$ we use the adjunction \eqref{fwefweiufhui23r}.    If $\bC$ is a $\C$-linear $*$-category, then we have a natural exact sequence of  $\C$-linear $*$-categories
 $$0\to R(\bC)\to \Lin( \cF_{\C}(\bC))\to \bC\to 0\ ,$$ with the caveat
 that $R(\bC)$ is non-unital. The   second map in this sequence is the counit of the   adjunction \eqref{fwefweiufhui23r}.
 For a diagram $\bA:I\to \ClinCat_{1}$ we now define $\bB$ as the quotient 
$$0\to \langle R(\bA(i))\:|\: i\in I\rangle \to  \Lin( \colim_{I}\cF_{\C}(\bA)) \to \bB\to 0\ ,$$
where  $\langle R(\bA(i))\:|\: i\in I\rangle$ is the non-unital $\C$-linear $*$-subcategory  of $\Lin( 
\colim_{I}\cF_{\C}(\bA))$ generated as an ideal by the images of $R(\bA(i))$ for all $i$ in $I$.
By construction  for every $i$ in $I$ we have a canonical factorization 
$$\xymatrix{ \Lin( \cF_{\C}(\bA(i)))\ar[r]\ar[d]&\Lin(\colim_{I}\cF_{\C}(\bA))\ar[d]\\\bA(i)\ar@{..>}[r]&\bB}$$
In this way we get a morphism of diagrams
$\bA\to \const(\bB)$. We now observe that this exhibits $\bB$ as the colimit $\colim_{I}\bA$ in $\ClinCat_{1}$.
Indeed, let $\bC$ be in $\ClinCat_{1}$, and  a morphism $\bA\to \const(\bC)$  of diagrams in $\ClinCat_{1}$ be given. Then we get a morphism
$$ \Lin(  \cF_{\C}(\bA))\to \const(\bC)$$ in $\ClinCat$ from the counit of the adjunction \eqref{fwefweiufhui23r}.
Hence we get a uniquely determined morphism
$$ \Lin( \colim_{I}  \cF_{\C}(\bA))\to \bC$$ since colimits in $\scat_{1}$ exist and the left adjoint  $\Lin$  preserves colimits.  We finally see that this morphism, by definition of $\bB$, uniquely factorizes over a morphism $\bB\to \bC$.

In order to deal with the marked case we use the marked version of the adjunction   \eqref{fwefweiufhui23r} and argue in a similar manner using that  we have shown above that $\scat_{1}^{+}$ is cocomplete.

  We now consider colimits in $\preCcat_{1}$.
The inclusion  functor $\cF_{\pre}$ is the left-adjoint of the adjunction \eqref{fiioioefjeofiewfewfwf}.
Consequently it preserves colimits.  It therefore suffices to show that  if $\bA:I\to  \preCcat_{1}$ is a diagram of pre-$C^{*}$-categories, then the $\C$-linear $*$-category
  $\colim_{I}\cF_{\pre}(\bA) $ is in fact  a pre-$C^{*}$-category.
  Let $f$ be a morphism in $\colim_{I}\cF_{\pre}(A)$. 
  We must show that $\|f\|_{\max}$ is finite. 
  By the description of colimits in $\Cat_{1}$ given in Remark \ref{rgrweuihzreiurefrergerg}, and by the construction of colimits in $\scat_{1}$, $\ClinCat_{1}$ given above,  the morphism  $f$ is equal to a  finite linear combinations of finite compositions morphisms  of the form $u_{i}(f_{i})$, where  $f_{i}$ belongs to $ \cF_{\pre}(\bA(i))$  and $u_{i}:\cF_{\pre}(\bA(i))\to \colim_{I}\cF_{\pre}(\bA)$ is the canonical morphism. 
Hence we can assume that $f=u_{i}(f_{i})$.   If $\rho$ is    a representation of $ \colim_{I}\cF_{\pre}(\bA) $   into a $C^{*}$-algebra $B$, then
  $\rho\circ u_{i}$ is a representation of  $\cF_{\pre}(\bA(i))$.   It follows that
  $\|\rho(f)\|_{B}\le \|\rho(u_{i}(f_{i}))\|_{B}\le \|f(i)\|_{\max}$.
  
The same argument applies in the marked case.

If $\bA:I\to  \Ccat_{1}^{(+)}$ is a diagram of (marked) $C^{*}$-categories, then we 
have an isomorphism \begin{equation}\label{referf43543r}
\colim_{I} \bA\cong  \compl(\colim_{I}\cF_{-}(\bA)) \ ,
\end{equation}
where the colimit on the right-hand side is interpreted in $\preCcat_{1}^{(+)}$ and exists as seen above.
 The isomorphism \eqref{referf43543r}   follows immediately from the adjunction \eqref{ewfwoijioiffewfwefw} since the forgetful functor $\cF_{-}$  is fully faithful.

\end{proof}

The remainder of   this section  is devoted to Proposition \ref{ewdfoijfowefewfewfw} showing that some of the $*$-categories considered in the present paper are locally presentable.
 
The following discussion serves as a preparation.
Let $\cC$ be some category. {Recall from \cite[Sec. 0.6]{AR} that  a generator of  $\cC$ is a set of objects $\cG$ of $\cC$ such that
for every two distinct morphisms $f,g:C\to D$ in $\cC$ there exists a morphism $h:G\to C$ for some $G$ in $\cG$ such that $f\circ h\not= g\circ h$.  
The generator is strong if in addition 
for every object $C$ of $\cC$ and proper subobject $D$ of $C$ there exists a morphism $h:G\to C$ for some $G$ in $\cG$ which does not factor over $D$.}

Let $\cG$ be a subset of objects of $\cC$.

\begin{lem}\label{bgorijgoergergerg}
If every object of $\cC$ is isomorphic to a colimit of a    diagram in $\cC$ with values in $\cG$, then  $\cG$ is a strong generator of $\cC$.
\end{lem}
\begin{proof}
Let $f,g:C\to D$ be two distinct  morphisms in $\cC$. Let $B:I\to \cC$ be a   diagram with values in $\cG$ such that
$C\cong \colim_{I} B$. Then we have a bijection of sets $$\lim_{I^{\mathrm{op}}} \Hom_{\cC} (B,D) \cong \Hom_{\cC}(C,D)\ .$$ Because of $f\not=g$   there   exists $i$ in $I$ such that
$f\circ h(i)\not=g\circ h(i)$, where $h(i):B(i)\to \colim_{I}B \cong  C$ is the canonical map.
Note that $B(i)$ belongs to $\cG$ by assumption.

Let now $\iota:D\to C$ be the inclusion of  a proper subobject. We again consider a diagram
  $B:I\to \cC$  with values in $\cG$ such that
$C\cong \colim_{I} B$.  

We argue by contradiction and assume that every morphism $G\to C$ with $G$ in $\cG$ factors over $\iota$. Since $D$ is a subobject this factorization is unique. The canonical morphism of $I$-diagrams
$B\to \underline{C}$ therefore provides a morphism of $I$-diagrams $B\to \underline{D}$ and hence a morphism
$\pi:C\cong \colim_{I}B\to D$ by the universal property of the colimit. Since
$B\to \underline{D}\stackrel{\underline{\iota}}{\to} \underline{C} $ is the canonical morphism of $I$-diagrams for the presentation of $C$ we conclude that 
$\iota \circ \pi=\id_{C}$. We now argue that also $\pi\circ \iota=\id_{D}$ and hence $\iota$ is an isomorphism. This is in conflict with the  assumption that $\iota$ is the inclusion of a proper subobject, and hence we get   the desired contradiction.
 We know that 
  $$\iota\circ (\pi\circ \iota)=(\iota\circ \pi)\circ \iota=\id_{C}\circ \iota=\iota\ .$$ Since also
$\iota\circ \id_{D}=\iota$ and $\iota$ is a monomorphism, we conclude that $\id_{D}=\pi\circ \iota$ as required.
\end{proof}

If $\cG$ satisfies the assumption of Lemma \ref{bgorijgoergergerg},
 then we say that $\cG$ strongly generates $\cC$.

\newcommand{\Quivers}{\mathbf{Quivers}}
\begin{prop}\label{ewdfoijfowefewfewfw}
If $\cC$ belongs to $\{\scat_{1},\ClinCat_{1},\scat_{1}^{+},\ClinCat_{1}^{+}\}$, then  the   category $\cC$ is locally presentable.
\end{prop}
\begin{proof} 
{The categories in question are cocomplete. By \cite[Thm. 1.20]{AR} it suffices to show that they have a strong generator  formed by $\kappa$-presentable objects for some regular cardinal $\kappa$. In order to exhibit a strong generator will  use the criterion shown in  Lemma \ref{bgorijgoergergerg}.}

 We start with the case $\scat_{1}$.  The following discussion is related with Remark \ref{rgrweuihzreiurefrergerg}. A directed $*$-graph is a directed graph with an involution which preserves vertices and flips the direction of edges.
   We consider the category 
  ${}^{*}\DGr$ of directed $*$-graphs and involution-preserving morphisms.  Then we have an adjunction
  \begin{equation}\label{iovjoijio3j4f3rf3ff333}
\Free_{{}^{*}Cat}:{}^{*}\DGr\leftrightarrows \scat_{1}:\cF_{\circ}\ ,
\end{equation} where $\cF_{\circ}$ forgets the category structure and retains the $*$-operation.
The left-adjoint $\Free_{{}^{*}Cat}$ sends a directed $*$-graph to the $*$-category freely generated by it.
The category ${}^{*}\DGr$  
 is locally presentable. Indeed, it is cocomplete (as in the case of directed graphs, colimits are given by the colimits of    the sets of vertices and edges, separately), and it is strongly generated   by the objects in the list 
$$\{pt\ , \cF_{\circ}(\beins)\}$$ of  compact directed $*$-graphs. Note that
$pt$ is the directed $*$-graph with one vertex and no edges, and the directed $*$-graph $\cF_{\circ}(\beins)$ has two vertices $0$ and $1$ and the edges $u:0\to 1$ and $u^{*}:1\to 0$.  Given a $*$-category $\bA$
we can consider the free $*$-category $$\bF(A):=\Free_{{}^{*}Cat}(\cF_{\circ}(\bA))$$ generated by the underlying directed $*$-graph of $\bA$.
The counit of the adjunction \eqref{iovjoijio3j4f3rf3ff333} provides  
a canonical morphism
$$ v_{\bA}:\bF(\bA)\to \bA$$
of $*$-categories.
We claim that $v_{\bA}$ is an effective epimorphism, i.e, canonical coequalizer map
$$cv_{\bA}:\CEq\big(\bF(\bA)\times_{\bA}\bF(\bA) \rightrightarrows \bF(\bA)\big)\to \bA$$
is an isomorphism in $\scat_{1}$. We first observe that $v_{\bA}$ induces a  bijection on the level of objects.
Consequently, the coequalizer map $cv_{\bA}$ is a bijection on the level of objects, too. 
The morphisms of the coequalizer are given as a quotient of the morphisms in $\bF(\bA)$ by  the  equivalence relation 
induced by $v_{\bA}$ which is compatible with the $*$-category structure. It is now clear that   $cv_{\bA}$ is also a bijection on morphisms.

We know that $\bF(\bA) $ is isomorphic to a colimit of a   diagram involving the generators
\begin{equation}\label{hrtiojort3ggergergergregre}
\{  pt\ ,\bF(\beins) \}\ .
\end{equation}
 Note that the fibre product over $\bA$ is not a colimit. But we have a surjection 
$$\bF\big(\bF(\bA)\times_{\bA}\bF( \bA)\big) \to \bF(\bA)\times_{\bA} \bF(\bA)$$ and therefore an isomorphism
$$\CEq\Big(\bF\big(\bF(\bA)\times_{\bA}\bF( \bA)\big)  \rightrightarrows \bF(\bA)\Big)\to \bA\ .$$
The $*$-category $\bF\big(\bF(\bA)\times_{\bA}\bF( \bA)\big) $ is again a colimit of a diagram involving the generators in the list above. Hence $\bA$ itself is a  colimit of a diagram built from the list \eqref{hrtiojort3ggergergergregre}.
Since $\bF(\bbI)$ has two objects and countable morphism sets it  is $\aleph_{1}$-presentable.
It follows that $\scat_{1}$ is {strongly} generated by the list \eqref{hrtiojort3ggergergergregre} of  $\aleph_{1}$-presentable objects.
 
In the case of $\scat_{1}^{+}$ we argue similarly. We use the adjunction \begin{equation}\label{gveriljeor43ferfref}
\Free_{{}^{*}Cat^{+}}:{}^{*}\DGr^{+}\leftrightarrows \scat_{1}^{+}:\cF_{\circ}
\end{equation}
and that the category of marked directed $*$-graphs ${}^{*}\DGr^{+}$
  is {strongly} generated by the  list compact objects  $$\{pt\ , \cF_{\circ}(\mi(\beins_{\scat_{1}})),\cF_{\circ}(\ma(\beins_{\scat_{1}}))\}\ .$$
We can now repeat the argument. 
It follows that $\scat_{1}^{+}$ is {strongly} generated by the list of $\aleph_{1}$-presentable objects
$$\{pt\ ,\bF(\ma(\beins_{\scat_{1}}))\ ,\bF( \mi(\beins_{\scat_{1}}))\}\ .$$

A similar argument applies  in the $\C$-linear case. Here we use the adjunction \eqref{fwefweiufhui23r}
and the counit
$$v_{\bA}: \bF_{\C}(\bA):=\Lin(\cF_{\C}(\bA))\to \bA\ .$$
We again show that $v_{\bA}$ is an effective coequalizer and get
an isomorphism
$$\Coeq\Big(\bF_{\C}\big(\bF_{\C}(\bA)\times_{\bA} \bF_{\C}(\bA)\big) \rightrightarrows \bF_{\C}(\bA)\Big)\to \bA\ .$$
From the already verified case $\scat_{1}$ and the fact that the left-adjoint $\Lin$ preserves colimits we conclude
that for every $\bB$  in $\ClinCat_{1}$ the  object $ \bF_{\C}(\bB)$ is a colimit of a diagram with values in the list
\begin{equation}\label{ververvv}
\{\Lin(pt)\ , \Lin( \bF(  \beins_{\scat_{1}}))\}\ .
\end{equation} 
Since $\Lin( \bF(\beins_{\scat_{1}}))$ has two objects and countable-dimensional morphism spaces it is $\aleph_{1}$-presentable.
It follows that $\ClinCat_{1}$ is {strongly} generated by the list \eqref{ververvv}  of $\aleph_{1}$-presentable objects.

Similarly, 
$\ClinCat_{1}^{+}$ is {strongly} generated by the list of $\aleph_{1}$-presentable objects
$$
\{ \Lin(*)\ , \Lin( \bF(\mi(\beins_{\scat_{1}})))\ , \Lin( \bF(\ma(\beins_{\scat_{1}}))) \}\ .
$$
\end{proof}

\begin{remark}
At the moment we do not have an argument that $\preCcat_{1}$ or $\preCcat_{1}^{+}$ are  locally presentable.

Let  $\bA$ be a pre-$C^{*}$-category. Then by the above there exists a functor
$S:I\to \ClinCat_{1}$ such that $S(i)$ belongs to the list  \eqref{ververvv} for all $i$ in $I$ together with a transformation 
$u:S\to \const(\cF_{\pre}(\bA))$  such that its adjoint is an isomorphism $\colim_{I}S\cong \cF_{\pre}(\bA)$.
Then
$$\colim_{I}\Bd^{\infty}(S)\to \bA$$
is a candidate for a presentation of $\bA$.\hB
 \end{remark}

\begin{prop}\label{efuifzuew9fewfewfwf}
The categories $\Ccat_{1}$ and $\Ccat_{1}^{+}$ are locally presentable.
\end{prop}
\begin{proof}
There is a set  ${}^{*}\DGr^{fin}$ of finite directed $*$-graphs.
For any directed $*$-graph $Q$ we can consider the $\C$-linear $*$-category
$\Lin(\Free_{{}^{*}Cat}(Q))$ (see \eqref{iovjoijio3j4f3rf3ff333} for notation). We have a set of $C^{*}$-norms $N(Q)$ on $\Lin(\Free_{{}^{*}Cat}(Q))$. For every norm  $\|-\|$ in  $N(Q)$ we form the $C^{*}$-category
$\bA(Q,\|-\|)$ by taking the closure.

We claim that the set of $C^{*}$-categories $\bA(Q,\|-\|)$ as described above for all finite directed $*$-graphs $Q$ and norms $\|-\|$ in $N(Q)$ strongly generates $\Ccat_{1}$.   

Let $\bA$ be a $C^{*}$-category.  Any finite $*$-invariant collection  $F$ of morphisms  of $\bA$ defines a finite directed $*$-graph $Q(F)$ by forgetting the composition. By the universal property of the $\Lin\circ \Free_{{}^{*}Cat}$-functor we have a canonical morphism  
$$\Lin(\Free_{{}^{*}Cat}(Q(F)))\to \bA$$ which induces a norm $\|-\|_{F}$ in $N(Q(F))$. 
Then $\bA(F):=\bA(Q(F),\|-\|_{F})$ is naturally isomorphic to a sub $C^{*}$-category of $\bA$.
The set of finite $*$-invariant subsets of morphisms of $\bA$ is partially ordered by inclusion and filtered.
If $F$ is contained in a larger subset $F^{\prime}$, then    we clearly get a monomorphism $\bA(F)\to \bA(F^{\prime})$ of $C^{*}$-categories.  
We claim that $$\colim_{F} \bA(F)\cong \bA\ .$$ To this end we verify the universal property of the colimit. Let $\bB$ be a $C^{*}$-category. We must produce a natural bijection
$$\Hom_{\Ccat_{1}}(\bA,\bB)\cong \lim_{F} \Hom_{\Ccat_{1}}(\bA(F),\bB)\ .$$
 This bijection  identifies  a  morphism $\Phi$ in $\Hom_{\Ccat_{1}}(\bA,\bB)$  with the morphism  $\Psi$ in $ \lim_{F} \Hom_{\Ccat_{1}}(\bA(F),\bB)$.
 Given $\Phi$ we can find the system $\Psi=(\Psi_{F})_{F}$ by $\Psi_{F}:=\Phi_{|\bA(F)}$.
 Vice versa, given $\Psi=(\Psi_{F})_{F}$, then we define $\Phi$ as follows.
Let $a$ be an object of $ \bA$. Then we define
$\Phi(a):=\Psi_{\{\id_{a}\}}(a)$.  For a morphism $(f:a\to a^{\prime})$ in $ \bA$ we consider any $F$ such that $f\in F$. Then we define $\Phi(f):=\Psi_{F}(f)$. Note that $\Psi_{F}(f)$ is really a morphism from $\Phi(a)$ to $\Phi(a^{\prime})$.
Furthermore one checks that this definition   is independent of the choice of $F$. Therefore we get maps 
$\Phi$ on the level of objects and morphisms. We now show  that $\Phi$ is a morphism between $\C$-linear $*$-categories.  We discuss the compatibility with composition. Let $f$ and $g$ be composable morphisms.
Then we choose $F$ such that $f,g,g\circ f\in F$ and use that $\Psi_{F}$ is a morphism of $C^{*}$-categories.

We now claim that the $C^{*}$-categories $\bA(Q,\|-\|)$ for finite directed $*$-graphs and norms $\|-\|$ in $N(Q)$ are $\aleph_{2}$-compact.
 Let $$B:I\to      \Ccat_{1}\ , \quad i\mapsto \ B(i)$$  be an  $\aleph_{2}$-filtered diagram of $C^{*}$-categories and consider the natural map
\begin{equation}\label{regkj34ktg34gt34}
V:\colim_{i\in I}\Hom_{\Ccat_{1}}(\bA(Q,\|-\|), B(i))\to\Hom_{\Ccat_{1}}(\bA(Q,\|-\|), \colim_{i\in I} B(i))  \ .
\end{equation} 
We must show that $V$ is a bijection. We first discuss the  surjectivity of $V$.
Let $\Psi$ belong to $$\Hom_{\Ccat_{1}}(\bA(Q,\|-\|), \colim_{i\in I} B(i)) \ .$$ Since the directed $*$-graph $Q$ is finite   there exists an element  $i$ in  $I$ such that $\Psi(Q)\subseteq B(i)$. 
For every   $  j$ in $I_{\ge i}$ we get a morphisms of $\C$-linear $*$-categories $$\rho_{j}:\Lin(\Free_{{}^{*}Cat}(Q))\to B(i)\to B(j)\ .$$ This morphism induces  a norm
$\|-\|_{j}$ on  $\Lin(\Free_{{}^{*}Cat}(Q))$ by  $\|a\|_{j}:= \|\rho_{j}(a)\|_{B(j)}$. 
 If we can show that
$\|-\|_{k}\le \|-\|$ for some $k$ in $I_{\ge i}$, then  we get the desired  factorization $\Phi:A(Q,\|-\|)\to  B(k)$  such that $V(\Phi)=\Psi$.

 \newcommand{\Norms}{\mathrm{Norms}}
In order to find the element $k$ we consider the set $N(Q)$ of norms on 
$\Lin(\Free_{{}^{*}Cat}(Q))$ as a partially ordered set. 
Then we have an order-preserving map
$$\ell:I_{\ge i}\to N(Q)^{\mathrm{op}}\ , \quad j \mapsto  \|-\|_{j}\ .$$

We now observe that the size of  $N(Q)$   is bounded by $\aleph_{1}$.
A norm on $\Lin(\Free_{{}^{*}Cat}(Q))$ is determined by its restriction to  the   subcategory $\mathrm{Lin}_{\Q}(\Free_{{}^{*}Cat}(Q))$. We then use that $\mathrm{Lin}_{\Q}(\Free_{{}^{*}Cat}(Q))$ has countably many morphisms.

We let $J$ be a subset of objects  of  $I_{\ge i}$ obtained by choosing a preimage under $\ell$ for every norm
in the image $\ell(I_{\ge i})$. Then the size of $J$ is bounded by $\aleph_{1}$.
Since $I$ is $\aleph_{2}$-filtered  the subset $J$ has an upper  bound $k$ in $I_{\ge i}$. Then  by construction
$\|-\|_{k}\le \|-\|_{j}$ for all $  j$ in $I_{ \ge i}$.

Since for every morphism $a$ in $ \Lin(\Free_{{}^{*}Cat}(Q))$ we have the inequality
$$\|a\|_{k}\le \lim_{j\in I_{\ge i} } \|\rho_{j}(a)\|_{B(j)}=\|\Psi(a)\|_{\colim_{j\in I} B(j)}\le \|a\|$$ we have
$\|-\|_{k}\le \|-\|$ as desired.

We now consider injectivity of $V$ in  \eqref{regkj34ktg34gt34}.
Assume that $\Phi,\Phi^{\prime}$ in $$\colim_{i\in I}\Hom_{\Ccat_{1}}(\bA(Q,\|-\|), B(i))$$ are such that $V(\Phi)=V(\Phi^{\prime})$.    
We can assume that there is an element $j$ in $I$ such that
$\Phi$ and $\Phi^{\prime}$ are represented by morphisms  $\Phi_{j},\Phi^{\prime}_{j}$ in $\Hom_{\Ccat_{1}}(\bA(Q,\|-\|), B(j))$. For  every $i$ in $I_{\ge j}$ we write $\Phi_{i}$ and $\Phi_{i}^{\prime}$ for the morphisms obtained from $\Phi_{j}$ and $\Phi_{j} $ by post-composition with $B(j)\to B(i)$.
We must then show that there exists $k$ in $I_{\ge j}$ such that $\Phi_{k}=\Phi_{k}^{\prime}$.

Using that $Q$ is finite, after increasing $j$ if necessary, we can assume that $\Phi_{j}$ and $\Phi_{j}^{\prime}$ coincide on objects.  We furthermore write $V_{i}$ for the composition of $V$ with the canonical  map
$$ \Hom_{\Ccat_{1}}(\bA(Q,\|-\|), B(i))\to \colim_{l\in I}\Hom_{\Ccat_{1}}(\bA(Q,\|-\|), B(l))\ .$$
For  a morphism $\phi$ in $\bA(Q,\|-\|) $ we have the equality  \begin{equation}\label{b4toigh4oig34g4343r34r}
0=\|V(\Phi( \phi ))-V(\Phi^{\prime}( \phi ))\|=\lim_{i \in I_{\ge j}} \| \Phi_{i}( \phi )-  \Phi_{i}^{\prime}( \phi) \|
\end{equation} 
(note that the difference makes sense since $\Phi_{i}$ and $\Phi_{i}^{\prime}$ coincide on objects).

We now use that, by continuity,  $\Phi_{i}$ and $\Phi_{i}^{\prime}$ are uniquely determined by their restrictions along the functor $$d:\mathrm{Lin}_{\Q}(\Free_{{}^{*}Cat}(Q))\to \bA(Q,\|-\|)\ .$$ Because of \eqref{b4toigh4oig34g4343r34r}, for every morphism
$\phi$ in $\mathrm{Lin}_{\Q}(\Free_{{}^{*}Cat}(Q))$ and positive real number $r$   we can choose  $  i(\phi,r)$ in $I_{\ge j}$ such that
  $$\| \Phi_{i(\phi,r)}(d( \phi) )-  \Phi_{i(\phi,r)}^{\prime}( d(\phi)) \|\le r\ .$$
  Since the size of the set of morphisms $\phi$ and positive real numbers $r$ is bounded by $\aleph_{1}$  and $I$ is $\aleph_{2}$-filtered there exists   an element $k$ in $I_{\ge j}$ which is greater than all the elements $i(\phi,r)$ chosen above.  We conclude that $\| \Phi_{k}(d( \phi) )-  \Phi_{k}^{\prime}( d(\phi)) \|=0$ for all morphisms $\phi$ in 
  $\mathrm{Lin}_{\Q}(\Free_{{}^{*}Cat}(Q))$ and consequently $\Phi_{k}=\Phi_{k}^{\prime}$.
  This finally implies that $\Phi=\Phi^{\prime}$.
   
This finishes the proof of the Proposition \ref{efuifzuew9fewfewfwf} in the case of $\Ccat_{1}$.
 In the   case  of $\Ccat_{1}^{+}$ we argue similary  with marked directed $*$-graphs and use the functor $\Free_{{}^{*}Cat^{+}}$ from \eqref{gveriljeor43ferfref} instead of $\Free_{{}^{*}Cat}$.
\end{proof}

\section{Verification of the model category axioms}
  
  Let $\cC$  be a member of the list  \begin{equation}\label{vrelkkjn43joioij434}
 \{\scat_{1},\ClinCat_{1},\preCcat_{1},\Ccat_{1},\scat_{1}^{+},\ClinCat_{1}^{+},\preCcat_{1}^{+},\Ccat_{1}^{+}\}\ .
\end{equation} 
 In this section we state the main theorem on the model category structures again.
 We first recall the description of cofibrations, fibrations and weak equivalences.
 \begin{ddd}\label{fkjhifhiueiwhuiwhfiuewefewfe}\mbox{}
\begin{enumerate}
\item  \label{ewoiwoirwerwrwww} A weak equivalence in $\cC$ is a (marked) unitary   equivalence (see Definition \ref{gihriugh3i4g3rgegegergege}).
\item A cofibration is a  morphism  in $\cC$ which is injective on objects.
\item A fibration is a morphism  in $\cC$ which has the right-lifting property with respect to trivial cofibrations.
\end{enumerate}
\end{ddd}

In condition \ref{ewoiwoirwerwrwww} and below the word \emph{marked} only applies to the four marked versions. In the marked case,   by Lemma \ref{lem:markedequivs}.\ref{friofj35t9u5gerkjg34t} a weak equivalence detects marked morphisms.

For the simplicial structure we refer to Definition   \ref{riooejrgegerreg}.

\begin{theorem}\label{eifweofwefewfew}
The structures described in Definition \ref{fkjhifhiueiwhuiwhfiuewefewfe} and Definition \ref{riooejrgegerreg}  equip $\cC$ with a  simplicial  model category structure.

 If $\cC$  is a member of  $ \{\scat_{1},\ClinCat_{1}, \Ccat_{1},\scat_{1}^{+},\ClinCat_{1}^{+}, \Ccat_{1}^{+}\}$, then this model category structure is  cofibrantly generated  and  combinatorial.
\end{theorem}
 \begin{remark}
In the case of $\cC=\Ccat_{1}$ a proof of the first part of the theorem has been given in \cite{DellAmbrogio:2010aa}. \hB
\end{remark}

\begin{remark}
It is a lack of suitable morphism classifier objects in the pre-$C^{*}$-category cases, which prevents us to show cofibrant generation  in these cases, see also Remark \ref{geieojeoijreoreirferferf}.\hB
\end{remark}

In the present section we show  that the structures explained above determine a  model category structure on $\cC$. The simplicial axioms will be verified in Section \ref{riowefwefewfew}. Finally, 
 the additional assertions on cofibrant generation and combinatoriality are shown in Section \ref{uiehfiwefwefewfewfwf}

In the following we list the axioms (cf. \cite{hovey}) which we have to verify in order to show that $\cC$ with the structures given in Definition \ref{fkjhifhiueiwhuiwhfiuewefewfe}  is a model category:
\begin{enumerate}
\item ((co)completeness) Completeness and cocompleteness have been verified in Section \ref{foijowefwefewfewf}. 
\item (retracts) This is Proposition \ref{wfeoifjowefwefewfw}.
\item ($2$ out of $3$) This is Lemma \ref{fwiowowfefwefwef334}.
\item (lifting) This is Proposition \ref{fiuehfieufwfewfwef} together with Corollary \ref{wfeiweiofewfewfewf} and Proposition \ref{fweiowefwefewffewf}.
\item (factorization) This is shown in Proposition  \ref{regfiowfoeweu98uz48935}.
 \end{enumerate}

In Definition \ref{fkjhifhiueiwhuiwhfiuewefewfe} we have characterized fibrations by
the lifting property. In the following we explicitly define  a set of morphisms called good morphisms for the moment. Later in Proposition \ref{fiuehfieufwfewfwef} it will turn out that these are exactly the fibrations.

 We consider a morphism $a:\bC\to \bD$   in a category $\cC$ belonging to the list  \eqref{vrelkkjn43joioij434}. 
\begin{ddd}\label{wfiojowefewfewfew}
The morphism $a $ is called good\footnote{The analog of this notion  in category theory is called an isofibration. So we could call these morphisms unitary or marked isofibrations, but these names are longer.}, if for every object $d$ of $ \bD$ and  unitary (marked)  morphism $u:a(c)\to d$ for some object $c$ of $\bC$ there exists a   unitary (marked)  morphism $v:c\to c^{\prime}$ such that $a(v)=u$.
\end{ddd}
Here the word \emph{marked} only applies in the marked cases.

  Let $\Delta^{0}$ in $ \cC$ be the object classifier (Definition \ref{fiojoiwfefwefewfew}) and $\bbI$ be the classifier   of invertibles  in $\Cat_{1}$ (Definition \ref{wffiweoffewfewfewf}). Note that $\bbI$ is a groupoid.

\begin{remark}
In the unmarked case we have an isomorphism
     $\beins\cong \Delta^{0}\sharp \bbI$, where $\beins$ is the unitary morphism classifier (Definition \ref{fewiojewofwefewfewf}). 
     
     In the marked case we have  $\beins^{+}\cong \Delta^{0}\sharp \bbI$,
where $\beins^{+}$ (Definition \ref{groiegegergegererer}) classifies the marked morphisms.     \hB
 \end{remark}

     Let $\Delta^{0}\to \Delta^{0}\sharp \bbI$ classify the object $0$. 
We consider a morphism $a:\bC\to \bD$  as above.
\begin{lem}\label{fweiojweoiffewfwefwef}
 The morphism $a$ is good if and only if it has the right lifting property with respect to
$$\Delta^{0} \to  \Delta^{0}\sharp \bbI\ .$$ 
\end{lem}
 \begin{proof}
In view of the universal properties of $\Delta^{0}$ and $ \Delta^{0}\sharp \bbI$ this is just a reformulation of  Definition \ref{wfiojowefewfewfew}.
\end{proof}

 \begin{prop}\label{fiuehfieufwfewfwef}
 The good morphisms in $\cC$ have the  right lifting property with respect to trivial cofibrations.
 \end{prop}
\begin{proof}
We consider a diagram   \begin{equation}\label{juihfuifwfuiweiuweffwef}
  \xymatrix{\bA\ar[r]^{\alpha}\ar[d]^{i}&\bC\ar[d]^{f}\\\bB\ar@{.>}[ur]^{\ell}\ar[r]^{\beta}&\bD}   \end{equation} 
where $f$ is good and $i$ is a trivial cofibration.
We can find a morphism $j:\bB\to \bA$ such that $j\circ i=\id_{\bA}$  and   there is a   unitary (marked) equivalence $u:i\circ j\to \id_{\bB}$  which in addition satisfies $u\circ i=\id_{i}$.

On objects we define $\ell$ as follows: If $b$ is an object of $\bB$ such that $b=i(a)$ for some object $a$ of $\bA$, then we set $\ell(b):=\alpha(a)$. This makes the upper triangle commute.
If $b$ is not in the image of $i$, then   we get a (marked) unitary $\beta(u_{b}):f(\alpha(j(b))=\beta(i(j(b)))\to  \beta(b)$. Using that $f$ is good we choose a  (marked)  unitary $v:\alpha(j(b))\to c$ such that $f(v)=\beta(u_{b})$. We then set $\ell(b):=c$. This makes the lower triangle commute.
\footnote{Note that the argument in \cite[Lemma 4.10]{DellAmbrogio:2010aa} contains a mistake at this point. With the definition given there the lower triangle would not commute on the level of objects}

We now define the lift $\ell$ on a morphism $\phi:b\to b^{\prime}$. We distinguish four cases:
\begin{enumerate}
\item 
If $b$ and $b^{\prime}$ are in the image of $i$, then (since $i$ is an equivalence) there exists a unique morphism $\psi $ in $\bA$ such that $i(\psi)=\phi$ and we set $$\ell(\phi):=\alpha(\psi)\ .$$ This again makes the upper
triangle commute. 
\item  If $b=i(a)$ and $b^{\prime}$ is not in the image of $i$, then we  let $v^{\prime}$ and $c^{\prime}$ be the choices as above made for $b^{\prime}$. In this case we set
$$\ell(\phi)=v^{\prime}\circ \alpha(j(\phi))\ .$$  \item Similarly, if $b^{\prime}=i(a^{\prime})$ and $b$ is not in the image of $i$, then we set
$$\ell(\phi):=\alpha(j(\phi))\circ v^{-1}\ .$$
\item Finally, if both $b$ and $b^{\prime}$ do not belong to the image, then we set
$$\ell(\phi):=v^{\prime}\circ \alpha(j(\phi))\circ v^{-1}\ .$$
\end{enumerate}

One can check that then the lower triangle commutes and that this really defines a functor.
One further checks (using that the morphisms $v$,$v^{\prime}$ are (marked) unitaries) that $\ell$ is a morphism of (marked) $*$-categories. Finally, if $\cC$ is one of the $\C$-vector space enriched cases, then $\ell$ is a functor between  (marked) $\C$-linear $*$-categories. 
\end{proof}

\begin{kor}\label{wfeiweiofewfewfewf}
The sets of good morphisms and the fibrations coincide.
\end{kor}
\begin{proof}
Since $\Delta^{0} \to  \Delta^{0}\sharp \bbI$ is a trivial cofibration, by Lemma \ref{fweiojweoiffewfwefwef} the fibrations are contained in the good morphisms. By Proposition \ref{fiuehfieufwfewfwef}
every good morphism is a fibration.
\end{proof}
\begin{prop}\label{fweiowefwefewffewf}
The cofibrations in $\cC$ have the left-lifting property with respect to the good morphisms which are in addition weak equivalences.
\end{prop}
\begin{proof}
We again consider a diagram \eqref{juihfuifwfuiweiuweffwef}.
Since the map $i$ is injective on objects and the morphism $f$ is surjective on objects
we can find a lift $\ell$ on the level objects. Let now $b,b^{\prime}$ be objects in $\bB$. 
Since $f$ is fully faithful (see Lemma \ref{lem:markedequivs})  we have a bijection 
 $$\Hom_{\bC}(\ell(b),\ell(b^{\prime})) \stackrel{f,\cong }{\to}\Hom_{\bD}(\beta(b),\beta(b^{\prime}) )  \ .$$
We can therefore define $\ell$ on   $\Hom_{\bB}(b,b^{\prime})$ by
$$\Hom_{\bB}(b,b^{\prime})\stackrel{\beta}{\to}  \Hom_{\bD}(\beta(b),\beta(b^{\prime}))\cong \Hom_{\bC}(\ell(b),\ell(b^{\prime}))\ .$$
The lower  triangle commutes by construction. One can furthermore check that the upper triangle commutes. Finally one checks that this really defines a functor.  
Since $f$   detects  marked morphisms the functor $\ell$ preserves them.
One  now checks that the functor $\ell$   is a morphism between (marked) $*$-categories. If $\cC$ is one of the $\C$-vector space enriched cases, then obviously $\ell$ is enriched, too. \end{proof}

\begin{lem}\label{fwiowowfefwefwef334}
The weak equivalences in $\cC$ satisfy the two-out-of three axiom. 
\end{lem}
\begin{proof}
It is clear that the composition of weak equivalences is a weak equivalence. 
Assume that $f:\bA\to \bB$ and $g:\bB\to \bC$ are morphisms such that $f$ and $g\circ f$ are weak equivalences. Then we must show that $g$ is a weak equivalence. Let
$m:\bB\to \bA$ and $n:\bC\to \bA$   inverse functors and
$u:m\circ f\to \id_{\bA}$ and $v:f\circ m\to \id_{\bB}$ and
$x:n\circ g\circ f\to \id_{\bA}$ and $y:g\circ f\circ n\to \id_{\bC}$ the corresponding unitary (marked)
isomorphisms. $$\xymatrix{&\bA\ar[d]^{f}&u:m f\to \id_{\bA} &x:n  g  f\to \id_{\bA}\\&\bB\ar[d]^{g} \ar@{-->}@/_-1.3cm/[u]_{m}&v:f  m\to \id_{\bB}  &\\&\bC \ar@{-->}@/_-3cm/[uu]_{n} \ar@{..>}@/_-1.3cm/[u]_{h=f\circ n}&&y:g  f  n\to \id_{\bC}}$$ Then we consider the functor
$h:=f\circ n:\bC\to \bB$. We have unitary (marked) isomorphisms
$$h\circ g=f\circ n\circ g\stackrel{ v^{-1} }{\to} f\circ n\circ g\circ f\circ m\stackrel{x}{\to}  f\circ m\stackrel{v}{\to} \id_{\bB}\ .$$
and
$$g\circ h= g\circ f\circ n\stackrel{y}{\to} \id_{\bC}\ .$$ \end{proof}

\begin{prop}\label{wfeoifjowefwefewfw}
The cofibrations, fibrations and weak equivalences are closed under retracts.
\end{prop}
\begin{proof}
Since fibrations maps are characterized by a right lifting property they are closed under retracts.
Cofibrations are closed under retracts since a retract diagram of categories induces a retract diagram on the level of sets of objects, and injectivity of maps between  sets is closed under retracts.

We finally consider weak equivalences (compare \cite[Lemma 4.9]{DellAmbrogio:2010aa}).
Consider a diagram
$$\xymatrix{\bA\ar[r]^{i}\ar[d]^{f}&\bA^{\prime}\ar[r]^{p}\ar[d]^{f^{\prime}}&\bA\ar[d]^{f}\\\bB\ar[r]^{j}&\bB^{\prime}\ar[r]^{q}&\bB}$$
with $p\circ i=\id_{\bA}$ and $q\circ j=\id_{\bB}$, and where $f^{\prime}$ is a weak equivalence. Let $g^{\prime}:\bB^{\prime}\to \bA^{\prime}$ be an inverse of $f^{\prime}$ up to unitary (marked) isomorphism.
Then $p\circ g^{\prime}\circ j:\bB\to \bA$ is an inverse of $f$ up to unitary (marked) isomorphism. \end{proof}

\begin{prop}\label{regfiowfoeweu98uz48935}
In the category $\cC$ we have  functorial factorizations.  
\end{prop}
\begin{proof}
We use a functorial cylinder object in order to factorize a morphism
as $$trivial\:  fibration\:\:\circ\:\: cofibration\ .$$
We use the notation Convention \ref{fwerifhiwefewfewf}.

For a morphism $a:\bA\to \bB$ in $\cC$ we define the cylinder object
$Z(a)$ as the push-out
$$\xymatrix{\bA\ar[r]^{a}\ar[d]^{(1)} &\bB\ar[d]^{\beta}\\\bA\sharp \bbI\ar[r]^{a^{\prime}}&Z(a) }\ ,$$
where $(1)$ is induced by the inclusion of the object $1$ in $\bbI$.
We have a morphism $\pr:\bbI\to pt$. 
Since $a\circ \pr\circ (1)=\id_{B}\circ a$, using the universal property of the push-out we can extend the diagram to 
$$\xymatrix{\bA\ar[r]^{a}\ar[d]^{(1)} &\bB\ar@/^1cm/[ddr]^{\id_{\bB}}\ar[d]^{\beta}&\\\bA\sharp \bbI\ar@/_1cm/[rrd]_{a\circ \pr}\ar[r]^{a^{\prime}}&Z(a) \ar@{..>}[dr]^{q}&\\&&\bB}\ .$$
We finally extend the diagram as follows
$$
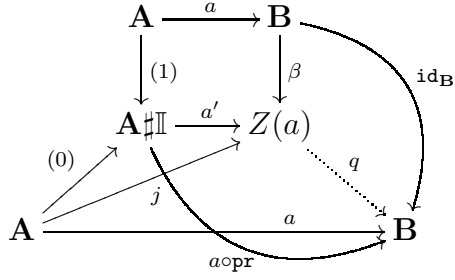
\xymatrix{&\bA\ar[r]^{a}\ar[d]^{(1)} &\bB\ar@/^1cm/[ddr]^{\id_{\bB}}\ar[d]^{\beta}&\\&\bA\sharp \bbI\ar@/_1cm/[rrd]_{a\circ \pr}\ar[r]^{a^{\prime}}&Z(a) \ar@{..>}[dr]^{q}&\\\bA\ar[ur]^{(0)}\ar[urr]^(0.38){j}\ar[rrr]^(0.7){a}&&&\bB}\ ,$$
using that $a\circ \pr\circ (0)=a$ and setting $j:= a^{\prime}\circ (0)$.

We claim that $j$ is a  cofibration and $q$ is a trivial fibration. In order to see these properties it is useful to calculate an explicit model $\tilde Z(a)$ for $Z(a)$.

We define $\tilde Z(a)$ as follows:
 \begin{enumerate}
 \item
 $\Ob(\tilde Z(a)):=\Ob(\bA)\sqcup \Ob(\bB)$.
 \item $\Hom_{\tilde Z(a)}(x,y):=\left\{\begin{array}{cc}\Hom_{\bA}(x,y)&x,y\in \bA\\\Hom_{\bB}(a(x),y)&x\in \bA,y\in \bB\\\Hom_{\bB}(x,a(y))&x\in \bB,y\in \bA\\
 \Hom_{\bB}(x,y)&x,y\in \bB
 \end{array}\right.$.
 \item The composition is defined in the only possible way.
\item The $*$-operation is induced by the $*$-operations on $\bA$ and $\bB$ in the canonical way.
 \item The $\C$ enrichement of $\bA$ and $\bB$ induces an enrichment of $\tilde Z(a)$.
 \item In the marked cases we mark all morphisms which are marked in $\bA$ or $\bB$.
 \end{enumerate}

 We have defined $\tilde Z(a)$ as an object of $\scat^{(+)}_{1}$ or $\ClinCat_{1}^{(+)}$ in the $\C$-enriched cases. In the case of (marked) pre-$C^{*}$-categories, if we can identify $\tilde Z(a)$ with $Z(a)$ as a (marked) $\C$-linear $*$-category, then we can conclude
 that it is itself a (marked) pre-$C^{*}$-category.
Here we use that   the inclusion of  (marked) pre-$C^{*}$-categories into  (marked)  $\C$-linear $*$-categories   preserves colimits.

 Furthermore we see by an inspection of the definition that $\tilde Z(a)$ is a (marked) $C^{*}$-category if $\bA$ and $\bB$ are (marked) $C^{*}$-categories.

  We have a canonical morphism $\bB\to \tilde Z(a)$. Furthermore we have a morphism 
 $\bA\sharp \bbI\to \tilde Z(a)$ given by
 \begin{enumerate}
 \item $(x,0)\mapsto x$
 \item $(x,1)\mapsto a(x)$
 \end{enumerate}
 for $x$ an object of $\bA$, and which is fixed on morphisms by
 \begin{enumerate}
 \item $((f,\id_{0}):(x,0)\to (y,0))\mapsto f$
 \item $(\id_{x},(0\to 1))\to \id_{a(x)}$
 \end{enumerate}
 and the compatibility with composition and the $*$-operation.
 With these definitions the square in 
$$\xymatrix{\bA\ar[r]^{a} \ar[d]^{(1)}&\bB\ar[ddr]^{\psi}\ar[d]&\\\bA\sharp \bbI \ar[drr]^{\phi}\ar[r]&\tilde Z(a)\ar@{.>}[dr]&\\&&\bD}$$
commutes. Let now $\phi$ and $\psi$ be given as indicated.
Then we define a  morphism
$\tilde Z(a)\to \bD$ on objects by  
$$x\mapsto \left\{\begin{array}{cc} \phi(x,0)&x\in \bA\\ \psi(y)&y\in \bB \end{array} \right.\ ,$$
and on morphisms by
$$f\mapsto  \left\{\begin{array}{cc} \phi(f,\id_{0})&x,y\in \bA\\  \psi(f) &else 
 \end{array}\right.$$
In fact this morphism is uniquely determined by the commutativity of the diagram.
This implies that $\tilde Z(a)$ is an explicit model for the push-out and hence a model for $Z(a)$.

From now on we assume that $Z(a)=\tilde Z(a)$.
In this model the morphism $q: Z(a)\to \bB$ is given by 
\begin{enumerate}
\item $q(x)= \left\{\begin{array}{cc} a(x)&x\in \bA\\ x&x\in \bB  \end{array} \right.$
\item $q(f:x\to y)=\left\{\begin{array}{cc}a(f) &x,y\in \bA\\f&else
\end{array}\right.$
\end{enumerate}
It is surjective on objects.  In order to see that $q$ is a weak equivalence we define an inverse $p:\bB\to Z(a)$ by
$$p(x):=x\ , \quad p(f):=f\ .$$
where both take values in the $\bB$-component. Then $q\circ p=\id_{B}$. Furthermore, a (marked) unitary isomorphism  $u:p\circ q\to \id_{Z(a)}$ is given by
$u_{x}=\id_{x}$ for $x$ in $\bB$ and $\id_{a(x)}$ for $x$ in $\bA$.
It follows that $q$ is good and a weak equivalence, hence a trivial fibration.

 The morphism $j:\bA\to Z(a)$ is the canonical embedding and clearly a cofibration. 
We therefore have constructed a  functorial factorization
$$(a:\bA\to \bB)\mapsto (\bA\stackrel{j}{\to} Z(a)\stackrel{q}{\to} \bB)\ .$$

We will use a functorial path object to obtain a functorial factorization of morphisms as
$$ fibration\:\:\circ \:\:trivial\:   cofibration\ .$$ We again use the notation Convention \ref{fwerifhiwefewfewf}.

For a morphism $a:\bA\to \bB$ we define $P(a)$ as the pull-back $$\xymatrix{  P(a)\ar[r]^{\alpha}\ar[d]^{a^{\prime}} &\bA\ar[d]^{a}\\ \cFun^{?}(\bbI,\bB)\ar[r]^(0.6){(1)^{*}}  &\bB\\ }\ .$$
Using the universal property of the pull-back we get an extension of the diagram to 
$$\xymatrix{\bA\ar@{..>}[dr]^{j}\ar@/_1cm/[ddr]_{\const\circ a}\ar@/^1cm/[drr]^{\id_\bA}&&\\&P(a)\ar[r]^{\alpha}\ar[d]^{a^{\prime}} &\bA\ar[d]^{a}\\&\cFun^{?}(\bbI,\bB)\ar[r]^(0.6){(1)^{*}}\ &\bB }\ .$$
We finally extend the diagram as follows
$$\xymatrix{\bA\ar@{..>}[dr]^{j}\ar@/_1cm/[ddr]_{\const\circ a}\ar@/^1cm/[drr]^{\id_\bA}&&\\&P(a)\ar[r]^{\alpha}\ar[d]^{a^{\prime}}\ar@/^4cm/[ddr]^{p}&\bA\ar[d]^{a}\\&\cFun^{?}(\bbI,\bB)\ar[r]^(0.6){(1)^{*}}\ar[dr]^{(0)^{*}}&\bB\\&&\bB}$$ by setting $p:=(0)^{*}\circ a^{\prime}$

The morphism $j$ is a   cofibration since it is injective on objects because of  
$\alpha\circ j=\id_{\bA}$.

We can describe $P(a)$ explicitly as the subcategory of
$\cFun^{?}(\bbI,\bB)\times \bA$ determined on objects $(\phi,x)$ by the condition
$\phi(1)=a(x)$  and on morphisms by $(u,f)$ by
$u(1)=a(f)$. In this picture $j$ is given by
$$j(x):=(\const(a(x)),x)\ , \quad  j(f)=(\const(a(f)),f)\ .$$
Note that $\alpha\circ j=\id_{\bA}$ by construction.
We furthermore find a (marked) unitary isomorphism
$\id_{P(a)} \to j\circ \alpha $ by
$$u(\phi,x):=  ((0\mapsto \phi(0\to 1),1\mapsto \id_{a(x)}),\id_{x}):(\phi,x)\to (\const(a(x)),x)\ .  $$
This shows that $j$ is a weak equivalence. Hence $j$ is a trivial cofibration.

It remains to show that $p$ is a fibration. By Corollary \ref{wfeiweiofewfewfewf} it suffices to show that $p$ is good.
Let a (marked) unitary morphism $u:p(\phi,x)\to b$ be given. Then we take
$$c:=((0\mapsto b, 1\mapsto a(x);(0\to 1)\mapsto u^{-1}), x)$$ and define
$v:(\phi,x)\to c$ by
$((0\mapsto u,1\mapsto \id_{a(x)}),\id_{x})$.
Then $p(v)=u$. This shows that $p$ is good. 
\end{proof}

\section{The simplicial axioms}\label{riowefwefewfew}

We assume that the category $\cC$ belongs to the list $$ \{\scat_{1},\ClinCat_{1},\preCcat_{1},\Ccat_{1},\scat_{1}^{+},\ClinCat_{1}^{+},\preCcat_{1}^{+},\Ccat_{1}^{+}\}\ .  $$
In this section we verify that the model category structure on $\cC$ described in Definition \ref{fkjhifhiueiwhuiwhfiuewefewfe} and with the simplicial structure introduced in Definition \ref{riooejrgegerreg} is a simplicial model category \cite[Def. 9.1.6]{MR1944041}. Note that the axiom M6
\cite[Def. 9.1.6]{MR1944041} is satisfied, in view of the bijections \eqref{lrwrknewkjfewkjfwefwfwef}, by the construction of the simplicial structure 
Definition \ref{riooejrgegerreg}. So it remains to verify the axiom  M7 \cite[Def. 9.1.6]{MR1944041}.
This follows from Proposition \ref{rhiuhwriugwgwefwefewfwe}  showing the dual version of M7 (as stated in \cite[Def. 9.1.6]{MR1944041}), and the validity of M6.

 We closely follow the argument given in \cite{DellAmbrogio:2010aa} for $C^{*}$-categories.

 \begin{lem} \label{reiofweiofweewf}
For $\bA$ in $ \cC$ the functor
$$\bA\sharp -:\sSet\to \cC$$
preserves (trivial) cofibrations.
\end{lem}
\begin{proof}
Recall that this functor is defined in Definition \ref{riooejrgegerreg}  as the composition
$$\sSet\stackrel{\Pi}{\to} \Groupoids_{1}\stackrel{\bA\sharp-}{\to} \cC\ .$$
In the following proof it is useful not to drop $\Pi$ from the notation.
If $i:X\to Y$ is a cofibration of simplicial sets, then $i$ is injective on zero simplices, and hence, by the explicit description   of  the functor $\Pi$ given below the Definition \ref{fewiohfwiof2323r23r},  the morphism of groupoids $\Pi(i)$ is injective on objects. This implies that
$\bA\sharp \Pi(i)$ is injective on objects.

Assume now that $i $ is  in addition  weak equivalence. Then
$\Pi(i)$ is an equivalence of groupoids. Let $j:\Pi(Y)\to \Pi(X)$ be an inverse equivalence and
$u:j\circ \Pi(i)\to \id_{\Pi(X)}$ and $v:\Pi(i)\circ j\to \id_{\Pi(Y)}$ be the corresponding  isomorphisms.
Then we get a (marked) unitary isomorphism
$$\bA\sharp u:(\bA\sharp j)\circ (\bA\sharp \Pi(i))\to \id_{\bA\sharp \Pi(X)}$$
by $(\bA\sharp u)_{(a,x)}:=(\id_{a},u_{x})$. 
Similarly, we have a (marked) unitary  isomorphism
$$\bA\sharp v: (\bA\sharp \Pi(i))\circ (\bA\sharp j) \to \id_{\bA\sharp \Pi(Y)}$$
given by $(\bA\sharp v)_{(a,x)}:=(\id_{a},v_{x})$.
 \end{proof}

\begin{lem}\label{efwiuhfuihfewihfweiufhewfewfewfewfwefwf}
For a groupoid $\bG$ the functor
$$-\sharp \bG:\cC\to \cC$$
preserves (trivial) cofibrations.
\end{lem}
\begin{proof}
If $a:\bA\to \bB$ is a cofibration, then it is injective on objects. But then
$a\sharp \bG$ is injective on objects and hence a cofibration.
If $a$ is in addition a (marked) unitary equivalence, then $a\sharp \bG$ is a (marked) unitary  equivalence, too. The argument is similar to 
the corresponding part of the argument  in the proof of Lemma \ref{reiofweiofweewf}.
\end{proof}

We onsider a commutative square \begin{equation}\label{2oihvwvwewv}
\xymatrix{
			\bA\ar[r]^{i}\ar[d]_{f} & \bB\ar[d]^{g} \\
			\bC\ar[r]^{j} & \bD
		}
\end{equation}
		in $\cC$.
\begin{lem}\label{lem:push.cofib.pull.fib}
	  If \eqref{2oihvwvwewv}
		is a pushout and $i$ is a trivial cofibration, then $j$ is a trivial cofibration.
		\end{lem}
 \begin{proof}  	 	Since $i$ is a trivial cofibration, there exists a morphism $i' \colon \bB \to \bA$ such that $i' \circ i = \id_\bA$ and a (marked) unitary isomorphism $u \colon i \circ i' \to \id_\bB$ satisfying $u \circ i = \id_i$.
	By the universal property of the push-out,  the morphism $f \circ i' \colon \bB \to \bC$ induces a morphism $j' \colon \bD \to \bC$ such that $j' \circ j = \id_\bC$.
	In particular, $j$ is a cofibration.
	
	The functor  $-   \sharp \bbI:\cC \to \cC $
	 is a left-adjoint by Proposition \ref{efiheiufhwiufwfwe} and therefore preserves pushouts.
	Moreover, $g \circ u$ provides a (marked) unitary isomorphism $j \circ f \circ i' = g \circ i \circ i' \to g$.
	
	Using Example \ref{griuhuiwfwefwefef} we consider  the  (marked)  unitary isomorphism   $g \circ u$ between  functors from $\bB$ to $\bD$ as a    morphism $\bB  \sharp \bbI \to \bD$. Together with the  morphism $\bC\sharp \bbI\to \bD$ corresponding   $\id_{j}$,    by the universal property of  the push-out diagram $ \eqref{2oihvwvwewv} \sharp \bbI$ we obtain an induced morphism $\bD \sharp \bbI \to \bD$ which we can interpret as   
	a (marked) unitary isomorphism $j \circ j' \to \id_\bD$. This proves that $j$ is a weak equivalence.
	 \end{proof}

We can now verify the simplicial axiom M7.

 \begin{prop}\label{rhiuhwriugwgwefwefewfwe} Let $a:\bA \to \bB$ be a cofibration in $\cC$ and
$i:X\to Y$ be a cofibration in $\sSet$.
 Then  \begin{equation}\label{dquihiduqwdqwd}
(\bA\sharp  Y )\sqcup_{\bA\sharp X } (\bB\sharp  X )\to \bB\sharp  Y 
 \end{equation}
 is a cofibration. Moreover, if $i$ or $a$ is in addition a weak equivalence, then
 \eqref{dquihiduqwdqwd} is a weak equivalence.
\end{prop}
\begin{proof}

 The set objects of the push-out one left-hand side of \eqref{dquihiduqwdqwd} is equal to  the push-out of the object sets. We write objects in $\bA\sharp  X $ as pairs $(\alpha,x)$.
 
 Assume that the classes of  $(\alpha,y)$ and $(\beta,x)$ in the push-out go to the same object which is then  $(\beta,y)$.
Then $a(\alpha)=\beta$ and $i(x)=y$. This means that $(\alpha,y)=(\id,i)(\alpha,x)$ and
$(\beta,x)=(a,\id)(\alpha,x)$. Consequently,  the classes of $(\alpha,y)$ and $(\beta,x)$ in $(\bA\sharp  Y )\sqcup_{\bA\sharp X } (\bB\sharp  X )$ coincide. 

Assume now that the classes of $(\alpha,y)$ and $(\alpha^{\prime},y^{\prime})$ go to the same object which is necessarily $(a(\alpha),y)$. Then $\alpha=\alpha^{\prime}$ and $ y=y^{\prime}$.

Similarly, if the classes of 
$(\beta,x)$ and $(\beta^{\prime},x^{\prime})$ go to the same object which is necessarily 
$(\beta,i(x))$, then $\beta=\beta^{\prime}$ and $x=x^{\prime}$.

This shows that the morphism marked by $?$   the extended diagram 
$$\xymatrix{\bA\sharp  X \ar[rr]\ar[dd]^{c}&&\bA\sharp  Y \ar[dd]^{e}\ar[dl]^{d}\\&(\bA\sharp  Y )\sqcup_{\bA\sharp  X } (\bB\sharp  X )\ar[dr]^{?}&\\\bB\sharp X\ar[rr]\ar[ur]&&\bB\sharp  Y}\ .$$  is injective on objects and hence a cofibration.

Assume that $a$ is a weak equivalence.
By Lemma \ref{efwiuhfuihfewihfweiufhewfewfewfewfwefwf} the map $c$ is a trivial cofibration. 
By Lemma \ref{lem:push.cofib.pull.fib}
  the morphism $d$ is again a trivial cofibration.  Since (again  
  by Lemma  \ref{efwiuhfuihfewihfweiufhewfewfewfewfwefwf}) the morphism $e$ is a trivial cofibration  it follows from the two-out-of-three property for weak equivalences verified in Lemma \ref{fwiowowfefwefwef334}  that
the morphism $?$ is a weak equivalence.

The case that $i$ is a weak equivalence is  similar using Lemma \ref{reiofweiofweewf} for the horizontal arrows.
\end{proof}

\section{Cofibrant generation}\label{uiehfiwefwefewfewfwf}

Let $\cC$ be a member of the list $$\{\scat_{1},\ClinCat_{1}, \Ccat_{1},\scat_{1}^{+},\ClinCat_{1}^{+}, 
\Ccat_{1}^{+}\}\ .$$  In this section we show that the model category structure on $\cC$ described in Definition \ref{fkjhifhiueiwhuiwhfiuewefewfe} is cofibrantly generated.  We 
  adapt the arguments given in \cite[Sec. 4.1]{DellAmbrogio:2010aa}.

Recall from Section \ref{fiowufoefewfewfewfwf} that $\Delta^{0}$  denotes the object classifyer object in $\cC$, and that
the groupoid  $\bbI$ denotes the isomorphism classifier object in $\Cat$.
 The  morphism
$\Delta^{0}\to \Delta^{0}\sharp \bbI$ classifying the object $0$ is a trivial cofibration since it is clearly injective on objects and a (marked) unitary equivalence.
So by Lemma \ref{fweiojweoiffewfwefwef} and Corollary  \ref{wfeiweiofewfewfewf} we can take $$J:=\{\Delta^{0}\to\Delta^{0}\sharp \bbI \}$$
 as the set of generating trivial cofibrations.

 We now define the set $I$ of generating cofibrations. We must distinguish various cases and the set $I$ will depend on the case:
 
 \begin{table}[htp]
\caption{generating cofibrations}
\begin{center}
\begin{tabular}{|c||c|}\hline
case&$I$\\\hline $\scat_{1}$,$\ClinCat_{1}$&$J\cup \{U,V,V^{u},W,W^{u}\}$\\\hline 
  
 $\Ccat_{1}$&$J\cup\{U,V^{\mathrm{bd}},V^{u},W^{\mathrm{bd}},W^{u}\}$\\\hline 
 $\scat_{1}^{+}$,$\ClinCat_{1}^{+}$&$J\cup\{U,V,V^{+},W,W^{+}\}$\\\hline 
  
 $\Ccat_{1}^{+}$&$J\cup\{U,V^{\mathrm{bd}},V^{+},W^{\mathrm{bd}},W^{+}\}$\\\hline 
\end{tabular}
\end{center}
\label{fwfwefew23rgg3g34f3}
\end{table} 

In the following we describe the details. 
 We first assume that $\cC$ belongs to the list $\{\scat_{1},\ClinCat_{1}\}$.
 Then we consider the cofibrations  
  $U,V,W$  defined as follows:
 \begin{enumerate}
 \item $U:\emptyset\to \Delta^{0}$.
 \item  
 We let $V:\Delta^{0}\sqcup \Delta^{0}\to \Delta^{1}$ classify the pair of objects $(0,1)$ of the morphism classifier $\Delta^{1}$, see Definition \ref{gergjeiogjogergergregergergergee}.    \item We define $P$ by the push-out
   $$\xymatrix{\Delta^{0}\sqcup\Delta^{0}\ar[r]^{V}\ar[d]^{V}&\Delta^{1} \ar[d]\\\Delta^{1} \ar[r]&P}$$
   and let $W:P\to\Delta^{1} $ be the  map induced by $\id_{\Delta^{1} }$ and the universal property of the push-out. \item   We let $V^{u}:\Delta^{0}\sqcup \Delta^{0}\to \beins$ classify the pair of objects $(0,1)$ of  the unitary morphism classifyer   $\beins$, see Definition \ref{fewiojewofwefewfewf}. 
 \item We define $P^{u}$ by the push-out
   $$\xymatrix{\Delta^{0}\sqcup\Delta^{0}\ar[r]^{V^{u}}\ar[d]^{V^{u}}&\beins\ar[d]\\\beins\ar[r]&P^{u}}$$
   and let $W^{u}:P^{u}\to\beins$  be the  map induced by $\id_{\beins}$ and the universal property of the push-out. 

  \end{enumerate}
We set $$I:=J\cup \{U,V,V^{u},W,W^{u}\}\ .$$

We now assume that 
 that $\cC$ belongs to the list $\{\scat^{+}_{1},\ClinCat^{+}_{1}\}$. Then  we consider  the following cofibrations:  
 \begin{enumerate}
  \item  We let $V^{+}:\Delta^{0}\sqcup \Delta^{0}\to \beins^{+}$ classify the pair of objects $(0,1)$ of  the marked morphism classifyer   $\beins^{+}$, see Definition \ref{groiegegergegererer}. 
 \item We define $P^{+}$ by the push-out
   $$\xymatrix{\Delta^{0}\sqcup\Delta^{0}\ar[r]^{V^{+}}\ar[d]^{V^{+}}&\beins^{+}\ar[d]\\\beins^{+}\ar[r]&P^{+}}$$
   and let $W^{+}:P^{+}\to\beins^{+}$  be the  map induced by $\id_{\beins^{+}}$ and the universal property of the push-out. 
  \end{enumerate}
   We then set
 $$I:=J\cup  \{U,V,V^{+},W,W^{+}\}\ .$$

We now consider the case that $\cC$ belongs to the list  $\{ 
\Ccat_{1} \}$. In this case we must replace the morphism classifier by the bounded morphism classifier, see Lemma \ref{gigjoer2tergtertet} and
Definition \ref{rgieroge34t34t34t34t4}.  We consider the following cofibrations:

\begin{enumerate}
  \item   We let $V^{\mathrm{bd}}:\Delta^{0}\sqcup \Delta^{0}\to \Delta^{1,\mathrm{bd}}$ classify the pair of objects $(0,1)$ of the bounded morphism classifier $\Delta^{1,\mathrm{bd}}$, see Definition \ref{rgieroge34t34t34t34t4}.    \item We define $P^{\mathrm{bd}}$ as the push-out
   $$\xymatrix{\Delta^{0}\sqcup\Delta^{0}\ar[r]^{V^{\mathrm{bd}}}\ar[d]^{V^{\mathrm{bd}}}&\Delta^{1,\mathrm{bd}} \ar[d]\\\Delta^{1,\mathrm{bd}} \ar[r]&P^{\mathrm{bd}}}$$
   and let $W^{\mathrm{bd}}:P^{\mathrm{bd}}\to\Delta^{1,\mathrm{bd}} $ be the   map induced by $\id_{\Delta^{1,\mathrm{bd}} }$  and the universal property of the push-out. 
  \end{enumerate}
  We set
  $$I:=J\cup \{U,V^{\mathrm{bd}},V^{u},W^{\mathrm{bd}},W^{u}\}\ .$$
  
Finally, in the case that  $\cC$ belongs to the list $\{  \Ccat_{1}^{+}\}$,
 we set
$$I:=J\cup \{U,V^{\mathrm{bd}},V^{+},W^{\mathrm{bd}},W^{+}\}\ .$$

Let $\cC$ be a member of the list $$\{\scat_{1},\ClinCat_{1}, 
\Ccat_{1},\scat_{1}^{+},\ClinCat_{1}^{+}, 
\Ccat_{1}^{+}\}\ .$$
   \begin{lem}
 The trivial  fibrations in $\cC$ are exactly the morphisms which have the right-lifting property with respect to $I$.
 \end{lem}
\begin{proof}
In all cases,  by Lemma \ref{fweiojweoiffewfwefwef} and Corollary  \ref{wfeiweiofewfewfewf},  a morphism $f$ has the right-lifting property with respect to $J$ if an only if it is a fibration. 
So it remains to show that the right-lifting property of $f$ with respect to the remaining morphisms in $I$ is equivalent to the fact that $f$ is a weak equivalence.

We first consider the case where $\cC$ is in  $\{\scat_{1},\ClinCat_{1}\}$.  By Lemma \ref{lem:markedequivs}.\ref{friofj35t9u5gerkjg34t1} it suffices to show that  the right-lifting property of $f$ with respect to $\{U,V,V^{u},W,W^{u}\} $ is equivalent to the property that $\cF_{\mathrm{all}}(f)$ (see \eqref{gvt4kjn4kjrtbrgbrgbrbrb} for notation) and $\ma(f)^{+}$  (see \eqref{wecoihorgrtb} and Remark \eqref{wfiuwehifewfwefwe})  are equivalences of categories. This follows from the following observations:
\begin{enumerate} 
\item The right-lifting property of $f$ with respect to $U$ is equivalent to surjectivity of $f$ on objects. 
\item The right-lifting property of $f$ with respect to $V$ is equivalent to fullness of $\cF_{\mathrm{all}}(f)$.
 \item The right-lifting property of $f$ with respect to $W$ is equivalent to faithfulness of $\cF_{\mathrm{all}}(f)$.
 \item The right-lifting property of $f$ with respect to $V^{u}$ is equivalent to fullness of $ \ma(f)^{+}$.
 \item The right-lifting property of $f$ with respect to $W^{u}$ is equivalent to faithfulness of $ \ma(f)^{+}$.

\end{enumerate}
Indeed, these conditions imply that $\cF_{\mathrm{all}}(f)$ and $\ma(f)^{+}$ are  equivalence of categories. For the converse, if $f$  is a fibration and  $\cF_{\mathrm{all}}(f)$  is an equivalence  of categories, then $f$ is necessarily surjective on objects.
   
We next discuss the case where $\cC$ is in $\{ 
\Ccat_{1} \}$. By Lemma \ref{lem:markedequivs}.\ref{friofj35t9u5gerkjg34t1} 
 it suffices to show that  the right-lifting property of $f$ with respect to $\{U,V^{\mathrm{bd}},V^{u},W^{\mathrm{bd}},W^{u}\}$   is equivalent to the property that $\cF_{\mathrm{all}}(f)$ and $\ma(f)^{+}$ are equivalences of categories.
 This follows from the following observations:
\begin{enumerate} 
\item The right-lifting property of $f$ with respect to $U$ is equivalent to surjectivity of $f$ on objects.
\item The right-lifting property of $f$ with respect to $V^{\mathrm{bd}}$ is equivalent to  the surjectivity of $f$ on the subspaces of the morphisms of maximal norm bounded by $1$. Since 
  a  linear map  between pre-normed vector spaces is surjective if it is so on vectors of norm bounded by $1$ this implies that $f$ is full.
 \item The right-lifting property of $f$ with respect to $W^{\mathrm{bd}}$ is equivalent to the injectivity of the restriction of $f$ to the subspace of morphisms of norm bounded by $1$. This implies that $f$ is faithful. 
  \item The right-lifting property of $f$ with respect to $\{V^{u},W^{u}\}$ is equivalent to fully faithfulness of $ \ma(f)^{+}$.
\end{enumerate}

We now consider the case that 
$\cC$ is in  $\{\scat_{1}^{+},\ClinCat_{1}^{+}\}$.
In view of Lemma \ref{lem:markedequivs}.\ref{friofj35t9u5gerkjg34t}  we must show  that the right-lifting property of $f$ with respect to $\{U,V,V^{+},W,W^{+}\}$ is equivalent to the fact that $\cF_{\mathrm{all}}(f)$ and $f^{+}$ are equivalences of categories.
 We conclude by the following observations.
\begin{enumerate}
\item The right-lifting property of $f$ with respect to
$\{U,V,W\}$ is equivalent to the fact that $\cF_{\mathrm{all}}(f)$ is an equivalence of categories which is surjective on objects.
\item  The right-lifting property of $f$ with respect to $V^{+}$ is equivalent to  fullness of $f^{+}$.
 \item The right-lifting property of $f$ with respect to $W^{+}$ is equivalent to   faithfulness of $f^{+}$.
\end{enumerate}

We finally  consider the case that 
$\cC$ is in  $\{ 
\Ccat^{+}_{1} \}$.
Again by Lemma \ref{lem:markedequivs}.\ref{friofj35t9u5gerkjg34t} we must show  that the right-lifting property of $f$ with respect to $\{U,V^{bf},V^{+},W^{\mathrm{bd}},W^{+}\}$ is equivalent to the fact that $\cF_{\mathrm{all}}(f)$ and $f^{+}$ are equivalences of categories.
This follows from the following two observations already made above:
\begin{enumerate}
\item  The right-lifting property of $f$ with respect to $\{U,V^{\mathrm{bd}},W^{\mathrm{bd}}\}$ is equivalent to the fact that $\cF_{\mathrm{all}}(f)$ is an equivalence of categories which is surjective on objects. 
\item  The right-lifting property of $f$ with respect to $\{U,V^{+},W^{+}\}$ is equivalent to the fact that $f^{+}$ is an equivalence of categories which is surjective on objects.
\end{enumerate}
\end{proof}

Let $\kappa$ be a regular cardinal. A partially ordered set $I$ is called $\kappa$-filtered if every subset of cardinality $< \kappa$ has an upper bound. A $\kappa$-filtered\footnote{We follow the terminology of \cite{htt}. In \cite[Def. 1.13]{AR} the word $\kappa$-directed is used.}  diagram is a diagram indexed by a $\kappa$-filtered partially ordered set.
An object $A$ in a category $\cC$ is called $\kappa$-compact\footnote{We again follow the terminoloy of \cite{htt}. In \cite{AR} the term $\kappa$-presented is used. A $\aleph_{0}$-compact object is also called finitely presented, or just compact.}
if the functor $$\Hom_{\cC}(A,-):\cC\to\Set$$ preserves $\kappa$-filtered colimits.
The object is called small if it is $\kappa$-compact for some regular cardinal $\kappa$.

Let $\cC$ be a member of the list $$\{\scat_{1},\ClinCat_{1}, 
\Ccat_{1},\scat_{1}^{+},\ClinCat_{1}^{+}, 
\Ccat_{1}^{+}\}\ .$$
In the following lemma the classifier objects (and the objects $P,P^{+}$ derived from them) are associated to $\cC$.

\begin{lem}
\begin{enumerate}
\item\label{rgioerge34r344334g} The objects $\emptyset$, $\Delta^{0}$, $\beins$,    $P^{u}$ and  $\beins^{+}$, $P^{+}$ (in the marked cases) are compact (i.e, $\aleph_{0}$-compact). 
\item\label{rgioerge34r344334g2}   The objects $\Delta^{1}$ and $P$ (if defined) are $\aleph_{1}$-compact.
\item \label{rgioerge34r344334g1} The objects $\Delta^{1,\mathrm{bd}}$ and $P^{\mathrm{bd}}$  (if defined)  are $\kappa$-compact, where $\kappa$ is a regular cardinal greater than the maximum of the dimensions of the morphism spaces $\Hom_{\Delta^{1,\mathrm{bd}}}(j,k)$ for $j,k$ in $\{0,1\}$.
 \end{enumerate}
\end{lem}
\begin{proof}
The assertions easily follow  by an inspection of the descriptions of the explicit models for these classifier categories given in Section \ref{fiowufoefewfewfewfwf}. The main observation for \ref{rgioerge34r344334g}. is that the respective categories have finitely many objects and finite-dimensional morphism spaces.
Similarly for \ref{rgioerge34r344334g2}. we use that $\Delta^{1}$ and $P$ have two objects and that their morphism spaces are countable or have countable dimension in the $\C$-linear cases. 
Finally
for \ref{rgioerge34r344334g1}. we use that the categories have the two objects $0,1$.   \end{proof}

\begin{kor}\label{fiowejofwefewfewf}
The model category $\cC$ is cofibrantly generated by finite sets of generating cofibrations and trivial cofibrations between small objects. 
\end{kor}

Recall that cofibrantly generated simplicial model category which is in addition locally presentable is called combinatorial. 
It turns out that  six of our eight examples are combinatorial.

\begin{kor}
The model category structures on the categories $\cC$ in the list $$\{\scat_{1},\ClinCat_{1},\Ccat_{1},\scat_{1}^{+},\ClinCat_{1}^{+},\Ccat_{1}^{+}\}$$  are combinatorial.
\end{kor}
\begin{proof}
Cofibrant generation is shown in  Corollary \ref{fiowejofwefewfewf}. The simplicial structure is discussed in Section \ref{riowefwefewfew}.
Finally, local presentability  is shown in  Proposition \ref{ewdfoijfowefewfewfw} for the list  $\{\scat_{1},\ClinCat_{1}, \scat_{1}^{+},\ClinCat_{1}^{+}, \}$, and in  Proposition \ref{efuifzuew9fewfewfwf}   for the list  $\{ \Ccat_{1}, \Ccat_{1}^{+}\}$.
\end{proof}

\section{The construction $\bA\mapsto \hat \bA^{G}$}\label{oifjweoiffiwwefwf}

Let $G$ be a group. The category  of  $G$-objects in   a category $\cC$  is defined as the functor category 
$\Fun(BG,\cC)$, where  $BG$ is as in Example \ref{fioefjewoifewfewfwfw}.

We now assume that the category $\cC$ belongs to the list $$\{\scat_{1},\ClinCat_{1},\preCcat_{1},\Ccat_{1},\scat_{1}^{+},\ClinCat_{1}^{+},\preCcat_{1}^{+},\Ccat_{1}^{+}\}\ .$$ 
As explained in Section \ref{gtio34fervervevervrv}
one of the purposes of the present paper is to calculate the object $\lim_{BG}\ell_{BG}(\bA)$ in $\cC_{\infty}$ for  $\bA$  in $\Fun(BG,\cC)$, and that
   calculation of the limit amounts more precisely to provide an object $\bB$ of $\cC$ and equivalence $\ell(\bB)\simeq \lim_{BG}\ell_{BG}(\bA)$. In this section we define the candidate for $\bB$ which will be denoted by $\hat \bA^{G}$. 
 We refer to Theorem \ref{fewoijowieffwefwef} for the justification and the actual  calculation of the limit. 
The main point of the present section is the explicit description of $\hat \bA^{G}$ provided in Remark \ref{weffiuewffu9ewue98uf89efe}.

We consider a $G$-object $\bA $ in $\cC$ and let $ \cFun^{?}(\tilde G,\bA)$
be as in Section \ref{oijfowiefjweoifewfewfewfe}.
\begin{ddd}\label{ioweoffewfewf} We define the object  $\hat \bA^{G}$ in $\cC$ by  {
\begin{equation}\label{r4jo3ir34r34r}
\hat \bA^{G}:=\lim_{BG} \cFun^{?}(\tilde G,\bA) 
\ .\end{equation}}
\end{ddd}

  \begin{remark}\label{weffiuewffu9ewue98uf89efe}
  In this remark we derive an 
 an explicit description of $\hat \bA^{G}$. {Note that the limit in  \eqref{r4jo3ir34r34r} is interpreted in $\cC$, and that the details of this interpretation depend on the case.}

We first assume that
$\cC$ belongs to the list \begin{equation}\label{fref234rf23rr23rr2r32}
\{\scat_{1},\ClinCat_{1},  \scat_{1}^{+},\ClinCat_{1}^{+}  \}\ .
\end{equation}
In all these cases the underlying category of $\hat \bA^{G}$ is the category of $G$-invariant functors    in $\cFun^{?}(\tilde G,\bA)$. Hence
an object 
   $a$  of $\hat \bA^{G}$  associates to every object $g$ of the $G$-groupoid $  \tilde G$ an object $a(g)$ in $\bA$. Furthermore, for every pair of objects $g,h$  in $  \tilde G$   we have a (marked) unitary morphism
$a(g\to h):a(g)\to a(h)$ in $\bA$. The condition that the functor $a$ is $G$-invariant means that for every group element $k$ in $G$ we have the equalities $a(kg)=k(a(g))$ and $a(kg\to kh)=k(a(g\to h))$.
Therefore, the object $a$ of $\hat \bA^{G}$ is completely determined by the object $a(1)$ of $\bA$ and a  collection of (marked) unitary morphisms
$\rho(g)=a(1\to g):a(1)\to g(a(1))$ which satisfy the cocycle condition
$$g(\rho(h))\circ \rho(g)=\rho(hg)$$ for all elements $g,h$ of $G$. 
We will therefore write objects of $\hat \bA^{G}$ as pairs $(b,\rho)$ with $b$ in $\bA$ and $\rho=(\rho(g))_{g\in G}$ a cocycle as above with $\rho(g):b\to g(b)$ for all $g$ in $G$.

 Again by $G$-invariance,  a morphism $a\to a^{\prime}$ between   objects of $\hat \bA^{G}$ is determined by its restriction to $a(1)$ which necessarily intertwines the  cocycles, i.e., which satisfies 
 $$\rho^{\prime}(g)\circ f(1)=g(f(1))\circ \rho (g)$$
 for all elements $g$ of $ G$. In other words, a morphism $f:(b,\rho)\to (b^{\prime},\rho^{\prime})$
 is a morphism $f:b\to b^{\prime}$ in $\bA$ such that $\rho^{\prime}(g)\circ f=f\circ \rho(g)$ for all $g$ in $G$.
 We call such  a morphism  an intertwiner.
 
Consequently, the  category
$\hat \bA^{G}$   is isomorphic to the category of pairs $(b,(\rho(g))_{g\in G})$  of an object $b$ of $ \bA$ and a cocycle $\rho$, and  intertwiners.

The $*$-operation on the category  $ \hat \bA^{G} $ is induced by the $*$-operation on $\bA$.
If $\bA$ was a $\C$-linear $*$-category, then so is $\hat \bA^{G} $.
In the marked case, marked morphisms in $\hat \bA^{G} $ are intertwiners which are marked morphisms in $\bA$.
This finishes the description of 
$\hat \bA^{G} $ in the case that $\cC$ belongs to the list \eqref{fref234rf23rr23rr2r32}.

In order to calculate $\hat \bA^{G}$ in the cases where $\cC$ belongs to the list  
$$\{\preCcat_{1},\preCcat_{1}^{+}\}$$
we use the adjunction $(\cF_{\pre},\Bd^{\infty})$ given in Lemma \ref{uhfiuhfiwefewfewfewfewf}.\ref{frefrkjhgiekgergerg1}. We have an isomorphism  \begin{equation}\label{f34foij3fko3lfm34f34f4f}
\lim_{BG}\cong \Bd^{\infty}\circ \lim_{BG}\circ \cF_{\pre}
\end{equation} 
of functors from $\Fun(BG,\cC)$ to $\cC$, where the limit on the right-hand side is interpreted (marked)  $\C$-linear $*$-categories.
Consequently we get an isomorphism \begin{equation}\label{freh3oi4r34rf34fr3}
\hat \bA^{G}\cong \Bd^{\infty}(\widehat{\cF_{\pre}(\bA)}^{G})\ .
\end{equation}
 In other words, if $\bA$ is a (marked) pre-$C^{*}$-category, then we can calculate $\hat \bA^{G}$ by applying the 
$\widehat{(-)}^{G}$-construction to   $\bA$ considered as a (marked) $\C$-linear $*$-category, and then applying the functor $\Bd^{\infty}$.

 Finally we assume that $\cC$ belongs to the list
 $$\{\Ccat_{1},\Ccat_{1}^{+}\}\ .$$  
 In this case we use that the forgetful  functor (see \eqref{ewfwoijioiffewfwefw})
 $$\cF_{-}:\Ccat_{1}^{(+)}\to \preCcat_{1}^{(+)}$$
 preserves limits and reflects isomorphisms.
 We thus have
 $$\hat \bA^{G}\cong \widehat{\cF_{-}(\bA)}^{G}\stackrel{\eqref{freh3oi4r34rf34fr3}}{\cong}  \Bd^{\infty}(\widehat{\cF_{\pre}(\cF_{-}(\bA))}^{G}) \stackrel{!}{\cong}\widehat{\cF_{\pre}(\cF_{-}(\bA))}^{G} \ .$$
  In order to justify the  isomorphism marked by $!$
 note that
  the  (marked) $\C$-linear $*$-category $\widehat{\cF_{\pre}(\cF_{-}(\bA))}^{G}$
 is again a $C^{*}$-category since the space of morphisms  from the object
 $(b,(\rho(g))_{g\in G})$ to the object  $(b^{\prime},(\rho^{\prime}(g))_{g\in G})$ is a closed subset of
 $\Hom_{\bA}(b,b^{\prime})$, and the $C^{*}$-property is induced. Here we use the description of $C^{*}$-categories given in Remark \ref{wreoiwoifwefewf}.

 In other words, if $\bA$ is a (marked) $C^{*}$-category, then we can calculate $\hat \bA^{G}$ by applying the $\widehat{(-)}^{G}$-construction to   $\bA$ considered as a marked $\C$-linear $*$-category, and then noting that the result is in fact   a $C^{*}$-category.

\hB \end{remark}

\section{Infinity-categorical $G$-fixed points}\label{geroijoegergregeg}

Let $\cC$ be a model category and $I$ be a small category.
 For every $i$ in $I$ we have an evaluation functor $e_{i}:\Fun(I,\cC)\to \cC$.

The following definition describes the weak equivalences, cofibrations, and fibrations   of the injective model category structure on $\Fun(I,\cC)$ provided it exists.

\begin{ddd}\label{fkjhifhiueiwhuiwhfiuewefewfe1}\mbox{}
\begin{enumerate}
\item A weak equivalence in $\Fun(I,\cC)$ is a morphism $f$ such that $e_{i}(f)$  is a weak equivalence in $\cC$ for every $i$ in $I$.
\item A cofibration  in $\Fun(I,\cC)$ is a  morphism $f$   such that $e_{i}(f)$ is a   cofibration in $\cC$ for every $i$ in $I$.
\item A fibration is a morphism  in $\Fun(I,\cC)$ which has the right-lifting property with respect to trivial cofibrations.
\end{enumerate}
\end{ddd}
Let  $\cC$ belong  to the list $$\{\scat_{1},\ClinCat_{1}, \preCcat_{1}, \Ccat_{1},\scat_{1}^{+},\ClinCat_{1}^{+}, \preCcat_{1}^{+} ,\Ccat_{1}^{+}\}$$ and $I$ be a small category.
{\begin{theorem}\label{gti34og3g34g3g}
The injective model category structure on $\Fun(I,\cC)$ exists.
\end{theorem}}
\begin{proof}
It is a non-trivial fact that the injective model category structure on a  functor category  $\Fun(I,\cC)$ exists provided that the model category structure  on the target $\cC$  is  combinatorial. The proof involves  Smith's theorem, see e.g. \cite[Thm. 1.7]{beke},  \cite[Sec. A.2.6 ]{htt}. A textbook reference of the fact stated  precisely in the form we need is  \cite[Prop. A.2.8.2]{htt}. 
So in view of the second assertion of Theorem \ref{eifweofwefewfew}   the assertion of the  Theorem follows for $\cC$ in the  list $$\{\scat_{1},\ClinCat_{1},   \Ccat_{1},\scat_{1}^{+},\ClinCat_{1}^{+},  \Ccat_{1}^{+}\}\ .$$ It remains to discuss the case where $\cC$
belongs to the list $$\{\preCcat_{1},\preCcat_{1}^{+}\}\ .$$
  
In this case we employ a result Cisinski.\footnote{I thank D. Ch. Cisinski for explaining   this fact to me.}
\begin{prop}\label{rgkowrthrhethertherth}
Let $\cC$ be a model category and assume in addition:
 \begin{enumerate}
 \item  \label{erigjewogwergwregwereg1}  $\cC$ is right proper.
 \item \label{erigjewogwergwregwereg}  
  The class of cofibrations in $\cC$ is closed under small limits.
  \item  \label{erigjewogwergwregwereg2q}  If $i,g$ are composeable morphisms in $\cC$ and $g\circ i$ is a cofibration, then also $i$ is a cofibration.
  \end{enumerate}
  Then the injective model category structure on $\Fun(I,\cC)$ exists.
\end{prop}
\begin{proof}
The proof is a slight modification   of the  one of \cite[Thm. 6.16]{cisi}.   In the following
  we indicate the necessary modifications using the notation of this reference.
  \begin{enumerate}
  \item[a)] In the reference model categories are only  required to admit finite limits and colimits. Therefore in the statement of  \cite[Thm. 6.16]{cisi}
 the existence of all small limits was included as an additional assumption. In the present paper a model category is complete and cocomplete by assumption so that this requirement is satisfied automatically.
  \item[b)] The assumption of right properness is copied from  \cite[Thm. 6.16]{cisi} and employed in the same way.
\item[c)]  In    \cite[Thm. 6.16]{cisi}  it is assumed that the cofibrations are  exactly the monomorphisms. 
Monomorphisms have a categorical characterization and are preserved by right-adjoints.
This is used in the proof  of  \cite[Thm. 6.16]{cisi} in order to ensure that the functors denoted by $\tau_{A*}$
preserve cofibrations. These functors are right Kan-extension functors. Using the pointwise formula for
the values we observe that Condition \ref{erigjewogwergwregwereg}. above is sufficient to ensure that $\tau_{A*}$ preserves cofibrations.
\item[d)] The conditions  \ref{erigjewogwergwregwereg1}. and  \ref{erigjewogwergwregwereg2q}. are used in the discussion of the last diagram in the proof of  \cite[Thm. 6.16]{cisi} with the goal to show that the arrow $i$ is a trivial cofibration. First of all,  by   \ref{erigjewogwergwregwereg1}.  the upper (unamed, let it call $g$ for our purposes) horizontal arrow in the square is a weak
equivalence, since the arrow $\tau_{A*}k$ is one and $\tau_{A*}q$ is a fibration. Using this it is shown that $i$ is a weak equivalence. So far  the argument is the same as is in the reference.  
Since $\tau_{A*}j=g\circ i$ is a cofibration (see c)), by  \ref{erigjewogwergwregwereg2q}. we can   conclude that $i$ is one. 
  \end{enumerate}
\end{proof}

  We now show that we can apply Proposition \ref{rgkowrthrhethertherth} for $\cC$ in the list $\{\preCcat_{1},\preCcat_{1}^{+}\}$.
  For the first Assumption \ref{erigjewogwergwregwereg1}.  we use the general fact that a model category $\cC$ is right-proper if all its objects a fibrant \cite[Cor. 13.1.3]{MR1944041}. Since  
  every (marked) pre-$C^{*}$-category is fibrant we can conclude that  $\preCcat^{(+)}_{1}$ is right-proper.
  
For Assumption  \ref{erigjewogwergwregwereg}. note  that a limit of a diagram of  injective maps of sets is injective. 
The action  of a limit of a diagram of functors on objects is the limit of the diagram of maps induced on objects. Therefore a limit of cofibrations in  $\preCcat_{1}^{(+)}$  is again a cofibration. 

 Finally, if $i\circ g$ is a composition of  morphisms in  $\preCcat^{(+)}_{1}$ which is a cofibration, then it is injective on objects. Consequently, $i$ is injective on objects and hence a cofibration, too. Hence $\preCcat^{(+)}_{1}$ satisfies  Assumption \ref{erigjewogwergwregwereg2q}.
 \end{proof}

Let $\cC$ belong to the list $$\{\scat_{1},\ClinCat_{1}, \preCcat_{1}, \Ccat_{1},\scat_{1}^{+},\ClinCat_{1}^{+}, \preCcat_{1}^{+} ,\Ccat_{1}^{+}\}\ .$$ We have a functor $\tilde G\to *$ in $G$-categories (see Section \ref{fewoihfiweiofjwefoiewfewfewf} for notation).
It induces a transformation of functors (see Convention \ref{fwerifhiwefewfewf} for notation)  \begin{equation}\label{vevlkjoiijeroiv}
\id\to \cFun^{?}(\tilde G,-):\Fun(BG,\cC)\to \Fun(BG,\cC)\ .
\end{equation}

\begin{prop}\label{wfeioweo9ffwefwf}
The functor
$$\cFun^{?}(\tilde G,-):\Fun(BG,\cC)\to \Fun(BG,\cC)$$
together with the natural transformation \eqref{vevlkjoiijeroiv}
  is a fibrant replacement functor with respect to the injective model structure on $\Fun(BG,\cC)$.
\end{prop}
\begin{proof}
We use Remark \ref{grrigherjiogreg434t3t34t4t34t} stating that
$\cFun^{?}(\tilde G,\bA)\cong \cFun^{u}(\tilde G, \bA)$, where on the right-hand side we consider $\bA$ as a (marked)
($\C$-linear) $*$-category. In this way we avoid a case-dependent discussion.

We must first show that for every object $\bA$ of $\cC$  the transformation \eqref{vevlkjoiijeroiv} induces a weak equivalence
$$r:\bA\to \cFun^{?}(\tilde G,\bA)$$ in $\cC$.
To this end we must find an inverse up to (marked) unitary isomorphism of $r$ on the level of underlying objects in $\cC$.
We define a (non $G$-equivariant) functor
$$
e: \cFun^{?}(\tilde G,\bA)\to \bA\ , \quad  a\mapsto a(1)\ , \quad e(f:a\to a^{\prime}):=(f(1):a(1)\to a^{\prime}(1))\ . $$
Then clearly $e\circ r=\id_{\bA}$.
We furthermore have a (marked) unitary isomorphism $  r\circ e\to \id_{\cFun^{?}(\tilde G,\bA)}$ given on $a  $ in $\cFun^{?}(\tilde G,\bA)$ by the collection of (marked) unitary morphisms $(a(1\to g):a(1)\to a(g))_{g\in G}$.

In order to finish the proof we must show that $ \cFun^{?}(\tilde G,\bA)$ is fibrant.
To this end we consider a square in $\Fun(BG,\cC)$:
$$\xymatrix{\bC\ar[r]\ar[d]^{c}&\cFun^{?}(\tilde G,\bA)\ar[d]\\\bD\ar[r]\ar@{..>}[ur]&{*}}\ ,$$
where $\bC\to \bD$ is a  {trivial cofibration in $\Fun(BG,\cC)$.} We must show the existence of the diagonal lift.

We use the identification $\cFun^{?}(BG,*)\simeq *$ (here $*$ denotes a final object in $\cC$) and  the exponential law Proposition \ref{eoifhewiufewfewfewfewfewf} in order to rewrite the problem as
$$\xymatrix{\bC\sharp \tilde G\ar[r]^{\phi}\ar[d]& \bA\ar[d]\\\bD\sharp \tilde G\ar[r]\ar@/^1cm/@{..>}[u]^{\tilde d}\ar@{..>}[ur] &{*}}\ .$$
Since the underlying morphism of $c:\bC\to \bD$ is a  {trivial cofibration in $\cC$} it is injective on objects.
We choose an inverse equivalence $d:\bD\to \bC$ (not necessecarily $G$-invariant) up to (marked) unitary equivalence which is a precise
inverse on the image of $c$. 
We can extend the composition $$\bD\stackrel{d}{\to} \bC\to \bC\times \{1\}\to \bC\sharp \tilde G$$ uniquely to a $G$-invariant morphism
$$\tilde d:\bD\sharp \tilde G\to \bC\sharp \tilde G\ .$$
Indeed, we set
$$\tilde d(D,g):=(g(d(g^{-1}(D))),g)\ , \quad \tilde d(f:D\to D^{\prime},g\to h):=gd(g^{-1}f)\sharp (g\to h) \ . $$

The desired diagonal can now be obtained as the composition $\phi\circ \tilde d$.
  \end{proof}

  \begin{remark}\label{gerihoiejfiowejhfoiwefwe}
Let  $(\cC,W)$ be a relative category. Then we can consider the localization  \begin{equation}\label{fgfzugfu4f34f3f}
\ell:\cC\to \cC_{\infty}:=\cC[W^{-1}]
\end{equation}    in the realm of $\infty$-categories, see  \eqref{fwefoiu39r32r32r}. For a small  category  $I$ we let 
\begin{equation}\label{fgfzugfu4f34f3f1111}
\ell_{I}:\Fun(I,\cC)\to \Fun(I,\cC_{\infty})
\end{equation}  
denote the functor given by post-composition with   $\ell$ in \eqref{fgfzugfu4f34f3f}.

The content of the following proposition  is  well-known since it provides   the basic justification
that,  in the case of limits,  $\infty$-categories and model categories yields equivalent homotopical constructions.
But since we do not know a reference where it is stated in this ready-to-use form we will give a proof.

Let  $(\cC,W)$ be a relative category and $I$ be a small category.  

\begin{prop}\label{ijwoiffewfewfewfewf}
Assume that $(\cC,W)$  extends to a simplicial  model category with the following properties:
\begin{enumerate}
\item The injective model category structure on $\Fun(I,\cC)$ exists.
\item\label{rgkoergergergreg} All objects of $\cC$ are  cofibrant.
\end{enumerate}
 Then for any fibrant replacement functor $r:\id_{\Fun(I,\cC)} \to R$ in the injective model category structure of $\Fun(I,\cC)$  we have an equivalence of functors 
$$\lim_{I}\circ \ell_{I}\simeq \ell\circ \lim_{I}\circ R:  \Fun(I,\cC)\to  \cC_{\infty}\ .$$
 \end{prop}
\begin{proof} 
Since $(\cC,W)$ extends to a simplicial model category with weak equivalences $W$ and with all objects cofibrant we can express the mapping spaces in $\cC_{\infty}$ in terms of  the simplicial mapping spaces $\Map_{\cC}$ of $\cC$. More precisely, if $A$ is a cofibrant and $A^{\prime}$ is a fibrant object of $\cC$, then by   Remark \ref{roigegergerg} we have an equivalence of spaces \begin{equation}\label{f34kij3lk4grgr}
\Map_{\cC_{\infty}}(\ell(A),\ell(A^{\prime}))\simeq \ell_{\sSet}(\Map_{\cC}(A,A^{\prime}))\ .
\end{equation}

We let $W_{I}$ denote the weak equivalences in the injective model category structure on $\Fun(I,\cC)$. We then have the commuting diagram \begin{equation}\label{hhjopkrpgi03klglkrg4985ug945geer}
\xymatrix{\Fun(I,\cC)\ar[rr]^{\ell_{I}}\ar[dr]^{\alpha}&& \Fun(I,\cC_{\infty})\\& \Fun(I,\cC)[W_{I}^{-1}]\ar[ur]^{\beta}&}\ ,
\end{equation}
where the arrow  $\beta$  is induced by the universal property of the  localization functor $\alpha$, see   \eqref{fwefoiu39r32r32r}. It is a crucial fact shown in  \cite[Prop. 7.9.2]{cisin}\footnote{Alternatively, if one in addition assumes that the model category structure on $\cC$ is combinatorial, then one could cite  \cite[Section 4.2.4]{htt}, or better,
\cite[Cor. 1.3.4.26]{HA} for this fact.} that
the functor $\beta$ is an equivalence.

For $A$ in $\cC$ and $B$ in $\Fun(I,\cC)$ we then have the following chain of natural equivalences of  spaces
\begin{eqnarray*}
\Map_{\cC_{\infty}}(\ell(A),\ell(\lim_{I}R(B)))&\stackrel{!}{\simeq}& \ell_{\sSet}(\Map_{\cC}(A,\lim_{I} R(B)))\\
&\simeq& \ell_{\sSet} (\Map_{\Fun(I,\cC)}(\underline{A},R(B)))\\&\stackrel{!!}{\simeq}&
\Map_{\Fun(I,\cC)[W_{I}^{-1}]}(\alpha(\underline{A}),\alpha(R(B)))\\&\stackrel{\beta,\simeq}{\to}&
\Map_{\Fun(I,\cC_{\infty}) }(\ell_{I}(\underline{A}),\ell_{I}(R(B)))\\&\stackrel{!!!,\simeq}{\leftarrow} &
\Map_{\Fun(I,\cC_{\infty}) }(\ell_{I}(\underline{A}),\ell_{I}(B))\\&\simeq& \Map_{\Fun(I,\cC_{\infty}) }(\underline{\ell(A)},\ell_{I}(B))\\
&\simeq & \Map_{ \cC_{\infty}}( \ell(A) ,\lim_{I}\ell_{I}(B))\ .
\end{eqnarray*}
For the equivalence marked by $!$ we use Assumption \ref{rgkoergergergreg}, that $\lim_{I}R(B)$ is fibrant,  and \eqref{f34kij3lk4grgr}.
For the equivalence marked by $!!$ we again use \eqref{f34kij3lk4grgr}, but now for the functor category  (note that all objects  of the functor category are cofibrant in the injective model category structure, and that the latter has   a simplicial extension),  and that by assumption 
  $R(B)$ is fibrant in the injective model category structure on   $\Fun(I,\cC)$.
Finally, for the   equivalence marked by $!!!$ we use that
$\ell_{I}(r):\ell_{I}\to \ell_{I}\circ R$ is an equivalence. 

The natural equivalence $$\Map_{\cC_{\infty}}(\ell(A),\ell(\lim_{I}R(B)))\simeq  \Map_{ \cC_{\infty} }( \ell(A) ,\lim_{I}\ell_{I}(B))$$
  implies the asserted equivalence of functors.  \end{proof}
  
 If we assume in addition that $(\cC,W)$ extends to a combinatorial model category, then Prop. \ref{ijwoiffewfewfewfewf} is an immediate consequence of \cite[1.3.4.23]{HA}.
  
Note that the assumption  \ref{ijwoiffewfewfewfewf}.\ref{rgkoergergergreg}  that the objects of the model category extension of $(\cC,W)$ are cofibrant comes in since we define $\cC_{\infty}$ as the localization of the whole category $\cC$ by the weak equivalences $W$. 
If not all objects are cofibrant, then the correct definition of the underlying $\infty$-category $\cC_{\infty}$ of the model category would be   $\cC^{c}[W^{-1}]$ \cite[1.3.4.15]{HA}.

  Below we will apply this proposition in the case $I=BG$.\hB
\end{remark}

Let $\cC$ be a member of the list $$ \{\scat_{1},\ClinCat_{1},\preCcat_{1},\Ccat_{1},\scat_{1}^{+},\ClinCat_{1^{+}},\preCcat_{1}^{+},\Ccat_{1}^{+}\}$$ and $\bA$ be an object of $\Fun(BG,\cC)$.
\begin{theorem}\label{fewoijowieffwefwef}
We have an equivalence
$$\lim_{BG}\ell_{BG}(\bA)\simeq \ell(\hat \bA^{G})\ .$$
\end{theorem}
\begin{proof}   We want to apply Proposition \ref{ijwoiffewfewfewfewf}. By Theorem \ref{fioweofefewfwf} the relative category $(\cC,W)$ extends to a simplicial model category.
By Theorem \ref{gti34og3g34g3g} the injective model category structure on $\Fun(BG,\cC)$ exists. Finally, by an inspection of the definitions (Definition \ref{fkjhifhiueiwhuiwhfiuewefewfe} for cofibrancy and, in addition, Proposition \ref{wfiojowefewfewfew} and Corollary \ref{wfeiweiofewfewfewf} for fibrancy),  all objects of $\cC$  are cofibrant and fibrant.

We now apply Proposition \ref{ijwoiffewfewfewfewf} for the explicit version of $R$ obtained in Proposition \ref{wfeioweo9ffwefwf} and Definition \ref{ioweoffewfewf}  in order to get  the equivalences
$$\lim_{BG}\ell_{BG}(\bA)\simeq 
\ell(\lim_{BG}(\cFun^{?}(\tilde G,\bA)))\simeq \ell(\hat \bA^{G})\ .$$ \end{proof}

\begin{remark}
 If   $\cC$ belongs to the list $$ \{\scat_{1},\ClinCat_{1},  \Ccat_{1},\scat_{1}^{+},\ClinCat_{1}^{+},  \Ccat_{1}^{+}\}\ ,$$
 then by Theorem \ref{fioweofefewfwf} in combination with Remark \ref{dgiowefwefewfewf} we could base the 
 proof of Theorem \ref{fewoijowieffwefwef}
 on the version of the proof of Proposition \ref{ijwoiffewfewfewfewf}  which only uses  \cite[Section 4.2.4]{htt}, see the footnote in the proof of   \ref{ijwoiffewfewfewfewf}.

 In the remaining two cases, where $\cC$ belongs to the list $$\{\preCcat_{1} ,\preCcat_{1}^{+}\}\ ,$$ we could then  deduce the assertion of Theorem  \ref{fewoijowieffwefwef} as follows.
 
 We use that   Theorem \ref{fewoijowieffwefwef} is true in the case of   $ \ClinCat_{1}$.
We let  $$\cF_{\pre}:\preCcat_{1}^{(+)}\to \ClinCat_{1}^{(+)}\ , \quad \cF_{\pre }:\preCcat^{(+)}\to \ClinCat^{(+)}$$ denote the  forgetful functors   on the level of $1$- and of $\infty$-categories. They are inclusions of   full subcategories
 fitting into adjunctions \eqref{fiioioefjeofiewfe22334wfwf}.
 We conclude that   the $\infty$-category
$\preCcat^{(+)}$ is complete and the limit of an  $I$-diagram
in $\preCcat^{(+)}$ can be calculated by the formula 
 \begin{equation}\label{foi3joi4j34gg34g4}
\lim_{I}\simeq \Bd^{\infty}\circ \lim_{I}\circ \cF_{\pre}\ ,
\end{equation} 
where the limit on the right-hand side is taken in  $ \ClinCat^{(+)}$.

For $\bA$ in $\Fun(BG,\cC)$ we then
  have the chain of equivalences \begin{eqnarray*}
 \lim_{BG} \ell_{BG}(\bA) &\stackrel{\eqref{foi3joi4j34gg34g4}}{\simeq}& \Bd^{\infty}(\lim_{BG}\cF_{\pre}(\ell_{BG}(\bA)))\\&\stackrel{!}{\simeq}&\Bd^{\infty}(\lim_{BG}\ell_{BG}(\cF_{\pre}(\bA)))\\&\stackrel{Thm. \ref{fewoijowieffwefwef}}{\simeq}&%
\Bd^{\infty}(\ell  (\widehat{\cF_{\pre}(\bA)}^{G}))\\
&\stackrel{!}{\simeq}& \ell( \Bd^{\infty} (\widehat{\cF_{\pre}(\bA)}^{G}))
\\&\stackrel{ \eqref{freh3oi4r34rf34fr3}}{\simeq}&\ell(  \hat \bA^{G})
 \ .\end{eqnarray*}
 At the marked morphisms we use that $\cF_{\pre}$ and $\Bd^{\infty}$ descend to the $\infty$-categories since they preserve weak equivalences. \hB
 \end{remark}

\section{Infinity-categorical $G$-orbits}

Let $\cC$ be a model category and $I$ be a small category. In the following definition we describe the weak equivalences, cofibrations, and fibrations of the projective model category structure on $\Fun(I,\cC)$ provided   it exists.
Recall, that for $i$ in $I$ we have the evaluation functor $e_{i}:\Fun(I,\cC)\to \cC$. 
\begin{ddd}\mbox{}
\begin{enumerate}
\item A weak equivalence    in $\Fun(I,\cC)$  is a morphism $f$ such that $e_{i}(f)$ is a weak equivalence  in $\cC$ for every $i$ in $I$.   
\item A fibration  in  $\Fun(I,\cC)$ is a morphism $f$    such that $e_{i}(f)$ is a fibration in $\cC$  for every $i$ in $I$.   
\item A cofibration is a morphism in  $\Fun(I,\cC)$  which has the left-lifting property with respect to 
trivial fibrations. 
\end{enumerate} 
\end{ddd}

It is known (see e.g. \cite[Thm. 11.6.1]{MR1944041}) that the projective model category structure on $\Fun(I,\cC)$ exists if the model category structure on $\cC$ is cofibrantly generated.

\begin{remark}
This remark is similar to Remark \ref{gerihoiejfiowejhfoiwefwe}.  
Let $(\cC,W)$ be relative category and $I$ be a small category.  
As before we let $ \ell:\cC\to \cC_{\infty}:=\cC[W^{-1}] $  be the localization and $ \ell_{I}:\Fun(I,\cC)\to \Fun(I,\cC_{\infty})$  be the induced functor. For an object $C$ in $\Fun(I,\cC)$ we want to calculate the colimit $\colim_{I}\ell_{I}(C)$ in $\cC_{\infty}$ using model categorical methods.  
 The following proposition is the analog of Proposition \ref{ijwoiffewfewfewfewf}. 
 Its  content is well-known, but we do not have a reference where it is stated in this ready-to-use form.

\begin{prop}\label{ijwoiffewfewfewfewf1111}
Assume that $(\cC,W)$  extends to a  combinatorial model category.  
 Then for any cofibrant replacement functor $l:L\to \id_{\Fun(I,\cC)} \ $ in the projective model category structure of $\Fun(I,\cC)$  we have an equivalence of functors 
$$\colim_{I}\circ \ell_{I}\simeq \ell\circ \colim_{I}\circ L:  \Fun(I,\cC)\to  \cC_{\infty}\ .$$
\end{prop}\begin{proof}
We shall sketch a proof which   is completely analogous to the proof of   Proposition \ref{ijwoiffewfewfewfewf}. Since a combinatorial model category structure is in particular cofibrantly generated the projective model category structure
on  $\Fun(I,\cC)$ exists. It is again combinatorial \cite[Prop. A.2.8.2]{htt}.

 We again have the commuting diagram \eqref{hhjopkrpgi03klglkrg4985ug945geer} 
   where the arrow  $\beta$  is an equivalence.
For a fibrant object $B$ in $\cC$ and $C$ in $\Fun(I,\cC)$ we then have the following chain of natural equivalences of  spaces
\begin{eqnarray*}
\Map_{\cC_{\infty}}(\ell(\colim_{I}L(C)),\ell(B))&\stackrel{!}{\simeq}& \ell_{\sSet}(\Map_{\cC}(\colim_{I} L(C),  B))\\
&\simeq& \ell_{\sSet} (\Map_{\Fun(I,\cC)}(L(C),\underline{B}))\\&\stackrel{!!}{\simeq}&
\Map_{\Fun(I,\cC)[W_{I}^{-1}]}(\alpha(L(C)),\alpha(\underline{B}))\\&\stackrel{\beta,\simeq}{\to}&
\Map_{\Fun(I,\cC_{\infty}) }(\ell_{I}( L(C)),\ell_{I}(\underline{B}))\\&\stackrel{!!!,\simeq}{\leftarrow} &
\Map_{\Fun(I,\cC_{\infty}) }(\ell_{I}( C),\ell_{I}(\underline{B}))\\&\simeq& \Map_{\Fun(I,\cC_{\infty}) }(\ell_{I}( C),\underline{\ell (B)})\\
&\simeq & \Map_{ \cC_{\infty}}( \colim_{I} \ell_{I}(C) , \ell (B))\ .
\end{eqnarray*}
with the same justifications  of the equivalences as in the proof of  Proposition \ref{ijwoiffewfewfewfewf}. 
For the equivalences marked by $!$ and $!!$ we use that the model category structures on $\cC$ and  the functor category are combinatorial so that we still can apply Remark \ref{roigegergerg} in order to justify the equivalence \ref{f34kij3lk4grgr} (note that in the projective model category structure on $\Fun(I,\cC)$ we  can not expect that all objects are cofibrant), where we  now have to use the existence of functorial factorizations and \cite[Rem. 1.3.4.16]{HA}.  Note that $\colim_{I}L(C)$ is cofibrant in $\cC$. 

 The natural equivalence $$\Map_{\cC_{\infty}}(\ell(\colim_{I}L(C)),\ell(B))\simeq  \Map_{ \cC_{\infty}}( \colim_{I} \ell_{I}(C) , \ell (B))$$ for all fibrant $B$
  implies the asserted equivalence of functors.
  \end{proof}
  Alternatively, the assertion of Prop. \ref{ijwoiffewfewfewfewf1111} is an immediate consequence of \cite[1.3.4.24]{HA}.  
\hB\end{remark}

  We now assume that $\cC$ belongs to the list   
 $$\{\scat_{1},\ClinCat_{1}, \Ccat_{1},\scat_{1}^{+},\ClinCat_{1}^{+}, \Ccat_{1}^{+}\}\ .$$
 \begin{remark}
 We must exclude the pre-$C^{*}$-category cases  since we do not know that the corresponding model categories are combinatorial. \hB
 \end{remark}
 The relative category
$(\cC,W)$ extends to a combinatorial  model category  (Theorem \ref{fioweofefewfwf} and Remark \ref{dgiowefwefewfewf}) in which all objects are cofibrant and fibrant.
 Consequently the projective model category structure on $\Fun(BG,\cC)$ exists and Proposition \ref{ijwoiffewfewfewfewf1111} applies to $(\cC,W)$.

 Recall the Definition \ref{fewoihfiweiofjwefoiewfewfewf} of the arrow category $\tilde G$ of $G$. Furthermore recall the Convention \ref{fwerifhiwefewfewf} concerning the usage of $\sharp$.
We consider the functor 
$$L:=-\sharp \tilde G:\Fun(BG,\cC) \to \Fun(BG,\cC)$$
together with the transformation $L\to \id_{\Fun(BG,\cC)}$ induced by  the morphism of $G$-groupoids $\tilde G\to *$. 
\begin{lem}\label{rgergr43tgergergerge}
The functor $L$ together with the transformation $L\to \id_{\Fun(BG,\cC)}$ 
is a cofibrant replacement functor for the projective model category structure on $\Fun(BG,\cC)$.  \end{lem} 
\begin{proof} Since $\tilde G\to \id$ is an (non-equivariant) equivalence of groupoids and  for every  object $\bD$ in  $ \Fun( BG,\cC)$  the functor $\bD\sharp -$ from groupoids to $\cC$ preserves (marked) unitary equivalences (see the proof of Lemma \ref{reiofweiofweewf}), the morphism $\bD\sharp \tilde G\to \bD$ is a weak equivalence.
We must show that
$L(\bD)$ is cofibrant. To this end we consider the lifting problem 
$$\xymatrix{\emptyset\ar[r]\ar[d]&\bA\ar[d]^{f} \\ \bD\sharp \tilde \bG\ar[r]^{u}\ar@{-->}[ur]^{c}&\bB}$$ 
where $f$ is a trivial fibration in $\cC$. Since $f$ is surjective on objects
we can find an inverse equivalence (possibly non-equivariant) $g:\bB\to \bA$ for $f$ such that $f\circ g=\id_{\bB}$. The map
$\bD\sharp \{1\}\stackrel{u_{|\bD\sharp \{1\}}}{\to} \bB\stackrel{g}{\to }\bA$ can uniquely be extended to an equivariant morphism $c$ which is the desired lift. \end{proof}

  If  $\bC$ is an object of $\cC$, then by $\underline{\bC}$ we denote the object of $\Fun(BG,\cC)$ given by $\bC$ with the trivial action of $G$.

We assume that $\cC$ belongs to the list  $$\{\scat_{1},\ClinCat_{1}, \Ccat_{1},\scat_{1}^{+},\ClinCat_{1}^{+}, \Ccat_{1}^{+}\}\ .$$  
Let $\bC$ be  an object of $\cC$. Note that $BG$ is a groupoid.
\begin{theorem}\label{weoijoijvu9bewewfewfwef} 
We have an equivalence
$$ \colim_{BG}\ell_{BG}(\underline{\bC})\simeq \ell (\bC\sharp BG)\ .$$ \end{theorem}
\begin{proof} 

By Proposition \ref{ijwoiffewfewfewfewf1111} and Lemma \ref{rgergr43tgergergerge} we have an equivalence
\begin{equation}\label{vtr4iuhgiu4hgiu4hgiuhgiuefvev}
\colim_{BG}\ell_{BG}(\underline{\bC})\simeq \ell(\colim_{BG} \underline{\bC}\sharp \tilde G)\ .
\end{equation} 
So it remains to calculate the colimit $\colim_{BG} \underline{\bC}\sharp \tilde G$ in $\cC$. To this end will show that the functor $ \bC\sharp-:\Groupoids_{1}\to \cC$ commutes with colimits and calculate that
$\colim_{BG}\tilde G\cong BG$.

 Let $\bA$ be an object of  $\cC$. For  a second object $\bD$ in $\cC$ we let $\Fun_{ \cC}(\bA,\bD)_{+}$
  denote the    subgroupoid of the functor category  $\Fun_{\cC}(\bA,\bD)$  of  all functors and (marked) unitary isomorphisms.
\begin{lem} We have an adjunction
$$\bA\sharp-: \Groupoids_{1}\leftrightarrows \cC: \Fun_{\cC}(\bA,-)_{+}\ .$$
\end{lem}
\begin{proof}
For $\bD$ in $\cC$ and $\bG$ in $ \Groupoids_{1}$ we construct a natural bijection
$$\Hom_{ \cC }(\bA\sharp \bG,\bD)\cong \Hom_{  \Groupoids_{1}}(\bG, \Fun_{ \cC}(\bA,\bD)_{+})\ .$$
This bjection sends a morphism $\Phi$ in $\Hom_{ \cC }(\bA\sharp \bG,\bD)$ to $\Psi$ in  $ \Hom_{  \Groupoids_{1}}(\bG, \Fun_{ \cC}(\bA,\bD)_{+}) $.  
Let $\Phi$ be given. We let $g,h$ denote objects of $\bG$ and $\phi:g\to h$ a be morphism. Then we define $\Psi$ by
$$\Psi(g)(a):=\Phi(a,g)\ , \quad  \Psi(g)(f):=\Phi(f,\id_{g})\ , \quad \Psi(\phi):=(\Phi(\id_{a},\phi))_{a\in \bA}\ .$$
Here $a$ is an object of $\bA$ and $f$ is a morphism in $\bA$. 
Observe that $\Psi(\phi)$ is a unitary isomorphism since $\Phi$ is compatible with the involution and $\phi$ is a unitary isomorphism. In the marked case, if $f$ is marked, then $\Psi(g)(f)$ is marked since $(f,\id_{g})$ is marked in $\bA\sharp \bG$ and $\Phi$ preserves marked morphisms. Furthermore, $\Psi(\phi)$ is implemented by marked isomorphisms.

Vice versa, let $\Psi$ be given.
Then we define
$$\Phi(a,g):=\Psi(g)(a)\ , \quad \Phi(f,\phi):=\Psi(\phi)_{a} \circ \Psi(g)(f)\ .$$
This formula determines $\Phi$ on the generators of the morphisms. It can be extended by linearity (in the $\C$-linear cases) and continuity (in the $C^{*}$-cases).  \end{proof}

  \begin{kor}
We have an adjunction
$$ \bC \sharp-:\Fun(BG,\Groupoids_{1})\leftrightarrows \Fun(BG,\cC):  \Fun_{\Fun(BG,\cC)}(\underline{\bC},-)_{+}\ .$$
 \end{kor}
 
Since $   \bC \sharp-$ is a left-adjoint functor it commutes with colimits. Consequently
we have an isomorphism \begin{equation}\label{diwejdoiefewfewf}
\colim_{BG} (\underline{\bC}\sharp \tilde G)\cong   \bC \sharp ( \colim_{BG}\tilde G)\ .
\end{equation}

\begin{lem}\label{oijewiofwfewfewf}
We have an isomorphism $\colim_{BG}\tilde G\cong BG$.
\end{lem}
\begin{proof}
We check the universal property of the colimit.
Let $\bK$ be any groupoid.
Then we have a natural bijection
$$\Hom_{\Fun(BG,\Groupoids_{1})}(\tilde G,\underline{\bK})\cong \Hom_{\Groupoids_{1}}(BG,\bK)\ .$$
This bijection sends
$\Phi$ in $\Hom_{\Fun(BG,\Groupoids_{1})}(\tilde G,\underline{\bK})$ to the morphism  $\Psi$ in  $\Hom_{\Groupoids_{1}}(BG,\bK)$ given by
$$\Psi(*):=\Phi(1)\ , \quad \Psi(g):=\Phi(1\to g)\ .$$
If $\Psi$ is given, then we define $\Phi$ by
$$\Phi(g):=\Psi(*)\ , \quad \Phi(g\to h):=\Psi(g^{-1}h)\ .$$ \end{proof}

The assertion of  Theorem  \ref{weoijoijvu9bewewfewfwef} now follows from a combination of the relations \eqref{vtr4iuhgiu4hgiu4hgiuhgiuefvev}, \eqref{diwejdoiefewfewf},  and  Lemma \ref{oijewiofwfewfewf}. \end{proof}

 In the following we discuss an application  of Theorem \ref{weoijoijvu9bewewfewfwef} to the  calculation of the values of an induction functor $J^{G}$ (see Definition \ref{fgwo9gfwegwwfwefewfwf}) from $\cC$ to functors from the orbit category $\Orb(G)$ of $G$ to $\cC_{\infty}$.

 The objects of  $\Orb(G)$
  are the transitive $G$-sets, and its morphisms are equivariant maps.
We can consider the underlying set of $G$ as a transitive $G$-set with respect to the right action. One can then identify $\End_{\Orb(G)}(G)$ with the group $G$ acting by left-multiplication. We therefore get a fully faithful functor
$$j:BG\to 
 \Orb(G)$$ which sends the unique object of $BG$ to the transitive $G$-set $G$, and which identifies the group
$\End_{BG}(pt)$ (given by $G$) with the group $ \End_{\Orb(G)}(G)$  as described above.

If $\cC_{\infty}$ is a presentable $\infty$-category (or sufficiently cocomplete), then we get an adjunction \begin{equation}\label{f34iofj3oifjf9uf984ff34f}
j_{!}:\Fun(BG,\cC_{\infty})\leftrightarrows \Fun(\Orb(G),\cC_{\infty}):j^{*}\ ,
\end{equation}
 where $j^{*}$ is the restriction functor along $j$.

We now assume that $\cC$ belongs to the list  $$\{\scat_{1},\ClinCat_{1}, \Ccat_{1},\scat_{1}^{+},\ClinCat_{1}^{+},\Ccat_{1}^{+} \}\ .$$ Then the corresponding $\infty$-category $\cC_{\infty}$ is presentable by Corollary \ref{wfeoifoweifjwfewfewfewfewf}   so that \eqref{f34iofj3oifjf9uf984ff34f} applies.

\begin{ddd} \label{fgwo9gfwegwwfwefewfwf} We define the induction functor $J^{G}$ as the composition \begin{equation}\label{f3ouoeriu3orou93434f}
J^{G} : \cC\stackrel{\underline{(-)}}{\to}\Fun(BG,\cC)\stackrel{\ell_{BG}}{\to} \Fun(BG,\cC_{\infty})\stackrel{j_{!}}{\to} \Fun(\Orb(G),\cC_{\infty}) \ .
\end{equation}
 \end{ddd}

For a subgroup $H$ of $G$ we consider $H\backslash G$ with the action of $G$ by right multiplication as an object of   $\Orb(G)$.
 
 \begin{prop}\label{rguihiufhwrufwfwefwefewf}
 We have an equivalence
 $$J^{G}(\bC)(H\backslash G)\simeq \ell( \bC\sharp BH)\ .$$
 \end{prop}
 \begin{proof}
 The functor  $j_{!}$ is a left Kan-extension functor. The point-wise formula for  the left Kan-extension gives an equivalence
$$J^{G}(\bC)(H\backslash G)\simeq j_{!}(\ell_{BG}(\underline{\bC}))(H\backslash G)\simeq \colim_{BG/(H\backslash G)}  \ell_{BG}(\underline{\bC})\ .$$
The functor $BH\to BG/(H\backslash G)$ which sends the object $pt$ to the  projection $G\to H\backslash G$ and the element
$h$ of $ H=\End_{BH}(pt)$ to the morphism in $BG/(H\backslash G)$ given by left-multiplication by $h$ is an equivalence of categories.
Consequently, we get an equivalence \begin{equation}\label{d2ieuz2iudzh2id3d3d2}\colim_{BG/(H\backslash G)}  \ell_{BG}(\underline{\bC})\simeq \colim_{BH}\ell_{BH}(\underline{\bC})\stackrel{Thm. \ref{weoijoijvu9bewewfewfwef}}{\simeq}
\ell(\bC\sharp BH)\ .
\end{equation}

 \end{proof}

In the following  examples we apply Proposition \ref{rguihiufhwrufwfwefwefewf} to the construction of equivariant $K$-theory functors.
Let $\bS$ be a stable $\infty$-category, e.g., the category of spectra. A Bredon-type $G$-equivariant $\bS$-valued homology theory is   determined by a functor
$$\Orb(G)\to \bS\ .$$ (see e.g. \cite{davis_lueck}).  
If $K:\cC_{\infty}\to \bS$ is some functor and we fix an object $\bC$ in $\cC$, then we can define  such a functor  by precomposing with the induction functor. We set
$$K_{\bC}^{G}:=K\circ J^{G}(\bC):\Orb(G)\to \bS\ .$$ 
By   Proposition \ref{rguihiufhwrufwfwefwefewf} the values of this functor are given by \begin{equation}\label{veruh3iuhniuefvdfvevfv}
K_{\bC}^{G}(H\backslash G)\simeq K(\ell(\bC\sharp BH))\ .
\end{equation}

\begin{example}\label{ergoieugoerregregeg}
We let $\cC=\ClinCat_{1}$ and $K:\ClinCat_{1}\to \Sp$ be the algebraic $K$-theory functor. The latter is defined as the composition \begin{equation}\label{f23lkmklfwfwefwefewfewff}
K:\ClinCat\xrightarrow{\cF} \mathbf{preAdd}\xrightarrow{(-)_{\oplus}}\Add \xrightarrow{K^{alg}} \Sp\ .
\end{equation}
$$$$
Here $ \mathbf{preAdd}$ and $\Add$ are the $\infty$-categories of preadditive and additive categories obtained from the corresponding $1$-categories by inverting the exact equivalences. The 
 forgetful functor $\cF$ takes the underlying preadditive category of a $\C$-linear $*$-category, $(-)_{\oplus}$ is the additive completion functor (the left-adjoint to the inclusion $\Add\to \mathbf{preAdd})$, and $K^{alg}$ is the $K$-theory functor for additive categories. We refer to \cite{addcats}  for further details. 

We now fix the object classifyer object $\Delta^{0}$ of $\ClinCat$ (given by the $\C$-linear $*$-category associated to the $*$-algebra $\C$).
Then we get the  functor
$$K^{G}_{\Delta^{0}}:\Orb(G)\to \Sp\ .$$
Let $K^{alg}_{ring}:\mathbf{Ring}\to \Sp$ be the algebraic $K$-theory functor for rings
given in terms of $K^{alg}$ as the composition
 \begin{equation}\label{vekj3b5gkjnberkjvefvefvev}
K^{alg}_{ring}:\mathbf{Ring}\to  \mathbf{preAdd} \xrightarrow{(-)_{\oplus}}\Add \xrightarrow{K^{alg}} \Sp\ ,
\end{equation}  
where the first functor interprets a ring as a pre-additive category with one object.
Then we see that $K^{G}_{\Delta^{0}}$
has the values
$$K^{G}_{\Delta^{0}}(H\backslash G)\simeq K^{alg}_{ring}(\C[H])\ .$$
Indeed, 
$$K^{G}_{\Delta^{0}}(H\backslash G)\stackrel{\eqref{veruh3iuhniuefvdfvevfv}}{\simeq} K (\ell(\Delta^{0}\sharp BH)) \stackrel{Ex. \ref{fwoijfofewfewfewfewfew} }{\simeq }
K (\ell(\C[H]))\stackrel{\eqref{f23lkmklfwfwefwefewfewff},\eqref{vekj3b5gkjnberkjvefvefvev}}{\simeq}K^{alg}_{ring}(\C[H])\ .$$
We see that   $K^{G}_{\Delta^{0}}$  provides a categorical construction of a functor which can be compared with the  
 usual equivariant algebraic   $K$-theory functor  as considered e.g. in \cite[Sec. 2]{davis_lueck}\footnote{Note that we have only discussed the values, not the action of the functor on morphisms.}. \hB\end{example}
 
\begin{example}
We let $\cC=\Ccat_{1}$ and $K^{top}_{1}:\Ccat_{1}\to \Sp$ be the  topological $K$-theory  functor  for $C^{*}$-categories.  We refer \cite[Sec. 7.5]{buen} for a construction  of such a functor  as a composition 
$$K^{top}_{1}:\Ccat_{1}\xrightarrow{A^{f}} C^{*}\Alg\xrightarrow{K^{top}_{C^{*}\Alg}}\Sp\ ,$$
where $A^{f}$ is the functor which associates to a $C^{*}$-category the free $C^{*}$-algebra \cite{joachimcat} generated by it, and 
$K^{top}_{C^{*}\Alg}$ is the usual topological $K$-theory functor for $C^{*}$-algebras.
The subscript $1$ indicates that the functor is defined on the $1$-category of $C^{*}$-categories.
 In particular, 
by \cite[Cor. 7.44]{buen}  the functor $K^{top}_{1}$ sends unitary equivalences of $C^{*}$-categories to equivalences of spectra and therefore has an essentially unique factorization $K^{top}$ as in 
$$\xymatrix{\Ccat_{1}\ar[rr]^{K_{1}^{top}}\ar[dr]^{\ell}&&\Sp\\&\Ccat\ar[ur]^{K^{top}}&}\ .$$
We again fix the object classifier object $\Delta^{0}$ in $\Ccat_{1}$ and consider the functor
$$K^{top,G}_{\Delta^{0}}:\Orb(G)\to \Sp\ .$$
We then see that $K^{top,G}_{\Delta^{0}}$ has the values
$$K^{top,G}_{\Delta^{0}}(H\backslash G)\simeq K^{top}_{C^{*}\Alg}(C)(C^{*}_{\max}(H))\ .$$
 Indeed,
 $$K^{top,G}_{\Delta^{0}}(H\backslash G)\stackrel{\eqref{veruh3iuhniuefvdfvevfv}}{\simeq}
 K^{top}(\ell(\Delta^{0}\sharp BH))\stackrel{Ex. 
 \ref{fwoijfofewfewfewfewfew}}{\simeq}  K^{top}(\ell(C^{*}_{\max}(H) ))\simeq K^{top}_{C^{*}\Alg}(C^{*}_{\max}(H))\ ,$$
 where the last equivalence follows from an inspection of the definition of the $K$-theory functor $K^{top}$.
 We again see that   $K^{top,G}_{\Delta^{0}}$  provides a categorical construction of a functor which can be compared with the  
  topological $K$-theory functors as considered in \cite[Sec. 2]{davis_lueck}. But note that our functor involves the maximal group $C^{*}$-algebra, while the functor constructed in 
 \cite{davis_lueck}  involves the reduced group $C^{*}$-algebra.\hB
\end{example}

\bibliographystyle{alpha}

\end{document}